\documentclass[3p,times,number]{elsarticle}

\usepackage{amsmath,amssymb,amsthm,mathtools}
\usepackage{booktabs}
\usepackage{float}
\usepackage{caption}
\usepackage{algorithm}
\usepackage{algpseudocode}
\usepackage{xcolor}
\usepackage{hyperref}
\usepackage{bookmark}
\hypersetup{hidelinks,hypertexnames=false}
\biboptions{sort&compress}
\newcommand{\doi}[1]{\url{https://doi.org/#1}}
\graphicspath{{figures/}{./}}

\newcommand{\phit}{\widetilde{\varphi}}
\newcommand{\R}{\mathbb{R}}
\newcommand{\pade}[2]{\ensuremath{[{#1}/{#2}]}}

\newtheorem{theorem}{Theorem}[section]
\newtheorem{lemma}[theorem]{Lemma}
\newtheorem{proposition}[theorem]{Proposition}
\newtheorem{corollary}[theorem]{Corollary}
\newtheorem{assumption}[theorem]{Assumption}
\newtheorem{remark}[theorem]{Remark}
\newtheorem{note}[theorem]{Note}
\newtheorem{definition}[theorem]{Definition}

\journal{Computer Methods in Applied Mechanics and Engineering}

\begin{document}

\begin{frontmatter}

\title{Modal Integrators for Isogeometric Stabilization of Outliers}

\author[oden,ase]{Sreeram Shankar\corref{cor1}}
\ead{sreeramshankar08@utexas.edu}
\cortext[cor1]{Corresponding author.}
\affiliation[oden]{organization={Oden Institute for Computational Engineering and Sciences,
The University of Texas at Austin},
            city={Austin},
            state={TX},
            postcode={78712},
            country={USA}}
\affiliation[ase]{organization={Department of Aerospace Engineering and Engineering Mechanics,
The University of Texas at Austin},
            city={Austin},
            state={TX},
            postcode={78712},
            country={USA}}

\begin{abstract}
High-order isogeometric discretizations produce a small number of non-physical outlier modes at the top of the discrete spectrum.
These modes can pollute the numerical solution in wave propagation and structural dynamics.
Most existing remedies modify the spatial discretization through outlier-free bases, mass lumping, subspace projection, and related constructions.
We introduce MIISO (Modal Integrators for Isogeometric Stabilization of Outliers), a family of exponential Rosenbrock--Krylov methods of orders two through six that suppress the outliers in the time integrator and leave the spatial discretization unchanged.
We construct a non-holomorphic amplification map by blending a Pad\'e approximation that preserves physical oscillatory modes with a Pad\'e approximation that damps high frequencies, with the blend placed at the known spectral location of the outliers.
For orders three and higher the damping approximant is repaired so that the method is unconditionally $A$-stable.
Physical modes receive no numerical dissipation and the outliers are driven to zero.
The outlier handling and the method steps are fully matrix-free, and the method extends to nonlinear problems.
Numerical experiments recover the predicted orders, match spatial outlier removal without changing the discretization, show lower cost than generalized-$\alpha$ at matched accuracy and in scaling,
and demonstrate the method on nonlinear and three-dimensional problems.
\end{abstract}


\end{frontmatter}

\section{Introduction}
\label{sec:intro}
Isogeometric analysis (IGA) has emerged as a popular technique for the numerical solution of differential equations~\citep{hughes2005,cottrell2006,cottrellhughesbazilevs2009}. A notable property of IGA is its eigenvalue spectrum, as smooth spline bases reproduce the exact eigenvalue behavior more accurately than $C^{0}$ finite elements~\citep{cottrell2006}.
However, a pathological feature of these smooth bases is that they produce a small number of eigenvalues
at the top of the discrete spectrum that are non-physical and inaccurate,
known as \emph{outlier modes}.
They arise from reduced continuity at patch boundaries and from boundary
conditions, and their number can be characterized a~priori from the
polynomial degree and boundary conditions~\citep{hiemstra2021}.
Removal of these outliers is of particular interest for stable and accurate simulation, since they can pollute aspects of the numerical solution in applications such as structural dynamics and wave propagation
and limit the step size of explicit time integration methods.

\subsection{Spatial treatment of the outliers}
\label{sec:intro-spatial}

The established remedies modify the spatial semidiscretization before time
integration of a dynamics problem.
Boundary penalization adds terms to the discrete operators that suppress the
boundary-layer behavior responsible for the worst
outliers~\citep{deng2021penalty}.
Optimally blended quadrature replaces the standard Gauss rule used to
assemble the mass and stiffness matrices with a mixture chosen to cancel
spectral error~\citep{calo2017blending}.
Variational perturbation methods treat outlier eigenvalues as a perturbation
of the continuous spectrum and correct the discrete operator
accordingly~\citep{hiemstra2023removal}.
Mass-lumping with deflation modifies the mass matrix and removes the outlier
eigenspace once it is identified~\citep{voet2024lumping}.
Outlier-free spline spaces have also been constructed so that no outlier
eigenvalues appear in the discrete
spectrum~\citep{manni2022,lamsahel2025}.
These approaches have been applied previously to nonlinear dynamic problems~\citep{nguyen2024rods}.

These approaches are effective, but each requires additional work at the
level of the spatial semidiscretization.
Standard IGA codes that do not already incorporate one of them must be
modified before they can suppress outliers.
A temporal remedy would leave existing codes and spatial discretizations
unchanged for dynamics problems.
Outlier removal would then be a concern of the integrator, which is often
simpler for a user and applies more broadly.

\subsection{Global temporal damping}
\label{sec:intro-global}

A separate class of methods controls high-frequency numerical dissipation through
time integration.
The generalized-$\alpha$ method~\citep{chunghulbert1993} and its predecessors (Newmark~\citep{newmark1959}, HHT-$\alpha$~\citep{hilberhughestaylor1977}, WBZ-$\alpha$~\citep{wood1980}) introduce a parameter, the spectral radius~$\rho_\infty$, that controls
how strongly high-frequency modes are damped each step and must be
tuned to the problem.
Related Pad\'e-based schemes also use a single damping parameter for the same
purpose~\citep{song2023}.
When that parameter induces any dissipation, every oscillatory mode receives
some damping at every step size
(Proposition~\ref{prop:global-dissipation}). We therefore
refer to these methods as global damping methods.
On an isogeometric discretization, this global damping approach can often overdamp physical
eigenvalues and thereby reduce the high spectral accuracy of the spline
basis.
A temporal remedy must therefore supply \emph{selective} damping, with no added
dissipation on physical modes but strong damping on the outliers.

\subsection{Exponential Rosenbrock--Krylov methods}
\label{sec:intro-exp}

Exponential Rosenbrock methods are matrix-function integrators for the solution of stiff
nonlinear ordinary differential equations~\citep{hochbruckostermann2009,hochbruckostermann2010}.
At each step they re-linearize about the local Jacobian and advance with
functions of that Jacobian related to the exponential.
When those functions are evaluated by Krylov projection, only
Jacobian--vector products are required and no dense matrix function is
formed~\citep{hochbrucklubich1997,saad1992}, so the scheme is matrix-free.
On large systems, this is typically cheaper per step than traditional implicit methods
that form and refactor the Jacobian, while the matrix-function treatment
avoids the explicit eigenvalue stability conditions, such as the CFL condition, of explicit schemes.
Applied previously to elastodynamics~\citep{michels2017}, these methods introduce little artificial damping across the spectrum.

\subsection{Contribution and relation to prior work}
\label{sec:intro-contrib}

The preceding sections point to a temporal treatment of outliers that
leaves the spatial discretization unmodified.
Global damping methods damp every oscillatory mode, so they cannot
selectively damp the outliers alone.
Exponential Rosenbrock--Krylov methods already handle nonlinearity and
can be constructed matrix-free. Modifying their amplification factor
is a natural route to introduce selective damping.

We introduce MIISO (Modal Integrators for Isogeometric Stabilization of
Outliers), a family of orders two through six obtained by modifying the
amplification factor of exponential Rosenbrock methods.
We perform a non-holomorphic blend of conservative and dissipative Pad\'e approximants placed at
a cutoff fixed by the a~priori outlier count, making the method
free of a tuning parameter.
Outlier modes are driven out of the solution while resolved physical modes
receive no artificial dissipation, thus providing a selective damping mechanism
for stabilization of outliers.
We construct the method to be fully matrix-free to ensure efficiency
and reduced computational cost.

The paper proceeds as follows.
Section~\ref{sec:prelim} recalls the a~priori outlier count and the
exponential Rosenbrock--Krylov setting that MIISO modifies.
Section~\ref{sec:pade-chapter} develops the conservative and dissipative
Pad\'e approximants used in the blend, including a repair of the
higher-order dissipative members.
Section~\ref{sec:construction} builds the blended amplification map,
embeds it in the exponential Rosenbrock stages for orders two through
six, and gives the matrix-free algorithms that locate the
resolved--outlier boundary, update it in nonlinear runs, and evaluate
the stages of the methods.
Section~\ref{sec:discussion} discusses resolved and outlier mode
behavior, stability and convergence, when the outlier handling mechanism remains valid, and practical cost.
Section~\ref{sec:results} confirms these properties numerically and
demonstrates the method on nonlinear and three-dimensional problems.
\section{Preliminaries}
\label{sec:prelim}

\subsection{Isogeometric outlier modes}
\label{sec:outliers}

For a one-dimensional bar or beam discretized with $C^{p-1}$ splines of
degree~$p$, \citet{hiemstra2021} show that the number of outlier
eigenpairs,~$k$, depends on~$p$ and the boundary conditions.
Writing $k_{\mathrm{bar}}(p,\mathrm{bc})$ for the contribution of one end of a
bar,
\begin{equation}
\begin{aligned}
k_{\mathrm{bar}}(p,\mathrm{fixed})
&= \Bigl\lfloor\frac{p-1}{2}\Bigr\rfloor,
&
k_{\mathrm{bar}}(p,\mathrm{free})
&= \Bigl\lfloor\frac{p}{2}\Bigr\rfloor,
\end{aligned}
\end{equation}
with analogous one-end counts for beams with fixed, pinned, and free ends.
For a boundary pair $(\ell,r)$ on a bar the total is
$k=k_{\mathrm{bar}}(p,\ell)+k_{\mathrm{bar}}(p,r)$, and likewise for beams
using the corresponding counts.
In two or three dimensions, where the multivariate basis is the tensor product of one-dimensional spline bases, a mode is resolved only if every one-dimensional factor is itself resolved. Writing $n_1,n_2,n_3$ for the number of unconstrained one-dimensional basis functions in each direction, that is, after the boundary conditions are applied, and $k_1,k_2,k_3$ for the corresponding one-dimensional outlier counts, the resolved count in two or three dimensions is the product $(n_1-k_1)(n_2-k_2)(n_3-k_3)$, and $k$ is the remainder after subtracting this product from the total number of unconstrained basis functions~\citep{hiemstra2021}.

After spatial discretization, the dynamics can be written in first-order form with
system size~$N$ and Jacobian~$J$.
With time step~$h>0$, the scaled operator is~$hJ$.
Ordering the eigenvalues of~$hJ$ by decreasing magnitude,
$|\lambda_1(hJ)|\ge\cdots\ge|\lambda_N(hJ)|$, the \emph{outlier count}~$k$ identifies the \emph{outlier band} as the first~$k$ of these eigenvalues.
A convenient scalar location of that band is the cutoff
\begin{equation}
\omega_c = \bigl|\lambda_k(hJ)\bigr|,
\label{eq:omegac}
\end{equation}
the modulus of the $k$-th largest-magnitude eigenvalue of~$hJ$.
The corresponding eigenpairs are the \emph{outlier modes}.
The remaining $N-k$ eigenpairs with $|\lambda_i(hJ)|<\omega_c$ are the
\emph{resolved modes}, the eigenvalues of the discrete spectrum that the spline
discretization approximates accurately.
For a fixed--fixed bar of degree~$p=4$, for example, $k=2$.

\subsection{Exponential Rosenbrock methods}
\label{sec:exprb}

For the initial-value problem
\begin{equation}
y'(t)=f(t,y(t)), \qquad y(t_0)=y_0,
\end{equation}
with $y(t)\in\R^N$, write
$J_n:=f_y(t_n,y_n)$ for the Jacobian of~$f$ with respect to~$y$ at step~$n$,
and $f_t(t_n,y_n)$ for the explicit time derivative of~$f$ when~$f$ is nonautonomous.
Exponential Rosenbrock methods re-linearize at each step about
$J:=J_n$ and advance with matrix functions of~$hJ$~\citep{hochbruckostermann2009}.

The exponential Rosenbrock--Euler step,
\begin{equation}
y_{n+1}
= y_n + h\,\varphi_1(hJ)\,f(t_n,y_n) + \tfrac12 h^2 f_t(t_n,y_n),
\qquad
\varphi_1(z)=\frac{e^z-1}{z},
\end{equation}
is second-order for nonautonomous~$f$.
Higher-order members introduce further $\varphi$-functions through the
recurrence
\begin{equation}
\varphi_{j+1}(z)=\frac{\varphi_j(z)-1/j!}{z},
\qquad
\varphi_0(z)=e^z,
\end{equation}
together with internal stages.
For nonlinear~$f$ those stages produce defect corrections that vanish
when~$f$ is linear, so the step reduces to the linear amplification factor.
The standard third- and fourth-order schemes of this type are exprb32 and
exprb43~\citep{hochbruckostermann2009}.
Parallel-stage variants for orders five and six are pexprb54s5 and
pexprb65s7~\citep{luanostermann2016}.

In large systems, the quantity $\varphi_j(hJ)v$ on a vector
$v\in\R^N$ can be evaluated by Krylov projection.
An $m$-step Arnoldi iteration~\citep{arnoldi1951} on~$hJ$ started from~$v$ yields an
orthonormal
basis $V_m\in\R^{N\times m}$ and an $m\times m$ upper Hessenberg matrix
$H_m$ ($m\ll N$) such that
\begin{equation}
\varphi_j(hJ)\,v
\;\approx\;
\|v\|\,V_m\,\varphi_j(H_m)\,e_1,
\label{eq:ritz}
\end{equation}
where $e_1=(1,0,\ldots,0)^\top\in\R^m$ is the first coordinate
vector~\citep{hochbrucklubich1997,saad1992}.
Only matrix-vector products with~$J$ are required.

\subsection{Global damping methods}
\label{sec:global-damp}

The generalized-$\alpha$ method modifies the amplification matrix of the
second-order structural recurrence to induce damping of high-frequency modes.
The mixed-order Pad\'e scheme of~\citet{song2023} applies the same idea
through a blend of different Pad\'e approximants to control dissipation.
In both cases a single parameter, often denoted $\rho_\infty$, sets the
damping level and is left to the user to tune.

For any fixed $\rho_\infty\in[0,1)$, both the generalized-$\alpha$
amplification matrix and the mixed-Pad\'e map of~\citet{song2023} have
spectral radius strictly below one at every nonzero frequency and every
step size (Proposition~\ref{prop:global-dissipation},
proof in Appendix~\ref{app:global-dissipation}).
Every oscillatory mode therefore receives some amount of dissipation at
every step size. These methods therefore cannot selectively annihilate the outlier band. Any dissipation applied to the outliers also damps physical modes to some degree, reducing the spectral accuracy that the spline basis provides.

The next section introduces the Pad\'e approximants that form MIISO's selective damping mechanism.

\section{Pad\'e approximants to the exponential}
\label{sec:pade-chapter}

\subsection{Definition and order}
\label{sec:pade}

A classical Pad\'e approximant, which we denote as~\pade{m}{n}, is a rational function
$P(z)/Q(z)$, where $P$ has degree~$m$, $Q$ has degree~$n$, and
$Q(0)=1$.
Applied to the exponential, the remaining coefficients are chosen so that
\begin{equation}
\frac{P(z)}{Q(z)}-e^{z}=O(z^{m+n+1}).
\end{equation}
The integer $m+n$ is the order of~\pade{m}{n}.
For given $m$ and~$n$ this match is unique, and no other rational function
of the same degrees attains a higher approximation order~\citep{hairerwanner1996}.
In what follows, an individual Pad\'e approximant is denoted as a \emph{branch}, and
approximants that enter a blend are denoted as \emph{parents}.
A \emph{column} is a family of branches that share a common property, for example the subdiagonal family with $m<n$, the diagonal family~\pade{n}{n}, or the zeroth-numerator family~\pade{0}{n}.

\subsection{Amplification factors and \texorpdfstring{$A$-stability}{A-stability}}
\label{sec:amplification}

On the linear test equation $y'=\lambda y$, a one-step method with step size
$h>0$ advances by a scalar amplification factor $R(z)$,
\begin{equation}
y_{n+1}=R(z)\,y_n,\qquad z=h\lambda.
\end{equation}
The method is $A$-stable if $|R(z)|\le 1$ whenever $\mathrm{Re}(z)\le 0$,
and $L$-stable if it is $A$-stable and also satisfies $R(z)\to 0$ as $|z|\to\infty$ in
the left half-plane.

\begin{definition}[$A$-acceptability]
\label{def:A-acceptable}
A rational function~$R$ is \emph{$A$-acceptable} if $|R(z)|\le 1$
for every $z$ with $\mathrm{Re}(z)\le 0$.
\end{definition}

Ehle's criterion~\citep{ehle1973,hairerwanner1996} states that~\pade{m}{n} is
$A$-acceptable if and only if
\begin{equation}
\label{eq:ehle}
n-2 \leq m \leq n.
\end{equation}

\subsection{Subdiagonal and diagonal columns}
\label{sec:gauss}

If $m<n$, the branch~\pade{m}{n} is a member of the subdiagonal column.
For every such member the numerator degree is less than the denominator degree, so
$\pade{m}{n}(z)\to 0$ as
$|z|\to\infty$ in the left half-plane and high-frequency modes are driven
out over many steps.
Whenever a member is $A$-acceptable, it satisfies
$\bigl|\pade{m}{n}(i\omega)\bigr|<1$ for all $\omega\neq 0$, so every
nontrivial oscillatory mode is dissipated to some extent.
By~\eqref{eq:ehle} this requires $n-m\le 2$.

Members of the diagonal column~\pade{n}{n} satisfy
\begin{equation}
\bigl|\pade{n}{n}(i\omega)\bigr|=1
\qquad\text{for all real }\omega,
\end{equation}
so a method with such an amplification factor introduces no dissipation on
purely oscillatory modes.

\paragraph{Naming of the Pad\'e branches}
\label{par:pade-naming}
An $n$-stage Gauss--Legendre fully implicit Runge--Kutta method has
stability function~\pade{n}{n}, an $n$-stage Radau~IIA method has
stability function~\pade{n-1}{n}, and an $n$-stage Lobatto~IIIC method has
stability function~\pade{n-2}{n}
\cite{hairerwanner1996}.
The lowest-order members recover the amplification factors of the classical one-step methods, namely that
backward Euler is~\pade{0}{1} (one-stage Radau~IIA) and the implicit
midpoint rule is~\pade{1}{1} (one-stage Gauss--Legendre).
We therefore refer to the diagonal column~\pade{n}{n} as the
\emph{Gauss} column, to the first subdiagonal column~\pade{n-1}{n} as the
\emph{Radau} column, and to the second subdiagonal column~\pade{n-2}{n} as the
\emph{Lobatto} column.
Gauss--Legendre methods are symplectic on the linear test problem and
are structure-preserving, while Radau~IIA and Lobatto~IIIC methods are
$L$-stable and introduce dissipation.
This contrast between a structure-preserving branch and a dissipative
branch is the mechanism MIISO draws on to treat resolved and outlier
modes selectively, and the motivation behind selecting
the above Pad\'e approximants.

\subsection{High-frequency dissipation}
\label{sec:dissipation}

The rate at which~\pade{m}{n} decays to zero as $|z|\to\infty$ is set by
the excess of the denominator degree~$n$ over the numerator degree~$m$.
A larger excess means faster decay and stronger per-step suppression of
high-magnitude modes. The strongest decay occurs when $m=0$.

The column~\pade{0}{n} has constant numerator~$1$ and denominator equal to
the degree-$n$ truncation of the Taylor series of $e^{-z}$,
\begin{equation}
\pade{0}{n}(z)
=
\Biggl(\sum_{j=0}^{n}\frac{(-z)^{j}}{j!}\Biggr)^{-1}.
\end{equation}
Hence $\pade{0}{n}(z)=O(z^{-n})$ as $|z|\to\infty$.
The first two members are
\begin{equation}
\pade{0}{1}(z)=\frac{1}{1-z},
\qquad
\pade{0}{2}(z)=\frac{1}{1-z+\tfrac12 z^{2}},
\end{equation}
the amplification factors of backward Euler and the two-stage Lobatto~IIIC method.

The column~\pade{1}{n-1} decays only as $O(z^{-(n-1)})$, weaker than the aforementioned~\pade{0}{n} at the same order~$n$.
We thus use the dissipative column~\pade{0}{n} on the IGA outliers because its
decay is the strongest the Pad\'e table allows at a given order.
Each step multiplies an outlier eigenmode by~$R(z)$, and
after $s$ steps the retained amplitude is $|R(z)|^s$, so a small
per-step factor drives the mode out of the solution.

The dissipative column~\pade{0}{n} meets~\eqref{eq:ehle} only for $n\le 2$.
Thus~\pade{0}{2} may be used as written, while~\pade{0}{3}
through~\pade{0}{6} are not $A$-acceptable.
The next subsection repairs those branches so that they become
$A$-acceptable and can be used on the IGA outliers.

\subsection{Repair of the subdiagonal column}
\label{sec:wedge}

The subdiagonal column~\pade{0}{N} fails $A$-acceptability on a bounded
region of the closed left half-plane for $N\ge 3$.
Repairing the instability requires a non-holomorphic blend of the raw column and a repair parent.
The repair parent satisfies Ehle's criterion~\eqref{eq:ehle} and agrees with~$e^{z}$ through the same order as~\pade{0}{N}, with the two parents blended continuously across the boundary of the region.
Outside the region, the map returns to the original column~\pade{0}{N} to maintain its strong decay.
The construction and resulting argument that the repaired maps are $A$-acceptable and $L$-stable
for all members of the family are given in full in Appendix~\ref{app:repair}.

The failing region and the repair parent differ by order, and the
repair parent alternates between members of the Radau and Lobatto columns of
Section~\ref{sec:gauss}.
Table~\ref{tab:repair-parents} lists the raw column and the repair
parent for $N=3$ through~$6$.

\begin{table}[H]
\centering
\caption{Raw subdiagonal column and repair parent for each MIISO order.}
\label{tab:repair-parents}
\begin{tabular}{@{}llll@{}}
\toprule
$N$ & Raw column & Repair parent & Repair type \\
\midrule
3 & \pade{0}{3} & \pade{1}{2} & Radau \\
4 & \pade{0}{4} & \pade{1}{3} & Lobatto \\
5 & \pade{0}{5} & \pade{2}{3} & Radau \\
6 & \pade{0}{6} & \pade{2}{4} & Lobatto \\
\bottomrule
\end{tabular}
\end{table}

Figure~\ref{fig:pade-repair} and Figure~\ref{fig:pade-repair-hi} show the
raw column, the repair parent, and the repaired map for each order.
\begin{figure}[H]
\centering
\includegraphics[width=\textwidth]{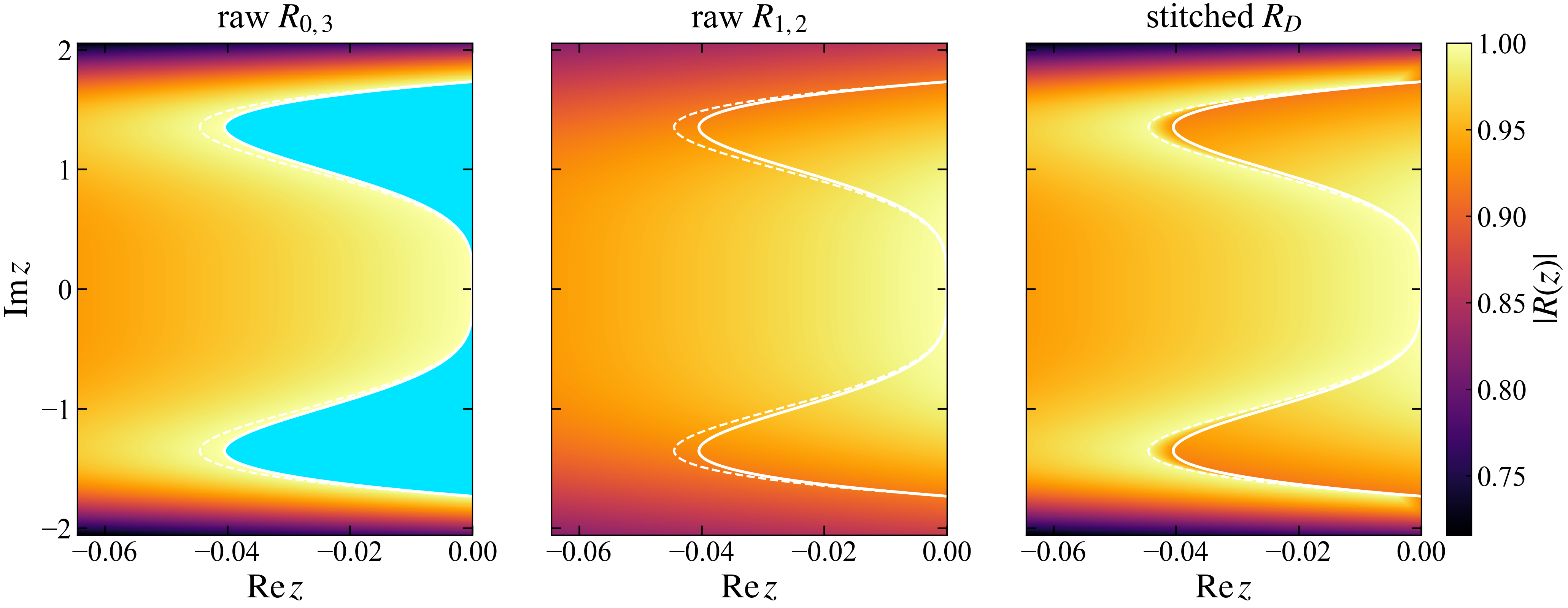}\\[0.6em]
\includegraphics[width=\textwidth]{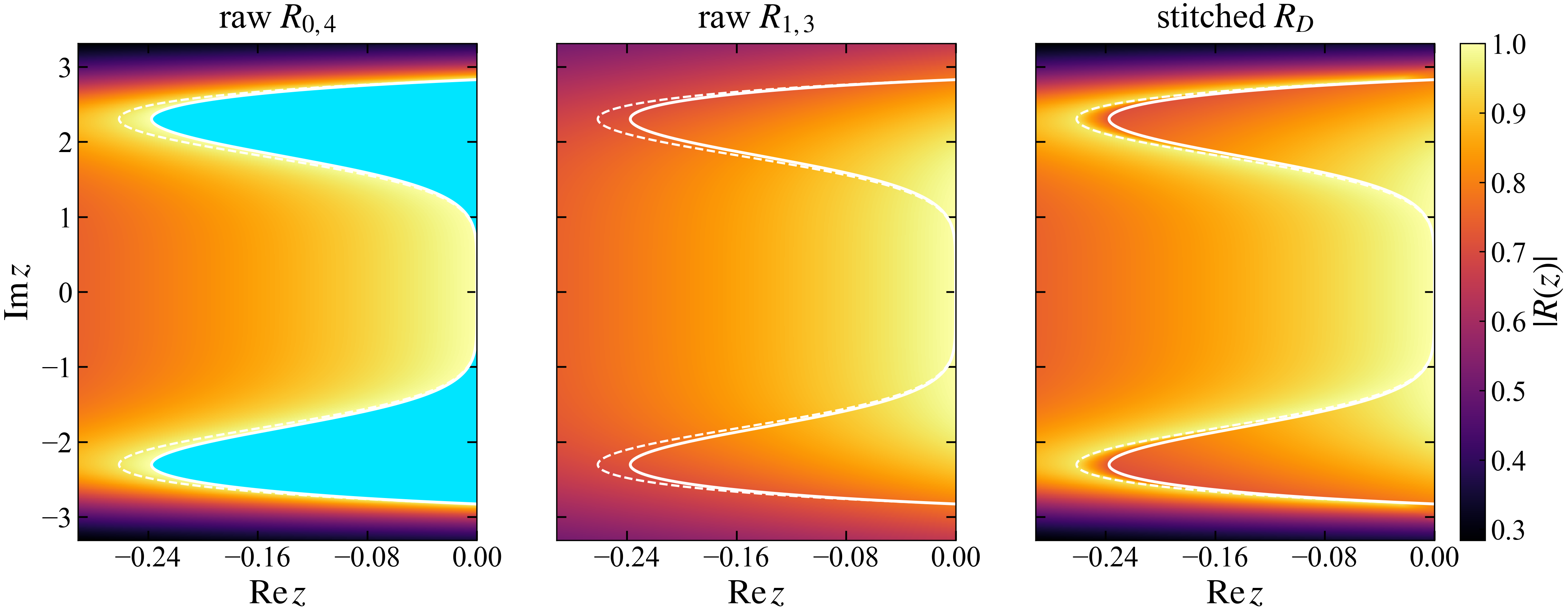}
\caption{Repair of the dissipative column~\pade{0}{N} for $N=3,4$. Cyan
marks $|R(z)|>1$. The solid curve marks the boundary of the instability region and the dashed
curve marks the outer edge of the transition region.}
\label{fig:pade-repair}
\end{figure}
\begin{figure}[H]
\centering
\includegraphics[width=\textwidth]{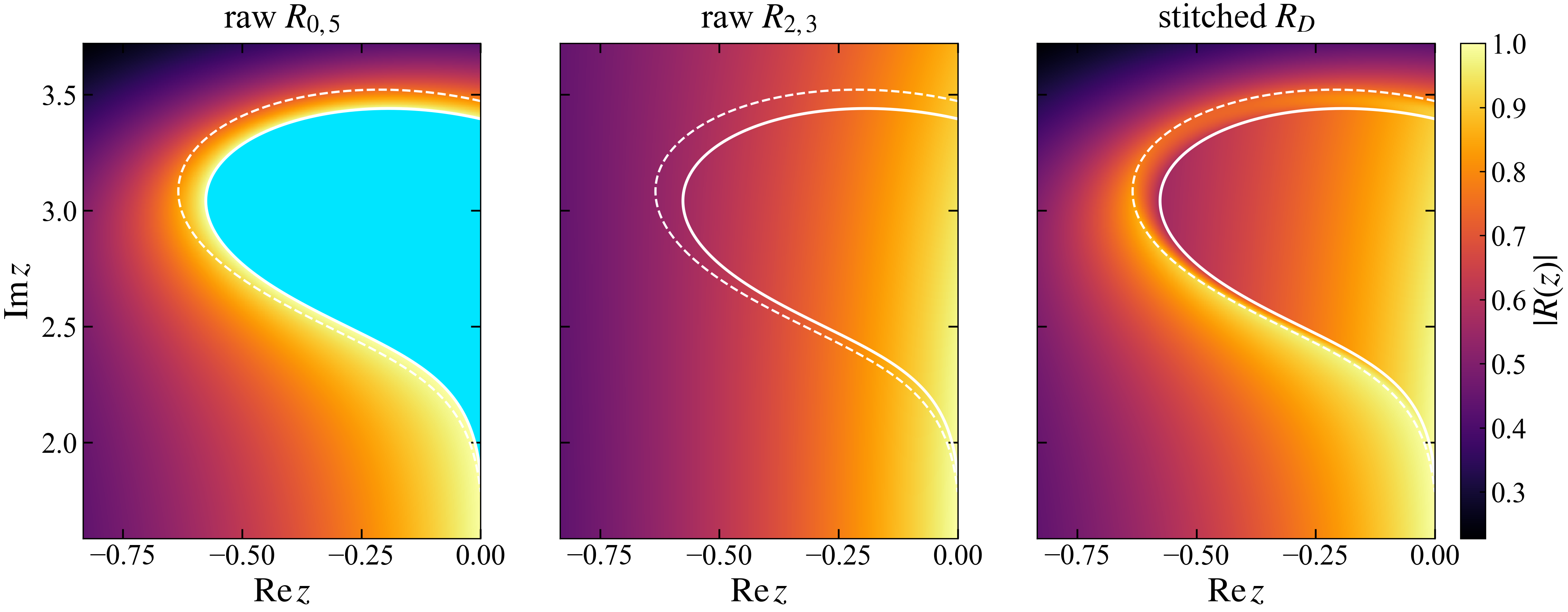}\\[0.6em]
\includegraphics[width=\textwidth]{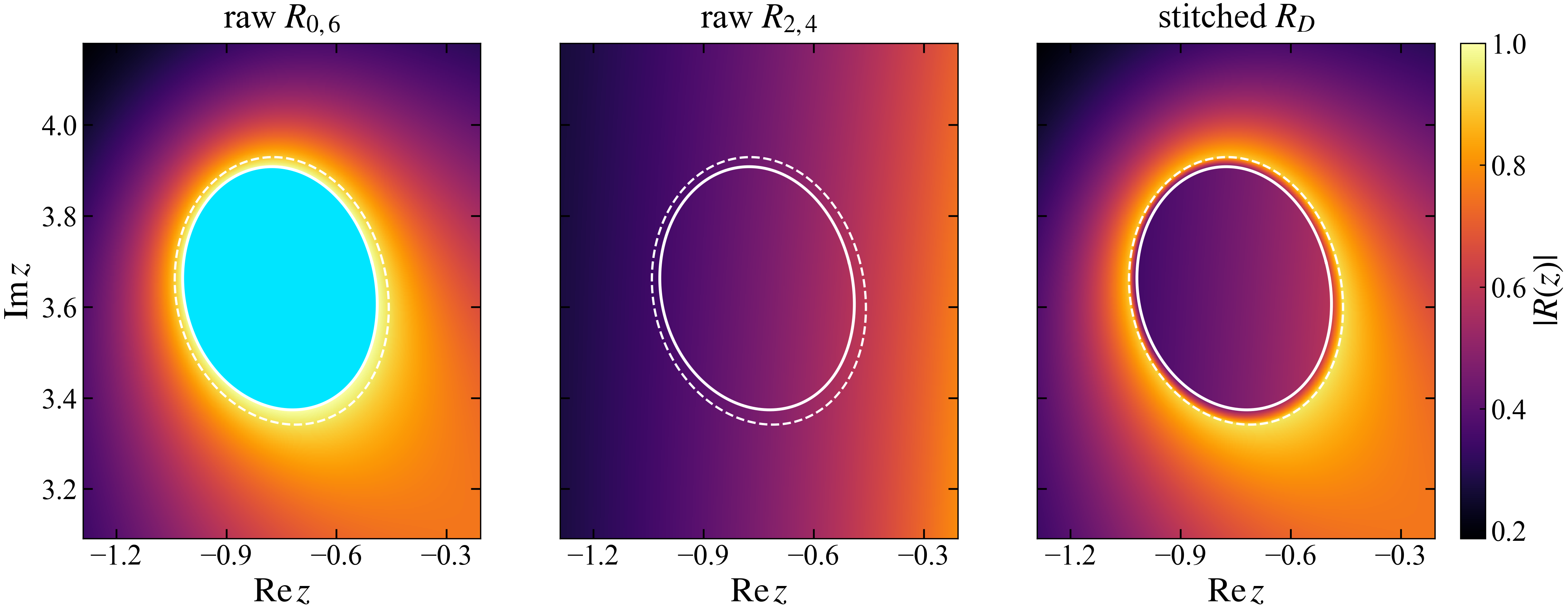}
\caption{Repair of the dissipative column~\pade{0}{N} for $N=5,6$
in the same convention.
The upper half of the complex plane is shown. The lower half is
the complex conjugate.}
\label{fig:pade-repair-hi}
\end{figure}

\section{Construction of MIISO}
\label{sec:construction}

\subsection{The amplification map}
\label{sec:blend}

Section~\ref{sec:outliers} defines the cutoff~$\omega_c=|\lambda_k(hJ)|$
in~\eqref{eq:omegac} as the modulus of the $k$-th largest-magnitude
eigenvalue of~$hJ$. Recall that the modes above the cutoff are in the outlier band and the modes below the cutoff are in the resolved band.
To construct an amplification map that selectively damps the outlier band, we blend from a diagonal Pad\'e parent to a
subdiagonal Pad\'e parent at this cutoff.
A second frequency~$\omega_0<\omega_c$ marks where the switch between the two parents begins.

Let $z=h\lambda\in\mathbb{C}$ and $\omega=|z|$.
For a method of order~$N$, write $R_C$ for the conservative Gauss parent
and $R_D$ for the dissipative parent, repaired as in
Section~\ref{sec:wedge} when the raw column~\pade{0}{N} is not
$A$-acceptable.
Resolved modes are advanced by~$R_C$, which has modulus one on the
imaginary axis.
Outliers are advanced by~$R_D$, which drives them out over many steps by the dissipative properties of the subdiagonal column.
The parents for each member of the family are listed in
Table~\ref{tab:parents}.

\begin{table}[H]
\centering
\caption{Conservative parent~$R_C$, raw dissipative parent~$R_D$, and
repair parent for each MIISO order.}
\label{tab:parents}
\begin{tabular}{@{}lccc@{}}
\toprule
Method & $R_C$ & $R_D$ (raw) & $R_D$ (repair) \\
\midrule
MIISO2 & Gauss~\pade{1}{1} & Lobatto~\pade{0}{2} & --- \\
MIISO3 & Gauss~\pade{2}{2} & \pade{0}{3} & Radau~\pade{1}{2} \\
MIISO4 & Gauss~\pade{2}{2} & \pade{0}{4} & Lobatto~\pade{1}{3} \\
MIISO5 & Gauss~\pade{3}{3} & \pade{0}{5} & Radau~\pade{2}{3} \\
MIISO6 & Gauss~\pade{3}{3} & \pade{0}{6} & Lobatto~\pade{2}{4} \\
\bottomrule
\end{tabular}
\end{table}

Fix $r\in(0,1)$ and an integer $s>1$, and set $\omega_0=r\,\omega_c$.
The defaults are $r=0.99$ and $s=4$, so the transition occupies
$0.99\,\omega_c\le\omega\le\omega_c$.
The taper
\begin{equation}
\theta(\omega)=
\begin{cases}
0, & \omega\le\omega_0,\\[2pt]
\displaystyle
\Bigl(\dfrac{\omega-\omega_0}{\omega_c-\omega_0}\Bigr)^{s},
& \omega_0<\omega<\omega_c,\\[6pt]
1, & \omega\ge\omega_c.
\end{cases}
\label{eq:taper}
\end{equation}
weights the non-holomorphic blend.
The blended amplification map is
\begin{equation}
R(z)
=
R(z,|z|)
=
|R_D(z)|^{\theta(|z|)}\,|R_C(z)|^{1-\theta(|z|)}\,
\exp\!\Bigl(
i\bigl[
\bigl(1-\theta(|z|)\bigr)\arg R_C(z)
+
\theta(|z|)\arg R_D(z)
\bigr]
\Bigr),
\label{eq:rblend}
\end{equation}
with $\theta$ from~\eqref{eq:taper}, and with $R_C$ and $R_D$ taken from
Table~\ref{tab:parents} for the chosen order~$N$.
Since $\theta$ depends on $\omega = |z|$, the resulting map $R(z)$ is non-holomorphic.
This structure is the mechanism that enables the damping to be placed at the spectral boundary between resolved and outlier modes.
On an eigenmode with $z=h\lambda$, each step multiplies by~$R(z)$.

\begin{note}[Notation for the non-holomorphic amplification map]
\label{note:non-holomorphic}
Because the weight~$\theta$ depends on~$|z|$, the map~$R$ is not
holomorphic.
We write $R(z)$ with the understanding that the value also depends
on~$|z|$.
\end{note}

Figure~\ref{fig:taper-amplification} shows the taper $\theta(\omega)$ and
the imaginary-axis amplification factor $|R(i\omega)|$ for each method
order.
The corresponding complex-plane surfaces $|R(z)|$ are shown after the
$A$-stability argument in Appendix~\ref{app:stability-convergence}.

\begin{figure}[H]
  \centering
  \includegraphics[width=\textwidth]{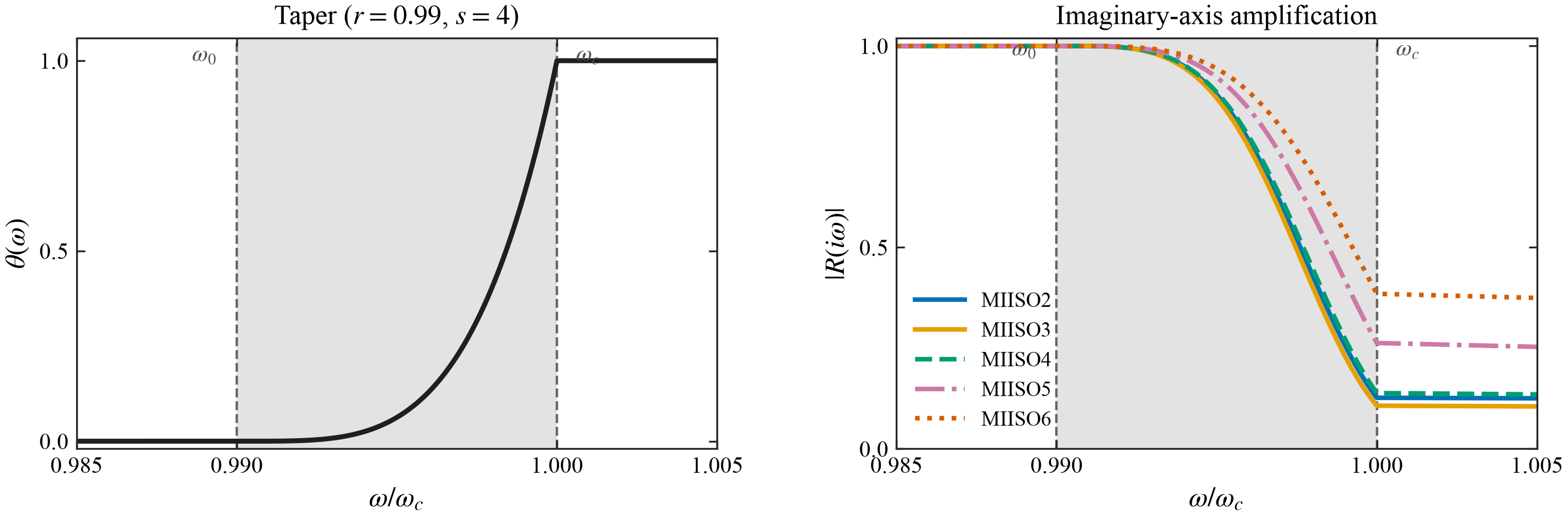}
  \caption{Taper $\theta(\omega)$ (left) and
    imaginary-axis amplification $|R(i\omega)|$ (right).}
  \label{fig:taper-amplification}
\end{figure}

\subsection{The cutoff frequency \texorpdfstring{$\omega_c$}{omega\_c}}
\label{sec:cutoff}

The cutoff~$\omega_c$, as defined in~\eqref{eq:omegac}, is the scalar
location of the frequency corresponding to the first of the $k$ outlier
eigenvalues.

MIISO accepts $\omega_c$ in one of two ways.
First, the user can directly supply~$\omega_c$, in which case the value is fixed and no
eigensolve is performed.
Second, the user can supply the outlier count~$k$, either directly or
through a boundary-condition helper,
which returns~$k$ from the polynomial degree and boundary conditions as detailed in Section~\ref{sec:outliers}.
Given~$k$, an implicitly restarted Arnoldi
iteration~\citep{sorensen1992,lehoucq1998} recovers~$\omega_c$
matrix-free.
If the system is second order, the block structure of the Jacobian is
exploited for a faster solve. Otherwise the iteration is run on the full
Jacobian.
Both paths are stated as Algorithm~\ref{alg:cutoff-setup} in
Appendix~\ref{app:cutoff}.

\subsection{Tracking the cutoff}
\label{sec:cutoff-tracking}

For linear and autonomous problems, the setup of
Section~\ref{sec:cutoff} is sufficient, and $\omega_c$ remains fixed for
the entirety of the integration.

For nonlinear and nonautonomous problems, $\omega_c$ is tracked at each
step by a Rayleigh quotient update, stated as
Algorithm~\ref{alg:cutoff-track} in Appendix~\ref{app:cutoff}.
The update is matrix-free and costs one Jacobian--vector product per
step, making it efficient to compute. It keeps the cutoff aligned with the live resolved--outlier
boundary when eigenvalues shift under nonlinear deformation.
The conditions under which this tracked value remains a valid cutoff are
given in Appendix~\ref{app:cutoff-bound}.

\begin{algorithm}[H]
\caption{Cutoff placement and tracking}
\label{alg:cutoff-main}
\begin{algorithmic}[1]
\Require Jacobian~$J$, step size~$h$, and either $\omega_c$, or $k$, or
boundary data to obtain~$k$
\If{$\omega_c$ supplied}
\State keep $\omega_c$ fixed
\Else
\If{$k$ not supplied}
\State obtain $k$ from the boundary-condition helper
\EndIf
\If{the system is second order}
\State recover $\omega_c$ by Arnoldi on the spatial block (Algorithm~\ref{alg:cutoff-setup})
\Else
\State recover $\omega_c$ by Arnoldi on the full Jacobian (Algorithm~\ref{alg:cutoff-setup})
\EndIf
\EndIf
\For{each step}
\If{the problem is linear and autonomous}
\State keep $\omega_c$ fixed
\Else
\State update $\omega_c$ by the Rayleigh quotient tracking of Algorithm~\ref{alg:cutoff-track}
\EndIf
\EndFor
\end{algorithmic}
\end{algorithm}

\subsection{Modified \texorpdfstring{$\varphi$}{phi}-functions and the MIISO family}
\label{sec:phi}

Exponential Rosenbrock methods advance with $\varphi$-functions of~$hJ$
generated from the exact recurrence
$\varphi_0(z)=e^{z}$, $\varphi_{j+1}(z)=(\varphi_j(z)-1/j!)/z$
\citep{hochbruckostermann2009}.
MIISO replaces $e^{z}$ by the blend~$R$ of~\eqref{eq:rblend} in the same
recurrence,
\begin{equation}
\phit_0(z)=R(z),
\qquad
\phit_{j+1}(z)=\frac{\phit_j(z)-1/j!}{z}.
\label{eq:phirec}
\end{equation}
Near $z=0$ this recurrence suffers cancellation in floating point
arithmetic, so the Taylor series values
$\phit_1\approx 1+z/2$, $\phit_2\approx 1/2+z/6$,
$\phit_3\approx 1/6+z/24$, $\phit_4\approx 1/24+z/120$,
$\phit_5\approx 1/120+z/720$, and $\phit_6\approx 1/720+z/5040$ are used
for $|z|<\varepsilon$.

\begin{remark}[Modified exponential calculus]
\label{rem:phi-analysis}
Note certain limitations of the modified $\varphi$-functions MIISO uses.
The maps~$\phit_j$ inherit the non-holomorphic character of~$R$, so
semigroup theory for exact exponential integrators does not apply.
Because the~$\phit_j$ are built from the selective damping maps~$R$, they
differ from the exact exponential $\varphi$-functions on the stiff
spectrum, and stiff-order results for the underlying Rosenbrock schemes
do not transfer without a separate argument.
\end{remark}

Since every~$\phit_j$ is generated from~$R$, all stages of every member
of the family use the selective damping map.

MIISO2 through MIISO4 follow the exponential Rosenbrock--Euler,
exprb32, and exprb43 schemes of~\citet{hochbruckostermann2009} with
$\varphi_j$ replaced by~$\phit_j$.
MIISO5 and MIISO6 follow the parallel-stage schemes pexprb54s5 and
pexprb65s7 of~\citet{luanostermann2016}, with the same replacement.
The internal stages~$Y_i$ are evaluations of the frozen linearization,
and the defects~$D_i$ measure the residual of that linearization at
those stages,
\begin{equation}
D_i=f(t_n+c_ih,Y_i)-f_n-J_n(Y_i-y_n)-c_ih f_{t,n}.
\label{eq:defect}
\end{equation}

For MIISO2.
\begin{equation}
y_{n+1}=y_n+h\,\phit_1(hJ_n)f_n+\tfrac12 h^2 f_{t,n}.
\label{eq:miiso2}
\end{equation}

For MIISO3.
\begin{equation}
\begin{aligned}
v_1&=\phit_1(hJ_n)f_n,\quad
v_2=\phit_2(hJ_n)f_{t,n},\\
Y_2&=y_n+hv_1+h^2 v_2,\\
D_2&=f(t_n+h,Y_2)-f_n-J_n(Y_2-y_n)-h f_{t,n},\\
y_{n+1}&=y_n+hv_1+h^2 v_2+2h\,\phit_3(hJ_n)D_2.
\end{aligned}
\label{eq:miiso3}
\end{equation}

For MIISO4.
\begin{equation}
\begin{aligned}
v_1&=\phit_1(hJ_n)f_n,\qquad
v_1^{1/2}=\phit_1(\tfrac12 hJ_n)f_n,\\
v_2&=\phit_2(hJ_n)f_{t,n},\qquad
v_2^{1/2}=\phit_2(\tfrac12 hJ_n)f_{t,n},\\
Y_2&=y_n+\tfrac12 h\,v_1^{1/2}+\tfrac14 h^2 v_2^{1/2},\\
D_2&=f(t_n+\tfrac12 h,Y_2)-f_n-J_n(Y_2-y_n)-\tfrac12 h f_{t,n},\\
Y_3&=y_n+h\,v_1+h^2 v_2,\\
D_3&=f(t_n+h,Y_3)-f_n-J_n(Y_3-y_n)-h f_{t,n},\\
y_{n+1}&=y_n+h\,\phit_1(hJ_n)f_n+h^2\phit_2(hJ_n)f_{t,n}\\
&\quad
+h\Bigl[(16\phit_3(hJ_n)-48\phit_4(hJ_n))D_2
+(-2\phit_3(hJ_n)+12\phit_4(hJ_n))D_3\Bigr].
\end{aligned}
\label{eq:miiso4}
\end{equation}

For MIISO5.
\begin{equation}
\begin{aligned}
v_1&=\phit_1(hJ_n)f_n,\qquad
v_2=\phit_2(hJ_n)f_{t,n},\\
Y_2&=y_n+\tfrac12 h\,\phit_1(\tfrac12 hJ_n)f_n+\tfrac14 h^2\phit_2(\tfrac12 hJ_n)f_{t,n},\\
D_2&=f(t_n+\tfrac12 h,Y_2)-f_n-J_n(Y_2-y_n)-\tfrac12 h f_{t,n},\\
Y_3&=Y_2+\tfrac12 h\,\phit_3(\tfrac12 hJ_n)D_2,\\
Y_4&=y_n+\tfrac13 h\,\phit_1(\tfrac13 hJ_n)f_n+\tfrac19 h^2\phit_2(\tfrac13 hJ_n)f_{t,n}
+\tfrac{8}{27}h\,\phit_3(\tfrac13 hJ_n)D_2,\\
Y_5&=y_n+hv_1+h^2 v_2+8h\,\phit_3(hJ_n)D_2,\\
D_i&=f(t_n+c_ih,Y_i)-f_n-J_n(Y_i-y_n)-c_ih f_{t,n},\quad i=3,4,5,\\
w_3&=-32D_3+81D_4+D_5,\\
w_4&=384D_3-729D_4-15D_5,\\
w_5&=-1152D_3+1944D_4+72D_5,\\
y_{n+1}&=y_n+hv_1+h^2v_2+h\bigl[\phit_3(hJ_n)w_3+\phit_4(hJ_n)w_4+\phit_5(hJ_n)w_5\bigr].
\end{aligned}
\label{eq:miiso5}
\end{equation}

For MIISO6.
\begin{equation}
\begin{aligned}
v_1&=\phit_1(hJ_n)f_n,\qquad
v_2=\phit_2(hJ_n)f_{t,n},\\
Y_2&=y_n+\tfrac12 h\,\phit_1(\tfrac12 hJ_n)f_n+\tfrac14 h^2\phit_2(\tfrac12 hJ_n)f_{t,n},\\
Y_3&=y_n+\tfrac14 h\,\phit_1(\tfrac14 hJ_n)f_n+\tfrac1{16}h^2\phit_2(\tfrac14 hJ_n)f_{t,n},\\
D_2&=f(t_n+\tfrac12 h,Y_2)-f_n-J_n(Y_2-y_n)-\tfrac12 h f_{t,n},\\
D_3&=f(t_n+\tfrac14 h,Y_3)-f_n-J_n(Y_3-y_n)-\tfrac14 h f_{t,n},\\
Y_4&=Y_2+h\bigl[(-\phit_3+6\phit_4)(\tfrac12 hJ_n)D_2+(8\phit_3-24\phit_4)(\tfrac12 hJ_n)D_3\bigr],\\
Y_5&=Y_3+h\bigl[(-\tfrac18\phit_3+\tfrac38\phit_4)(\tfrac14 hJ_n)D_2+(\phit_3-\tfrac32\phit_4)(\tfrac14 hJ_n)D_3\bigr],\\
Y_6&=y_n+\tfrac13 h\,v_1+\tfrac19 h^2 v_2
+h\bigl[(-\tfrac{8}{27}\phit_3+\tfrac{32}{27}\phit_4)(\tfrac13 hJ_n)D_2+(\tfrac{64}{27}\phit_3-\tfrac{128}{27}\phit_4)(\tfrac13 hJ_n)D_3\bigr],\\
Y_7&=y_n+hv_1+h^2v_2
+h\bigl[(-8\phit_3+96\phit_4)(hJ_n)D_2+(64\phit_3-384\phit_4)(hJ_n)D_3\bigr],\\
D_i&=f(t_n+c_ih,Y_i)-f_n-J_n(Y_i-y_n)-c_ih f_{t,n},\quad i=4,5,6,7,\\
w_3&=32D_4+\tfrac{1024}{3}D_5-243D_6-\tfrac13 D_7,\\
w_4&=-768D_4-6144D_5+5103D_6+9D_7,\\
w_5&=7296D_4+45056D_5-40824D_6-104D_7,\\
w_6&=-23040D_4-122880D_5+116640D_6+480D_7,\\
y_{n+1}&=y_n+hv_1+h^2v_2+h\bigl[\phit_3(hJ_n)w_3+\phit_4(hJ_n)w_4+\phit_5(hJ_n)w_5+\phit_6(hJ_n)w_6\bigr].
\end{aligned}
\label{eq:miiso6}
\end{equation}

The coefficients in~\eqref{eq:miiso5} are those of pexprb54s5, and the
coefficients in~\eqref{eq:miiso6} are those of pexprb65s7
\citep{luanostermann2016}.
In the linear case every defect vanishes and each step reduces to
multiplication by~$R$, so each eigenmode with $z=h\lambda$ is advanced
by~$R(z)$.

\subsection{Matrix-free evaluation}
\label{sec:matrixfree}

Every quantity in~\eqref{eq:miiso2} through~\eqref{eq:miiso6} reduces to
an evaluation of $\phit_j(hJ_n)v$ for some vector~$v$.
As in Section~\ref{sec:exprb}, each such product is computed by an
Arnoldi iteration on~$hJ_n$, producing a basis~$V_m$ and a Hessenberg
matrix~$H_m$.
Because the blend~$R$ is not holomorphic, $H_m$ cannot be passed to a
rational matrix function directly as in~\eqref{eq:ritz}.
To form $\phit_j(hJ_n)v$, $H_m$ is diagonalized once, and each~$\phit_j$ is applied to the
resulting Ritz values as a scalar.
This holds for every stage of every member of the family, so every step of
every MIISO method avoids forming a matrix or a matrix function explicitly.

The Krylov dimension~$m$ is chosen adaptively.
The procedure increases~$m$ one Arnoldi vector at a time and, at fixed
intervals, compares the current $\phit_j(hJ_n)v$ with the value from the
previous check.
When that change falls below a tolerance, the iteration stops.
Steps with small defects or a well-separated spectrum then use fewer
Arnoldi vectors, while the procedure can raise $m$ when the residual requires it.
This gives a reduction of cost on typical steps relative to picking an
upper bound on~$m$, since the method uses higher dimensions only when
needed for accuracy.
It also eliminates the need to choose and tune a fixed Krylov
dimension by hand for each problem and step size.
Algorithm~\ref{alg:adaptive-arnoldi} in Appendix~\ref{app:adaptive-m}
states the procedure.

Several of the Arnoldi runs within a single step act on vectors that do
not depend on one another, for instance $f_n$ and~$f_{t,n}$, or the
defects at different internal stages once those stages are available.
These independent runs are executed concurrently across threads, with a
sequential fallback when only one thread is available. This mirrors
the parallel approaches of~\citet{luanostermann2016} for the underlying exponential Rosenbrock schemes
and allows for further cost reductions.

Recall that the cutoff setup and tracking of
Section~\ref{sec:cutoff-tracking} also use~$J$ only through
matrix-vector products.
Together with the Arnoldi evaluation above, MIISO is therefore fully
matrix-free.
The concurrent stage evaluations listed above give a further route to
efficiency through parallelization.

\subsection{Work and storage}
\label{sec:workstorage}

Locating~$\omega_c$ from~$k$ is a one-time cost needed only when the user does not supply it directly.
With the block structure of Algorithm~\ref{alg:cutoff-setup}, Arnoldi
runs on the spatial operator~$A$ for $k+1$ eigenpairs.
Without the block structure, it runs on the full Jacobian~$J_0$ for the
largest-magnitude $2k$ eigenpairs, since outlier eigenvalues of a general
$J$ occur in conjugate pairs.
Either path uses a short sequence of matrix-vector products and one small
dense eigendecomposition without factorization, run once at the start of
the integration.

An adaptive Krylov evaluation of $\phit_j(hJ)v$
(Algorithm~\ref{alg:adaptive-arnoldi}) costs
$O(mC_J+Nm^2+m^3)$ and stores $O(Nm+m^2)$ at the dimension~$m$ it
reaches, where~$C_J$ is the cost of one Jacobian--vector product.
Here $m\le m_{\max}$ and adapts downward.
The $m^3$ term is negligible for the Krylov dimensions used in practice.
Tracking~$\omega_c$ by Algorithm~\ref{alg:cutoff-track} adds one
Jacobian--vector product per step.

The number of Krylov evaluations of $\phit_j(hJ)v$ per step is fixed by the order.
MIISO2 uses one evaluation, on~$f_n$.
MIISO3 uses at most three, on~$f_n$, $f_{t,n}$, and~$D_2$.
MIISO4 uses at most four, on~$f_n$, $f_{t,n}$, $D_2$, and~$D_3$.
MIISO5 uses at most six, on~$f_n$, $f_{t,n}$, $D_2$, $w_3$, $w_4$,
and~$w_5$.
MIISO6 uses at most eight, on~$f_n$, $f_{t,n}$, $D_2$, $D_3$, $w_3$,
$w_4$, $w_5$, and~$w_6$.
An evaluation is skipped when its vector falls below the defect
threshold~$\tau$.
Independent evaluations run concurrently across threads.

In a single step the Pad\'e parents of Section~\ref{sec:pade-chapter} are blended into the
amplification map~$R$ of Section~\ref{sec:blend}, with the switch placed
at the cutoff~$\omega_c$ from Section~\ref{sec:cutoff}.
That cutoff is either static or updated matrix-free by the Rayleigh
quotient of Section~\ref{sec:cutoff-tracking}.
The exact exponential in the Rosenbrock~$\varphi$-functions is replaced by
this map~$R$, which yields the modified functions~$\phit_j$ and the
stage formulas~\eqref{eq:miiso2}--\eqref{eq:miiso6}.
Every product $\phit_j(hJ)v$ is then formed by the adaptive
Arnoldi procedure of Algorithm~\ref{alg:adaptive-arnoldi}.
The resulting Hessenberg matrix is diagonalized, $\phit_j$ is applied to
its Ritz values, and the process is repeated for each independent
vector required by the stages until the final update of~$y$ is formed.
A full step of a general MIISO method of order~$N$ is given as
Algorithm~\ref{alg:miiso-general} in Appendix~\ref{app:miiso-step}.

\section{Discussion}
\label{sec:discussion}

The construction of MIISO from the previous sections gives the following properties.

\subsection{Resolved and outlier mode behavior}
\label{sec:discuss-modes}

Below the taper, every eigenmode is advanced with no numerical
dissipation.
Above the cutoff, every outlier eigenmode is advanced by an
$L$-stable map and driven toward zero.
Both statements hold exactly for every member of the family, and are
proved in Appendix~\ref{app:mode-behavior}.

\subsection{Step size and the withdrawal of dissipation}
\label{sec:discuss-withdrawal}

The cutoff and the taper are placed in terms of~$z=h\lambda$, so the
dissipation applied to a fixed physical frequency~$\lambda$ depends on
the step size.
As~$h$ is refined for a fixed~$\lambda$, $z$ moves toward the origin,
and the dissipative parent's modulus moves toward one, so less
dissipation is applied per step.
We call this phenomenon the withdrawal of dissipation.

Away from the origin, higher-order members dissipate more strongly, as
established in Section~\ref{sec:dissipation}.
Closer to the origin, the relationship is reversed since higher-order members
approximate the exact exponential more accurately.
The lower-order members withdraw more slowly as~$h\to 0$ and remain
more dissipative there, and the higher-order members lose their dissipation faster.
The rate for each member is given in
Proposition~\ref{prop:withdrawal-general} of
Appendix~\ref{app:mode-behavior}.

This phenomenon follows from consistency of the approximate amplification maps.
The dissipative map must agree with~$e^{z}$ near the origin to the
order the method claims, and the effect is close to negligible in
practice.
Higher-order members can take larger steps for the same accuracy, and
on realistic isogeometric discretizations the step size at which
withdrawal becomes significant is smaller than any step size used for
accuracy or efficiency.

Thus, no significant accommodation for withdrawal is needed.
The recommended usage is to set a step size for accuracy and efficiency as usual.
The chance that this step size falls in the withdrawal regime is low
unless the step is already very small or the outlier band is very close to the origin, and if it is, increasing the
step size resolves it.
A simple rule for guaranteeing strong damping for all members of MIISO on the outlier band is one step per
period of the largest-magnitude resolved mode.

\subsection{Stability, consistency, and convergence}
\label{sec:discuss-stability}

Every member of the family is unconditionally $A$-stable, and converges
at global order~$N$ on non-stiff problems under the hypotheses stated
in Appendix~\ref{app:stability-convergence}.
Recalling Remark~\ref{rem:phi-analysis}, this convergence theory covers the
non-stiff regime only. Stiff-order theory is not covered here.

\subsection{Validity of the outlier tracking}
\label{sec:discuss-tracking}

Propositions~\ref{prop:rayleigh-bound} and~\ref{prop:mode-tracking} in
Appendix~\ref{app:cutoff-bound} give a sufficient condition under which
the outlier band stays correctly separated from the resolved band when
the cutoff is updated using the Rayleigh quotient update of Section~\ref{sec:cutoff-tracking}.
As stated in Remark~\ref{rem:cutoff-bounds-conservative}, this condition is a
worst-case bound.
In practice, tracking remains valid well past the deformation this bound certifies, and breaks only at moments of comparatively large deformation.

When tracking fails, and the gap between the outlier band and the resolved band closes, the following phenomenon occurs.
A small number of modes near the resolved--outlier boundary are
misassigned to the wrong branch of the blend for the short interval
during which the gap closes, and correct identification returns once
the deformation recedes and the gap reopens. We call this phenomenon mode mistracking.
Appendix~\ref{app:tracking-validity}, Figure~\ref{fig:tracking-validity},
shows this empirically by tracking the regions where mistracking occurs.

\subsection{Nonlinear and nonautonomous systems}
\label{sec:discuss-nonlinear}

The MIISO construction of Section~\ref{sec:construction} already
accounts for nonautonomous problems through the term~$f_t$, which is
omitted when~$f_t=0$, and for nonlinearity through the defect terms of
the higher-order members, following the local linearization
of~\citet{hochbruckostermann2009,luanostermann2016}.
Theorems~\ref{thm:A-stable-general} and~\ref{thm:convergence-general}
require only that~$J$ be frozen for the duration of one step, which
holds for any problem type.
Provided the outlier band remains correctly identified, by
Proposition~\ref{prop:mode-tracking} or empirically as in
Section~\ref{sec:discuss-tracking}, both results extend unchanged to
nonlinear and nonautonomous problems, stated as
Corollary~\ref{cor:nonlinear} in
Appendix~\ref{app:stability-convergence}.

\subsection{Efficiency and practical use}
\label{sec:discuss-efficiency}

In addition to the properties established above, MIISO is constructed to be computationally efficient and practical to use.
The construction of Section~\ref{sec:matrixfree} evaluates every
product $\phit_j(hJ)v$ by Krylov projection, and independent Krylov
evaluations run concurrently across threads.
The adaptive Krylov dimension of Section~\ref{sec:matrixfree} and the
defect threshold of Algorithm~\ref{alg:miiso-general} further reduce
cost where the problem allows it, while leaving the stability and
convergence results of this section unchanged.
Section~\ref{sec:results} confirms this cost behavior, along with
numerical confirmation of the properties established above.

The only required user input is the outlier location, either~$\omega_c$
directly or the outlier count~$k$ of Section~\ref{sec:cutoff}.
This quantity is a property of the discretization, fixed by the
polynomial degree and boundary conditions, and is not intended to be
tuned.
All other quantities, namely the taper width~$r$ and exponent~$s$
of~\eqref{eq:taper}, the Krylov safety factor and tolerance of
Algorithm~\ref{alg:adaptive-arnoldi}, and the defect threshold~$\tau$
of Algorithm~\ref{alg:miiso-general}, are internal machinery fixed at
a default. They are not intended to be tuned by the user, although
they can be adjusted if desired.
An overview of which MIISO member is recommended for a given problem
is given in~\nameref{par:choosing-order}.

\section{Numerical results}
\label{sec:results}

This section confirms the claims of the preceding sections with
numerical and empirical experiments, and demonstrates the practical use
of MIISO.

\subsection{Modal amplification}
\label{sec:results-modal}

This section shows the modal amplification of MIISO on a variety of
isogeometric discretizations, contrasting it with generalized-$\alpha$
at $\rho_\infty=1,0.5,0$,
and its mechanism detailed in Section~\ref{sec:global-damp}.
The problems vary by boundary conditions, degrees, and dimensions.

Figure~\ref{fig:modal-gallery-main} shows one case per dimension,
drawn from the full gallery of Appendix~\ref{app:modal-gallery},
confirming Proposition~\ref{prop:zero-dissipation} and
Proposition~\ref{prop:annihilation-general} against
generalized-$\alpha$.
The full gallery is given in Appendix~\ref{app:modal-gallery},
Figures~\ref{fig:modal-fig7}--\ref{fig:modal-fig12}.

\begin{figure}[H]
\centering
\includegraphics[width=\textwidth]{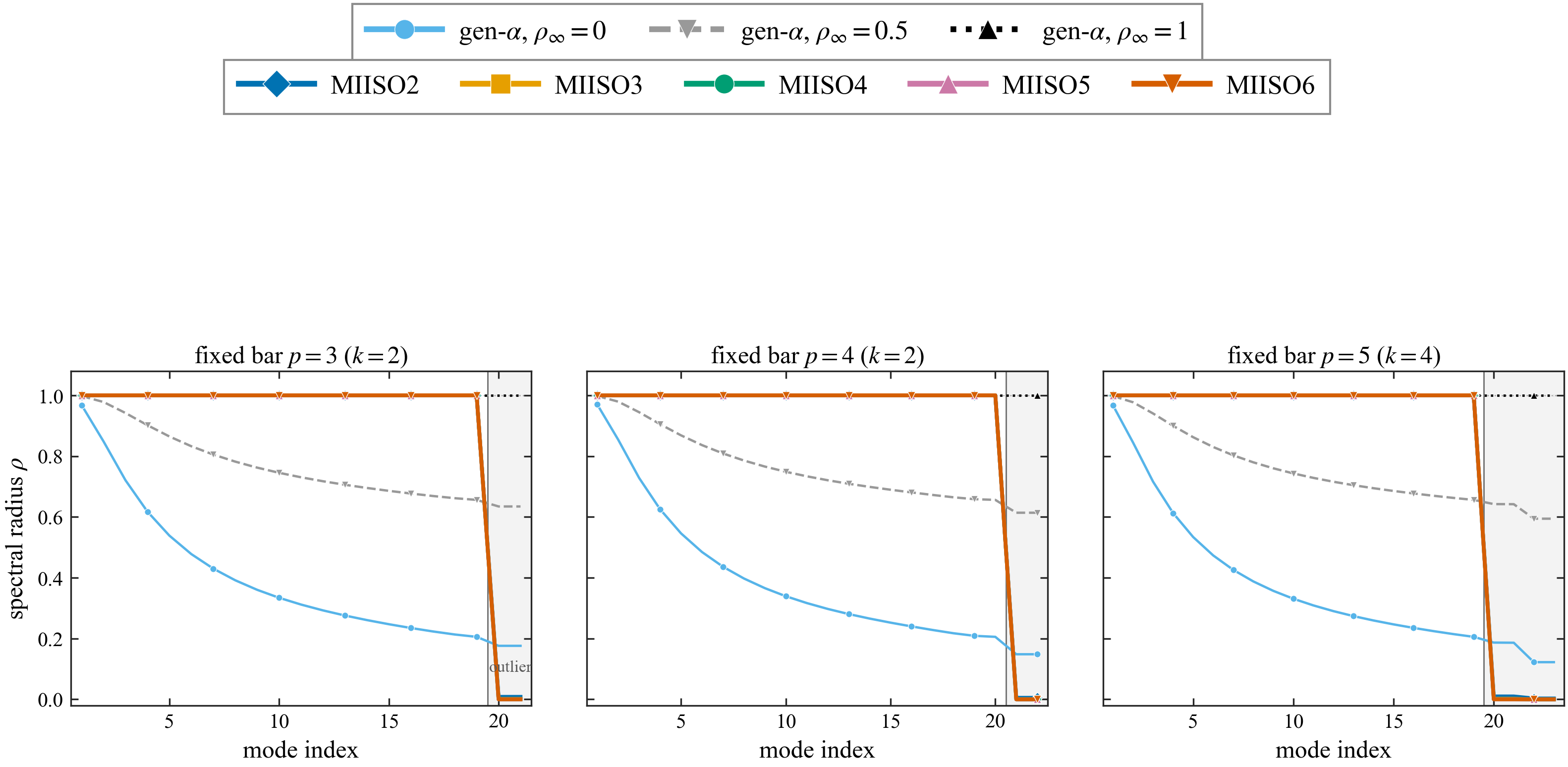}\\[0.6em]
\includegraphics[width=\textwidth]{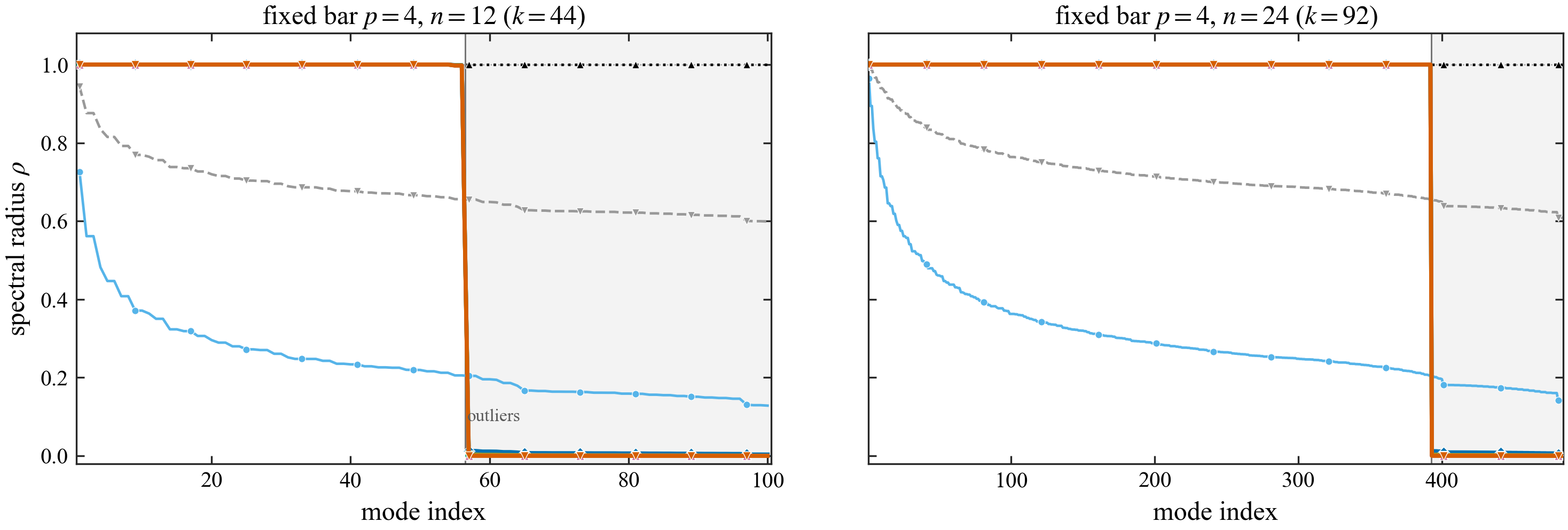}\\[0.6em]
\includegraphics[width=\textwidth]{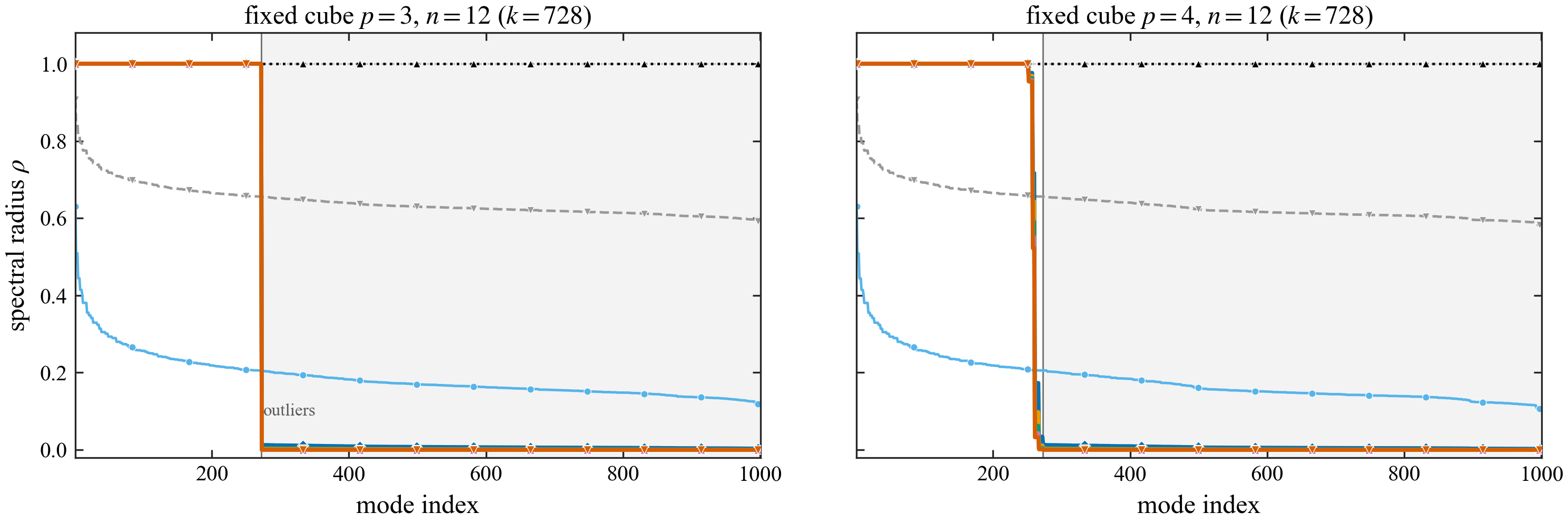}
\caption{One case per dimension of the fully discrete modal
amplification with fixed--fixed boundary conditions.}
\label{fig:modal-gallery-main}
\end{figure}

\subsection{Resolved and outlier mode evolution}
\label{sec:results-energy}

Section~\ref{sec:discuss-modes} showed that MIISO leaves resolved modes
undamped and drives outlier modes to zero.
Here that is checked in time on a linear free-vibration problem,
against a spatial outlier-removal baseline~\citep{hiemstra2021}.

We integrate the linear membrane
\begin{equation}
M\ddot q + Kq = 0
\label{eq:membrane-linear}
\end{equation}
on the unit square with clamped boundary.
Let $(\varphi_n,\omega_n)_{n=1}^N$ be the $M$-orthonormal eigenpairs of
the rest pencil $(K,M)$, with $0\le\omega_1\le\cdots\le\omega_N$, and
write the resolved and outlier bands as
$\{\varphi_1,\ldots,\varphi_{N-k}\}$ and
$\{\varphi_{N-k+1},\ldots,\varphi_N\}$, following
Section~\ref{sec:outliers}.
With $c_n=\varphi_n^\top Mq$, we define the resolved and outlier energies as
\begin{equation}
E_{\mathrm{res}}(q)=\Bigl(\sum_{n=1}^{N-k}c_n^2\Bigr)^{1/2},
\qquad
E_{\mathrm{out}}(q)=\Bigl(\sum_{n=N-k+1}^{N}c_n^2\Bigr)^{1/2},
\label{eq:energies}
\end{equation}
both constant on the undamped exact evolution
of~\eqref{eq:membrane-linear}.

The discretization is a degree-$p=4$ B-spline basis
on an $8\times 8$ mesh, $N=100$, with outlier count $k=36$.
The initial displacement seeds the outlier band at $5\%$ of the
resolved energy,
\begin{equation}
q(0)=\varphi_1+0.05\,\frac{1}{\sqrt{k}}\sum_{n=N-k+1}^{N}\varphi_n,
\qquad
\dot q(0)=0,
\label{eq:membrane-ic}
\end{equation}
and the step size is $h=2\pi/\omega_{N-k}$, one step per period on the
last resolved mode, run to $T=16\pi/\omega_1$.

Spatial outlier removal projects onto the resolved eigenspace and
advances with the three-stage Gauss--Legendre method, whose
amplification factor~\pade{3}{3} is the same conservative parent as
MIISO5 and MIISO6.
MIISO2 through MIISO6 run on the unmodified discretization with fixed
cutoff $\omega_c=h\,\omega_{N-k+1}$, against generalized-$\alpha$ at
$\rho_\infty\in\{0,0.5,1\}$.
The raw exponential Rosenbrock scheme exprb65 is the same stage
formula as MIISO6 with $R=e^{z}$ used for all frequencies.
Table~\ref{tab:energy-results} reports $E_{\mathrm{res}}(q(T))$ and
$E_{\mathrm{out}}(q(T))$.

\begin{table}[H]
\centering
\caption{Resolved and outlier band energies~\eqref{eq:energies} at
$t=T$ for~\eqref{eq:membrane-linear} with~\eqref{eq:membrane-ic}.}
\label{tab:energy-results}
\begin{tabular}{@{}lcc@{}}
\toprule
Method & $E_{\mathrm{res}}(q(T))$ & $E_{\mathrm{out}}(q(T))$ \\
\midrule
spatial projection + Gauss--Legendre (3-stage) & $1.000\times 10^{0}$ & $1.235\times 10^{-15}$ \\
MIISO2 & $7.156\times 10^{-1}$ & $1.979\times 10^{-15}$ \\
MIISO3 & $9.997\times 10^{-1}$ & $3.725\times 10^{-15}$ \\
MIISO4 & $9.997\times 10^{-1}$ & $3.523\times 10^{-15}$ \\
MIISO5 & $1.000\times 10^{0}$ & $6.869\times 10^{-15}$ \\
MIISO6 & $1.000\times 10^{0}$ & $4.136\times 10^{-15}$ \\
exprb65 & $1.000\times 10^{0}$ & $3.423\times 10^{-2}$ \\
gen-$\alpha$ ($\rho_\infty=0$) & $1.366\times 10^{-2}$ & $1.944\times 10^{-17}$ \\
gen-$\alpha$ ($\rho_\infty=0.5$) & $7.206\times 10^{-1}$ & $5.310\times 10^{-11}$ \\
gen-$\alpha$ ($\rho_\infty=1$) & $7.156\times 10^{-1}$ & $3.744\times 10^{-2}$ \\
\bottomrule
\end{tabular}
\end{table}

Spatial outlier removal drives~$E_{\mathrm{out}}$ to machine roundoff
and leaves~$E_{\mathrm{res}}$ at its initial value by construction.
MIISO5 and MIISO6 match both energy quantities on the full unmodified
discretization, sharing the same Gauss parent as the baseline.
MIISO3 and MIISO4 likewise reduce the outlier energy to zero and keep
$E_{\mathrm{res}}$ near its initial value.
MIISO2 also drives the outlier band to zero, with its lower
$E_{\mathrm{res}}$ reflecting the lower-order Gauss~\pade{1}{1} parent at
this step size.
The raw exprb65 scheme preserves~$E_{\mathrm{res}}$ but leaves the
outlier energy near its initial value, as expected with no selective
damping.
Generalized-$\alpha$ at $\rho_\infty=0$ and~$0.5$ clears the outliers
while also damping the resolved band, and at $\rho_\infty=1$ the
outlier energy $3.744\times 10^{-2}$ stays near its initial value.
MIISO therefore matches the outlier annihilation of spatial outlier-subspace removal
without projecting out any degrees of freedom.

\subsection{Convergence}
\label{sec:results-convergence}

This section and the following section use the nonlinear isogeometric discretization detailed below.
We integrate a square membrane with cubic-hardening strain energy,
\begin{equation}
W(u)=\int_\Omega\Bigl[\tfrac12|\nabla u|^2+\tfrac14\alpha|\nabla u|^4\Bigr]\,d\Omega,
\qquad \alpha=1,
\label{eq:membrane-nonlinear}
\end{equation}
clamped on~$\partial\Omega$. We discretize this problem with a degree-$p=4$
B-spline basis on a $6\times6$ mesh, giving $N=64$ degrees of freedom and outlier count $k=28$.
The membrane is released from rest at peak amplitude~$A=0.2$ and
advanced to $T=1$.
The cutoff is placed above the full spectrum so~$\theta\equiv0$
throughout to turn selective damping off, isolating temporal order from the outlier mechanism.
Step sizes are $h\in\{0.05,0.025,0.0125,0.00625\}$, and error is
measured in relative~$L^2(\Omega)$ norm against an eighth-order
Runge--Kutta reference~\citep{verner2010} at $\Delta t=5\times10^{-5}$.

Figure~\ref{fig:convergence} shows the resulting error against step
size for MIISO2 through MIISO6, each matching its predicted
slope~$N$.
The full table of step sizes, errors, and observed orders is given in
Appendix~\ref{app:convergence-table}, Table~\ref{tab:convergence}.

\begin{figure}[H]
\centering
\includegraphics[width=0.85\textwidth]{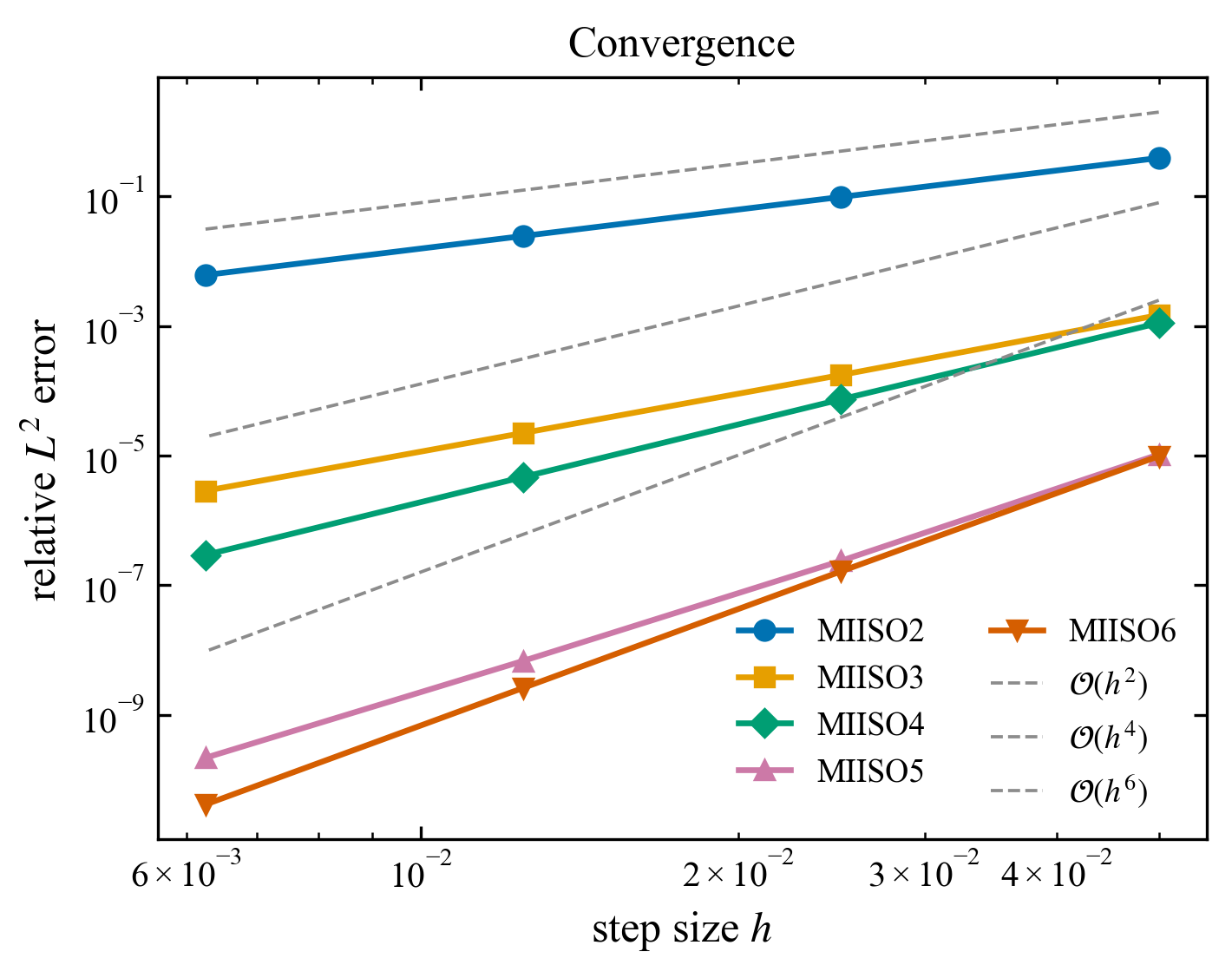}
\caption{Temporal convergence on the nonlinear
membrane~\eqref{eq:membrane-nonlinear}.
Dashed lines show $\mathcal{O}(h^2)$, $\mathcal{O}(h^4)$, and
$\mathcal{O}(h^6)$ for reference.}
\label{fig:convergence}
\end{figure}

\subsection{Cost}
\label{sec:results-cost}
\label{sec:results-scaling}

We return to the nonlinear membrane~\eqref{eq:membrane-nonlinear} of
Section~\ref{sec:results-convergence} and compare wall time against
generalized-$\alpha$ with $\rho_\infty=0.5$.
In both runs below, the cutoff is placed by
Algorithm~\ref{alg:cutoff-setup} and advanced by
Algorithm~\ref{alg:cutoff-track}.

On a finer mesh ($p=4$, $20\times20$ elements, $N=484$, $k=84$), we
compare wall time at matched relative~$L^2(\Omega)$ error, using an
extension of the step-size ladder and the same reference solver as
Section~\ref{sec:results-convergence}.
At each target tolerance, the fastest step size reaching that error
for each method is recorded.
Figure~\ref{fig:cost-accuracy} shows wall time against error for each
method.
At a tolerance of~$10^{-2}$, the cheapest MIISO run is roughly
$8\times$ faster than generalized-$\alpha$, and at~$10^{-3}$ this grows
to roughly~$30\times$.
Generalized-$\alpha$ does not reach a relative error below
$10^{-4}$ within the tested step-size range, while every member of the
MIISO family does.
The full table of step size, wall time, and error is given in
Appendix~\ref{app:cost-table}, Table~\ref{tab:cost-raw}.
The cheapest wall time at each tolerance is given in
Table~\ref{tab:cost-accuracy}.

\begin{figure}[H]
\centering
\includegraphics[width=0.85\textwidth]{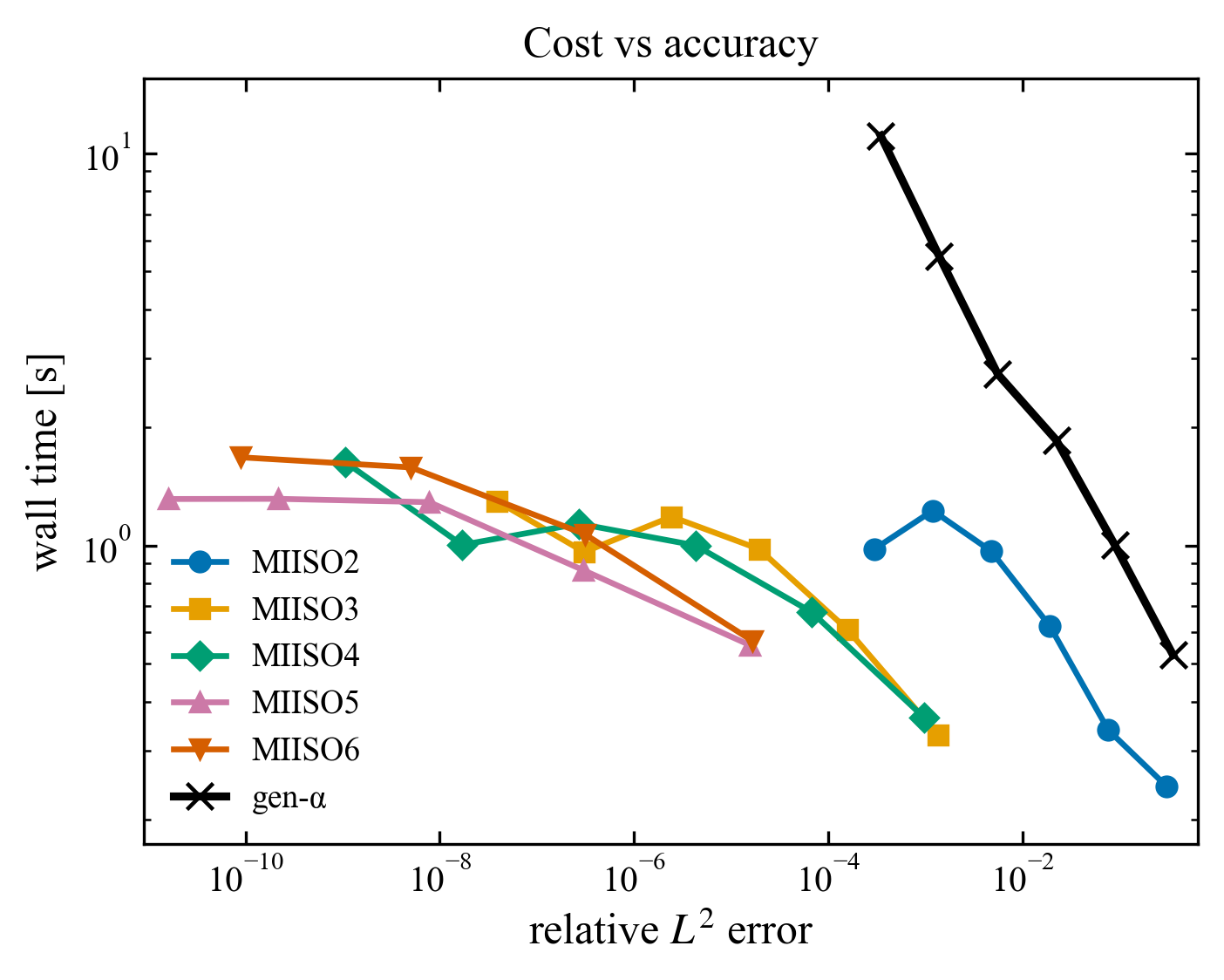}
\caption{Wall time against relative $L^2(\Omega)$ error.}
\label{fig:cost-accuracy}
\end{figure}

We next fix the step size at~$h=0.00625$, a value at which every method
in Table~\ref{tab:cost-accuracy} reaches relative error below
$10^{-2}$. We refine the mesh, beginning with 6 and ending with 96 elements
per side, giving $N=64$ through~$9604$ and $k=28$ through~$388$. We
record wall times and relative~$L^2(\Omega)$ errors against the same
eighth-order Runge--Kutta reference solver as in
Section~\ref{sec:results-convergence} at each mesh size.
Generalized-$\alpha$ is run at~$\rho_\infty=0.5$.
The errors stay comparable to those reported for this step size in the
cost-accuracy study above.
Figure~\ref{fig:scaling} shows wall time against~$N$ for each method.
At the smallest mesh ($N=64$), generalized-$\alpha$ is about $1.5\times$
slower than MIISO2, while MIISO3 through MIISO6 cost more than
generalized-$\alpha$ at this size.
By the largest mesh ($N=9604$), every MIISO member is cheaper than
generalized-$\alpha$.
MIISO2 is about $7.8\times$ faster, and MIISO6, the costliest
member of the family, is about $3.5\times$ faster.
The full table of wall times and relative errors by mesh size is given
in Appendix~\ref{app:scaling-table}, Table~\ref{tab:scaling}.

\begin{figure}[H]
\centering
\includegraphics[width=0.85\textwidth]{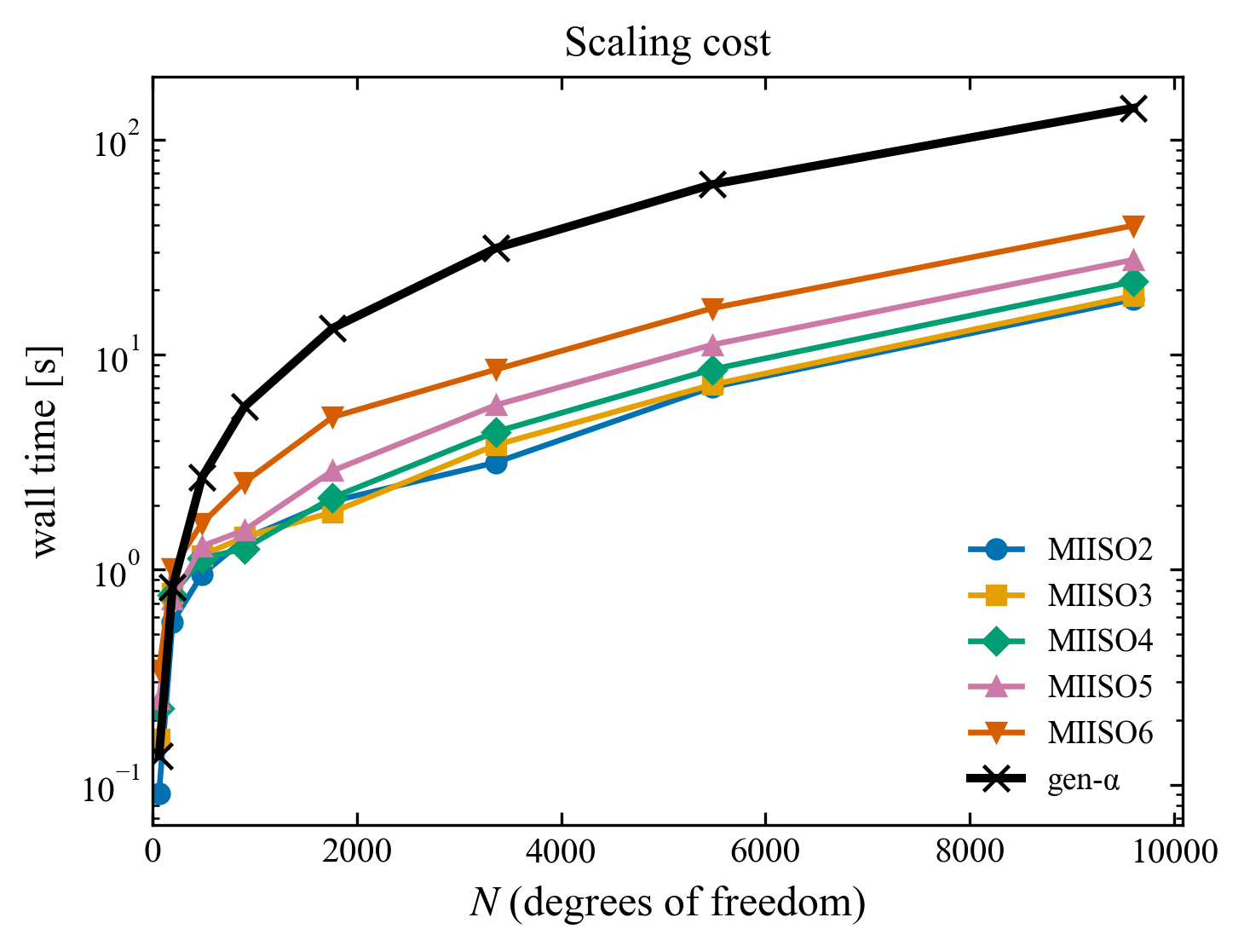}
\caption{Wall time against problem size~$N$.}
\label{fig:scaling}
\end{figure}

\begin{note}[Experimental setup]
\label{note:experimental-setup}
The generalized-$\alpha$ runs use the classical Chung--Hulbert
parameters~\citep{chunghulbert1993} at the stated~$\rho_\infty$.
We use $\rho_\infty=0.5$ as a standard level of damping.
Higher and lower values do not change the performance of the method
significantly, so the middle of the range is kept.
Each step is solved by full Newton iteration. The sparse effective
stiffness is refactored at every iteration by Julia's sparse
Cholesky~\citep{bezanson2017} factorization, which calls SuiteSparse
CHOLMOD~\citep{chen2008cholmod}, and the iteration stops when both the
residual and the increment fall below a relative tolerance of~$10^{-8}$.
The generalized-$\alpha$ code against which we time MIISO is a standard,
optimized implementation of the method.

All MIISO and generalized-$\alpha$ runs in this section are sparse.
The spatial operators are assembled with
{Ferrite.jl}~\citep{ferrite2024} and its isogeometric extension
{FerriteIGA.jl}.

Both Algorithm~\ref{alg:cutoff-setup} and
Algorithm~\ref{alg:cutoff-track} are active in these runs, so the wall times
already include the one-time placement and the per-step update.
That makes the reported times a conservative lower bound on the
advantage of MIISO over a fixed-cutoff run.
Even with this additional cost, the figures and tables above show MIISO
performing significantly faster than generalized-$\alpha$.
\end{note}

\paragraph{Choosing a MIISO order}
\label{par:choosing-order}
Tables~\ref{tab:cost-raw}, \ref{tab:cost-accuracy}, and~\ref{tab:scaling}
together suggest a
practical default.
At fixed step size, MIISO2 is the cheapest member of the family at
every mesh size tested (Table~\ref{tab:scaling}), though it cannot
reach relative errors below roughly~$10^{-3}$ within the tested
step-size range (Tables~\ref{tab:cost-raw} and~\ref{tab:cost-accuracy}), so it is suited to
applications where accuracy demands are modest. MIISO3
is the next cheapest member and gives higher order at a nearly equal cost.
MIISO4 gives the best combination of per-step cost and accuracy across
the tolerances tested, reaching errors several orders of magnitude
tighter than MIISO2 at a comparable wall time, and is the recommended
default.
MIISO5 becomes the cheapest member at tight tolerances, $10^{-8}$ and
below, and scales comparably to MIISO4, making it the better choice
when high accuracy is specifically required.
MIISO6 is consistently the most expensive member of the family at every
mesh size, and its accuracy gain over MIISO5 is marginal at the
tolerances tested.
It is retained for completeness and for cases that specifically call
for sixth-order accuracy.

\subsection{Outlier stabilization}
\label{sec:results-outlier-stab}

We keep the same nonlinear membrane~\eqref{eq:membrane-nonlinear} of
Section~\ref{sec:results-convergence} and change the initial condition to
corner-localized waves of amplitude~$A=0.2$ at all four corners, released
from rest.
Recalling the band energies~\eqref{eq:energies} of
Section~\ref{sec:results-energy}, about~$9.6\%$ of $\|q(0)\|_M$ lies in
the outlier band at the initial time.
With the same $M$-orthonormal eigenpairs and $c_n=\varphi_n^\top Mq$, write
$u_{\mathrm{out}}=\sum_{n=N-k+1}^{N}c_n\,\varphi_n$
for the reconstruction of the displacement in the outlier eigenspace.
Its $M$-norm is the scalar $E_{\mathrm{out}}(q)$ already defined
in~\eqref{eq:energies} of Section~\ref{sec:results-energy}.
The spatial discretization is degree~$p=4$ on a $12\times12$ mesh,
giving $N=196$ degrees of freedom and outlier count~$k=52$.
The step size is $h=2\pi/\omega_{N-k}$, one step per period of the last
resolved mode as recommended in Section~\ref{sec:discuss-withdrawal}, and
the run extends to $T=16h$.
The cutoff is placed as in Section~\ref{sec:cutoff}
(Algorithm~\ref{alg:cutoff-setup}) and advanced as in
Section~\ref{sec:cutoff-tracking} (Algorithm~\ref{alg:cutoff-track}).

We advance with MIISO4 and with exprb43.
The latter uses the same stages as MIISO4 with the exact exponential
$R=e^{z}$ at every frequency, as in Section~\ref{sec:exprb}.
At~$t\approx1.41$ the MIISO4 displacement remains bounded and physical,
while exprb43 produces a numerical instability.
This instability arises from undamped high-frequency outliers interacting
with a step size that is not fine enough to naturally resolve the outlier band.
Figure~\ref{fig:corner-u} shows the displacement field of both methods
at that time.
Figure~\ref{fig:corner-uout} shows $u_{\mathrm{out}}$ for both methods
at the same time.

At the time of the instability, $E_{\mathrm{out}}$ for MIISO4 is
$\mathcal{O}(10^{-4})$, while that for exprb43 is $\mathcal{O}(10^{0})$.
MIISO damps the outlier band and keeps the solution bounded.
exprb43 leaves the outliers undamped, and the step produces a
numerical blow-up.
This comparison demonstrates the outlier stabilization mechanism of
MIISO on a two-dimensional nonlinear example.
The same conclusion extends to generalized-$\alpha$.
Proposition~\ref{prop:global-dissipation} shows that
generalized-$\alpha$ at $\rho_\infty=1$ applies no dissipation to any
frequency, so it would produce the same kind of unphysical outlier
blow-up as exprb43 on this problem, for the same reason described above.
Smaller values of~$\rho_\infty$ avoid this specific failure only by
damping every frequency, including the resolved band, as already shown
in Table~\ref{tab:energy-results} of Section~\ref{sec:results-energy}.
We do not run generalized-$\alpha$ on this problem directly, since the
mechanism and its consequence are already established.

\begin{figure}[H]
\centering
\includegraphics[width=0.85\textwidth]{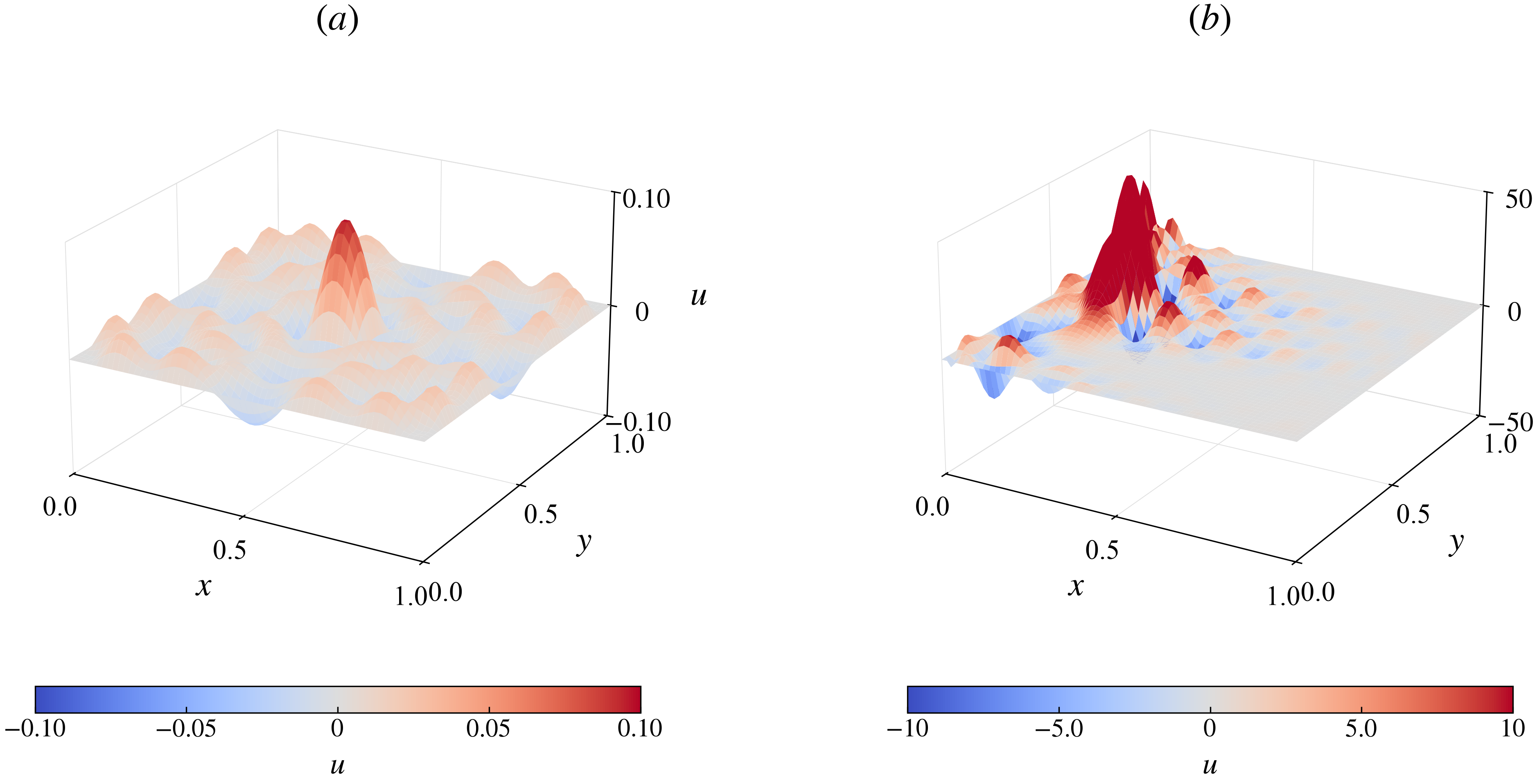}
\caption{Displacement~$u$ of the nonlinear membrane at the exprb43
blow-up time.
Panel~(a) is MIISO4.
Panel~(b) is exprb43; the color scale
is clipped at $|u|=10$ for display, true peak displacement is $|u|\approx 47.5$.}
\label{fig:corner-u}
\end{figure}

\begin{figure}[H]
\centering
\includegraphics[width=0.85\textwidth]{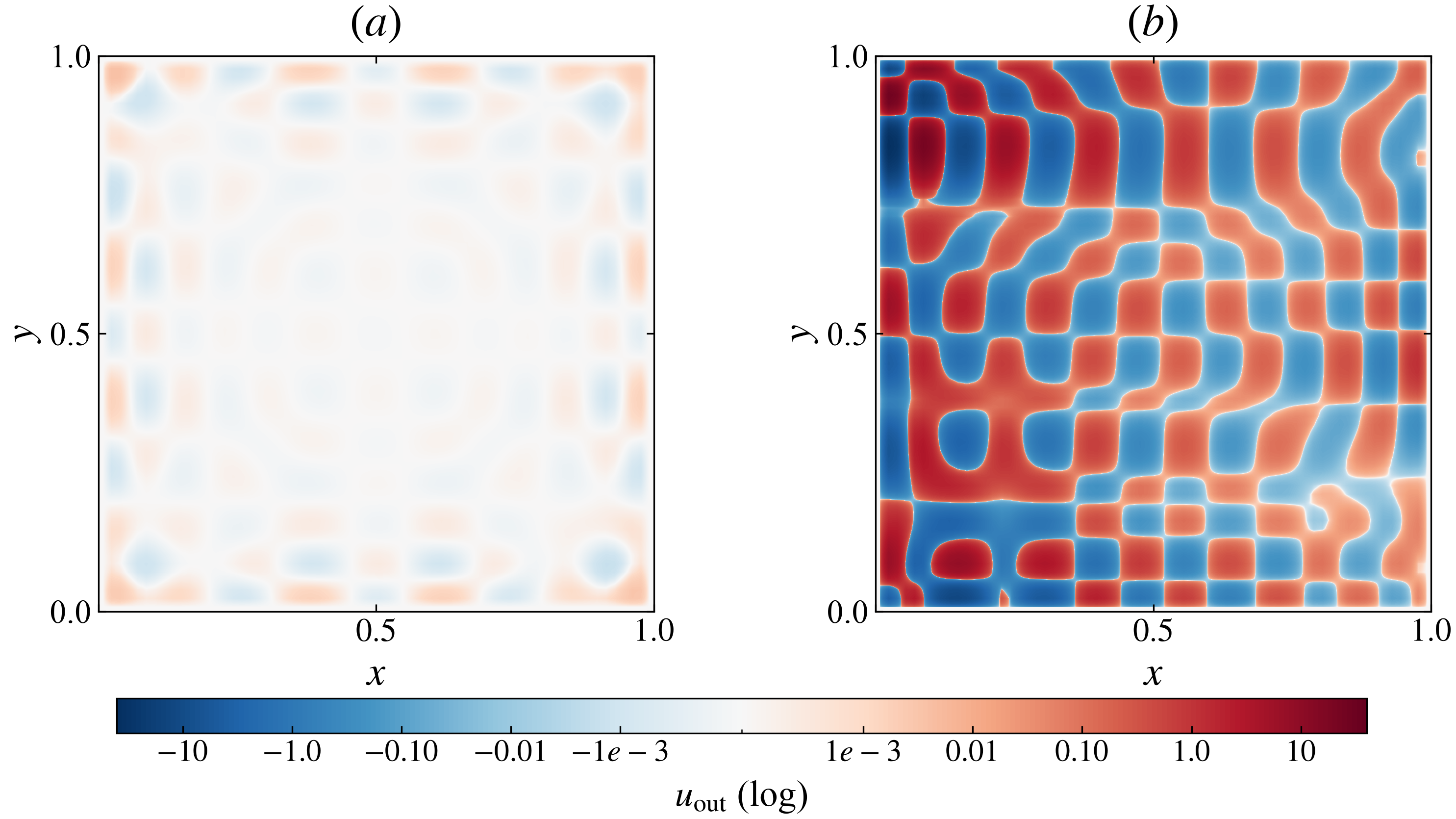}
\caption{Outlier projection~$u_{\mathrm{out}}$ at the same time as
Figure~\ref{fig:corner-u} in logarithmic scale.
Panel~(a) is MIISO4.
Panel~(b) is exprb43.}
\label{fig:corner-uout}
\end{figure}

\subsection{Continuum earthquake simulation}
\label{sec:results-earthquake}

The previous sections have established that MIISO selectively damps
and stabilizes the outlier band, which existing methods leave undamped
or only damp through global dissipation, and that MIISO gives a cost
advantage.
To demonstrate the method applied to a practical problem,
we now run MIISO4 on a three-dimensional continuum arch under gravity
and synthetic ground motion.
The arch has a rectangular cross-section of span~$L=120\,\mathrm{m}$,
rise~$f=24\,\mathrm{m}$, thickness~$d=2.5\,\mathrm{m}$, and
width~$b=10\,\mathrm{m}$, discretized with a degree-$p=4$
NURBS mesh of $32\times6\times6$ elements. 
The mesh is built with Ferrite.jl and FerriteIGA.jl~\citep{ferrite2024}.
The degree of freedom vector has length~$N=10200$.
Figure~\ref{fig:bridge-mesh} shows the undeformed geometry and the axis
convention used throughout this section.

\begin{figure}[H]
\centering
\includegraphics[width=0.60\textwidth]{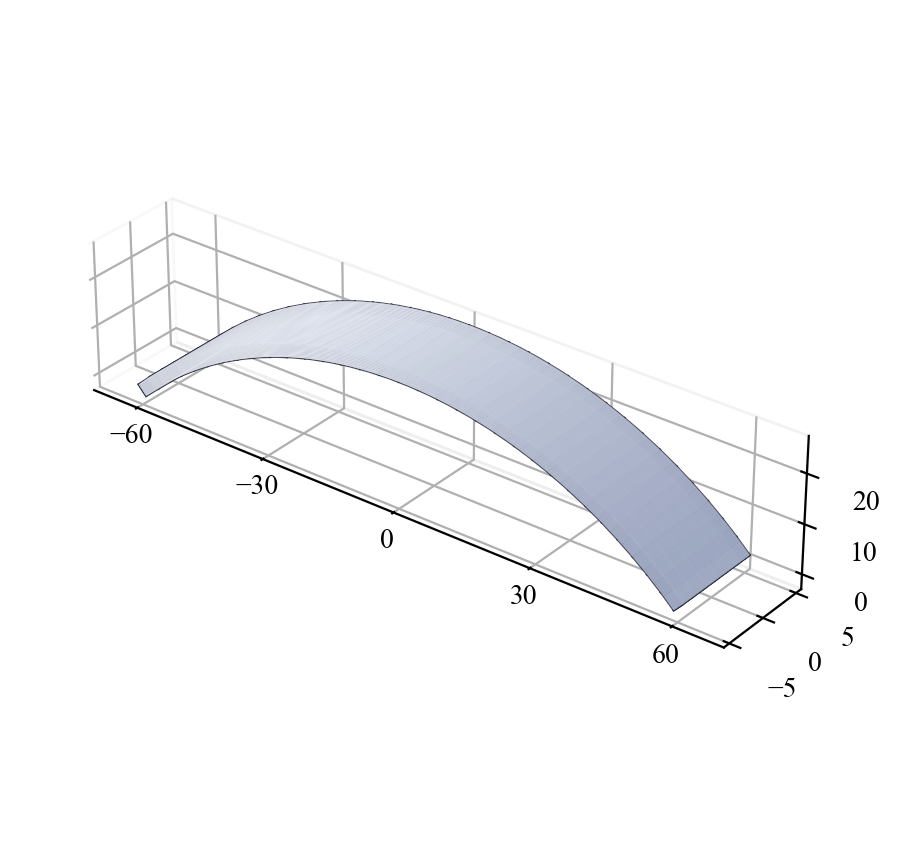}
\caption{Arch geometry with the axis convention used
throughout this section.
}
\label{fig:bridge-mesh}
\end{figure}

The material is Saint-Venant--Kirchhoff hyperelasticity with Young's
modulus~$E=34\times10^{9}\,\mathrm{Pa}$, Poisson ratio~$\nu=0.2$, and
density~$\rho=2500\,\mathrm{kg/m^3}$, together with a distributed deck mass of
$2\times10^{4}\,\mathrm{kg/m}$ per unit span.
Light Rayleigh damping is present at ratio~$\zeta=0.05$ between the
reference frequencies $f_1=1.2\,\mathrm{Hz}$ and
$f_2=25\,\mathrm{Hz}$.
The constitutive map uses the deformation gradient $F=I+\nabla u$ and
the Green--Lagrange strain
$E=\tfrac12(F^{\mathsf T}F-I)$.
The second Piola--Kirchhoff stress is the linear map
$S=\lambda(\mathrm{tr}\,E)\,I+2\mu E$, and the first Piola--Kirchhoff
stress is $P=FS$, with Lam\'e parameters $\lambda$ and~$\mu$ from
$E$ and~$\nu$ above.
Internal forces follow from the standard virtual-work assembly of~$P$.

Gravity is ramped smoothly over the interval $[0,1]\,\mathrm{s}$
and then held.
A synthetic horizontal and vertical ground-acceleration record starts
at $t_{\mathrm{eq}}=1.5\,\mathrm{s}$, scaled to peak ground accelerations
$\tfrac25\,g$ and $\tfrac{4}{15}\,g$, and lasts $10\,\mathrm{s}$ until
$t_{\mathrm{eq}}=11.5\,\mathrm{s}$.
We integrate with MIISO4 to $T=11.5\,\mathrm{s}$ at step size
$h=1\times10^{-3}\,\mathrm{s}$,
with adaptive cutoff tracking on and Krylov dimension capped at~$40$.

Figure~\ref{fig:earthquake} shows the displacement magnitude~$|u|$ on
the undeformed geometry at twelve times during the run.
The first two panels cover the initial state and the settled state before
the record begins.
The remaining panels show peaks of the arch response in the strong-shaking
and the late windows, ending at~$T=11.5\,\mathrm{s}$.
The color scale is shared across panels.
The timeline under the panels marks the twelve sample times over the
full run.
The strain-energy density~$\psi$ at the corresponding
times is shown in Appendix~\ref{app:bridge-psi}.

\begin{figure}[H]
\centering
\includegraphics[width=0.44\textwidth]{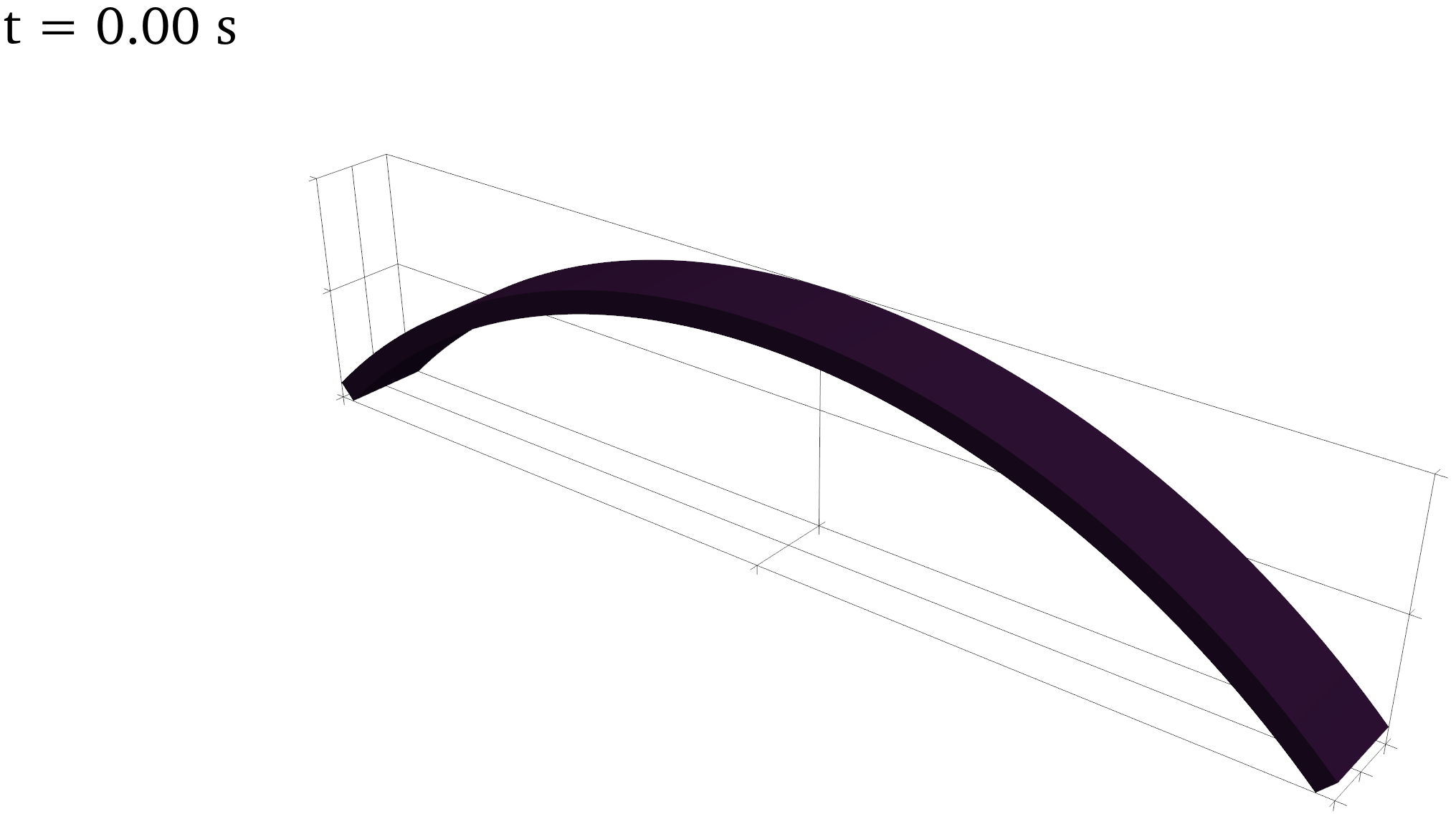}\hfill
\includegraphics[width=0.44\textwidth]{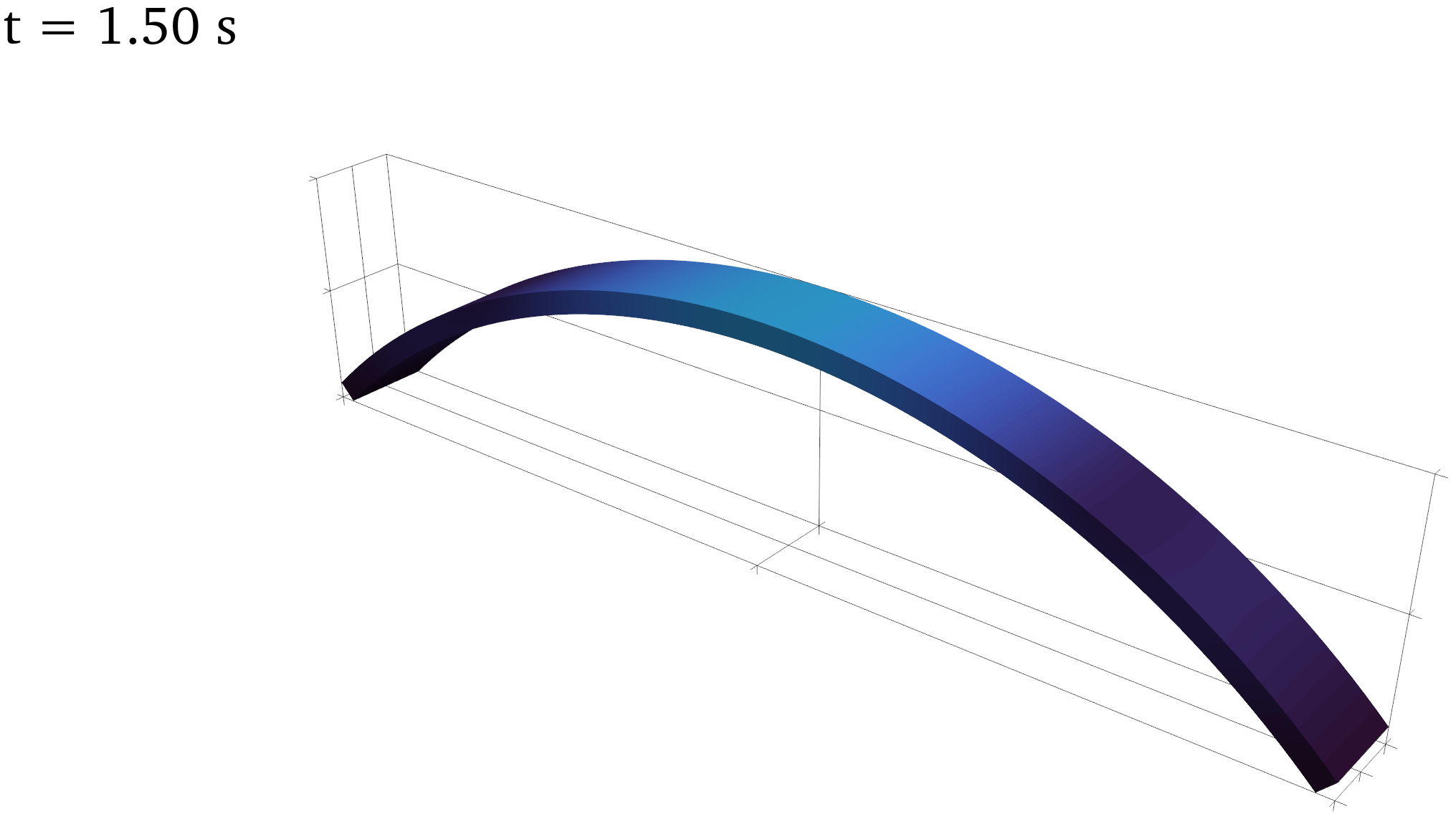}
\end{figure}
\begin{figure}[H]
\centering
\includegraphics[width=0.44\textwidth]{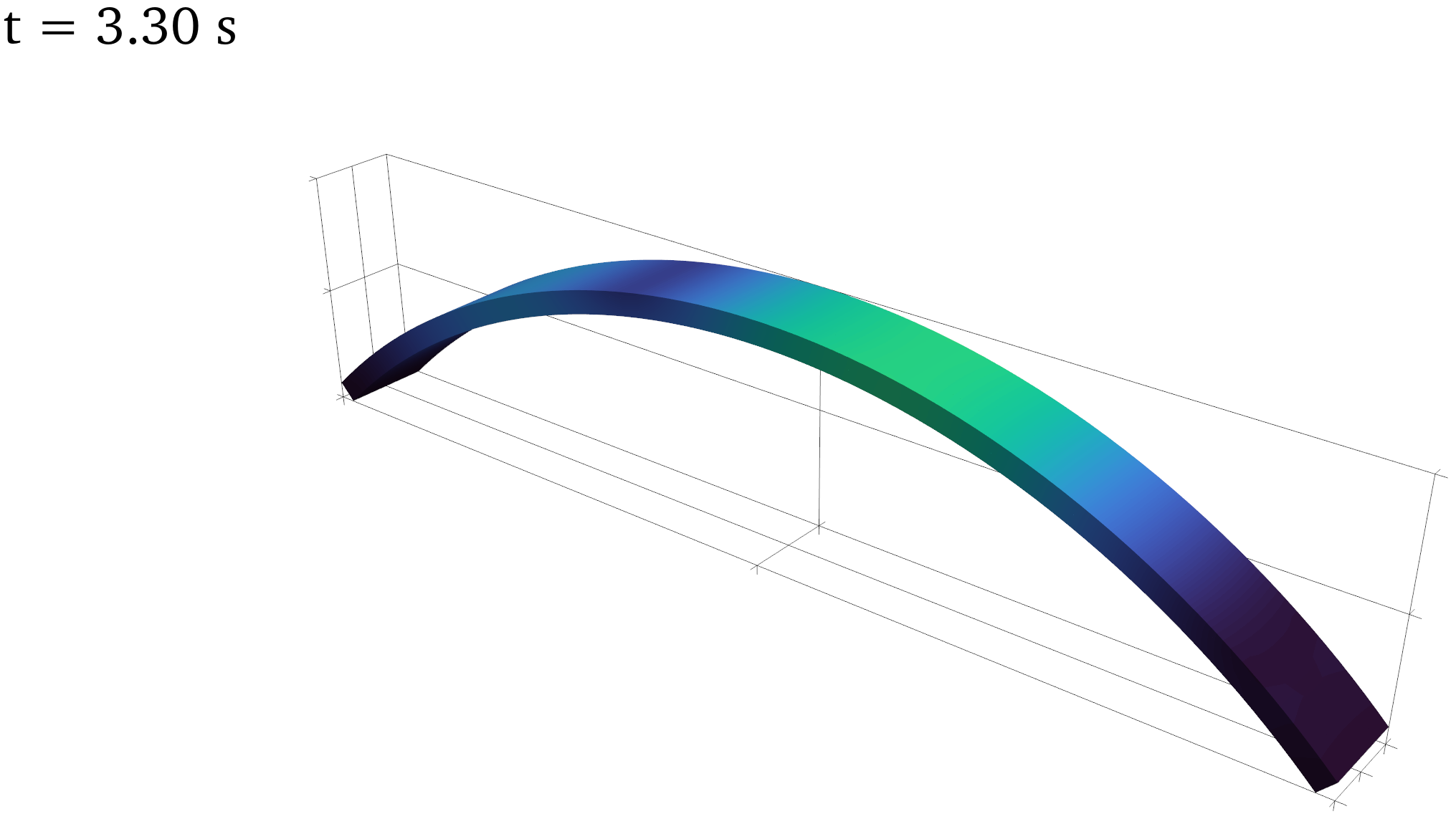}\hfill
\includegraphics[width=0.44\textwidth]{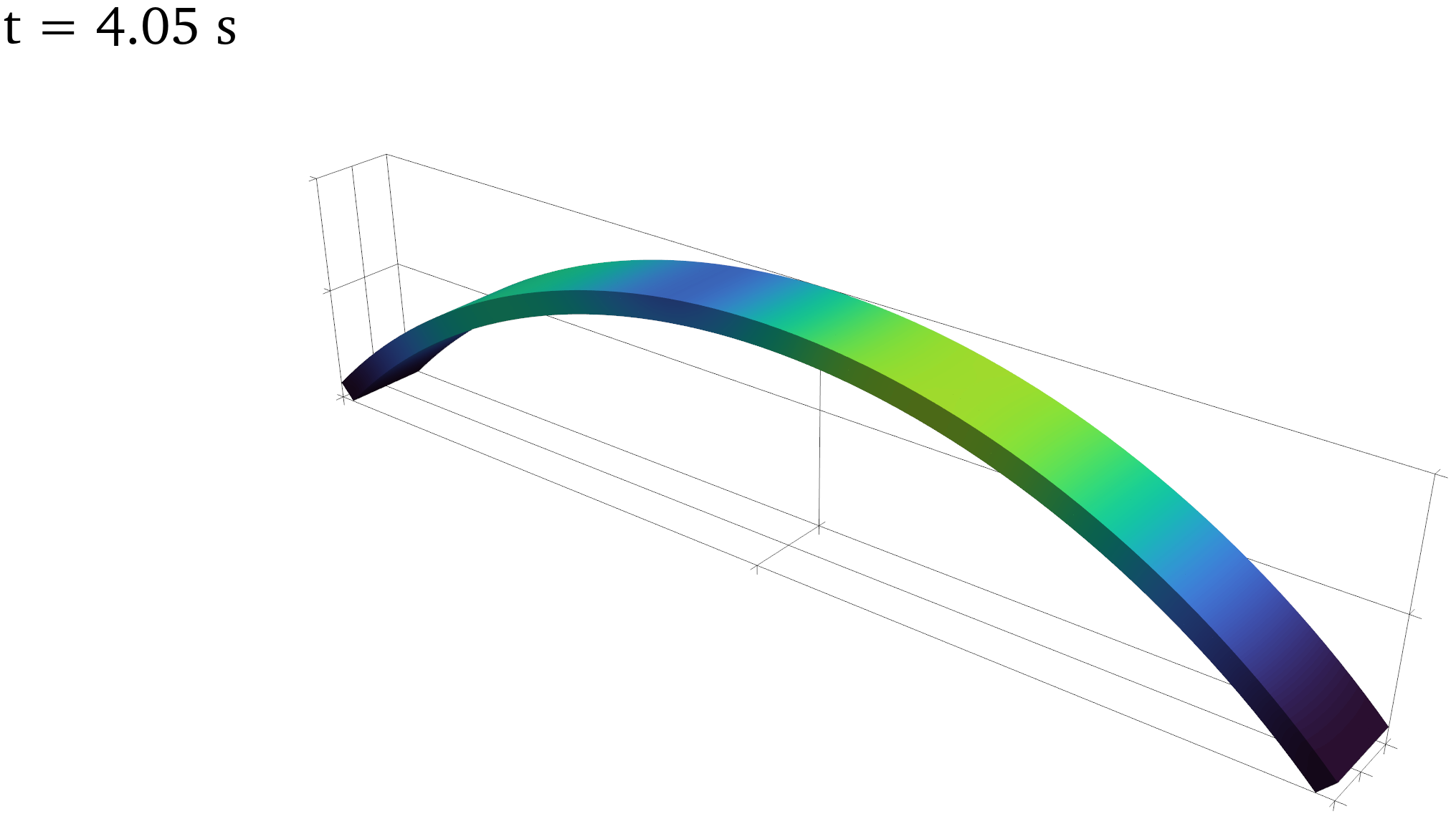}
\end{figure}
\begin{figure}[H]
\centering
\includegraphics[width=0.44\textwidth]{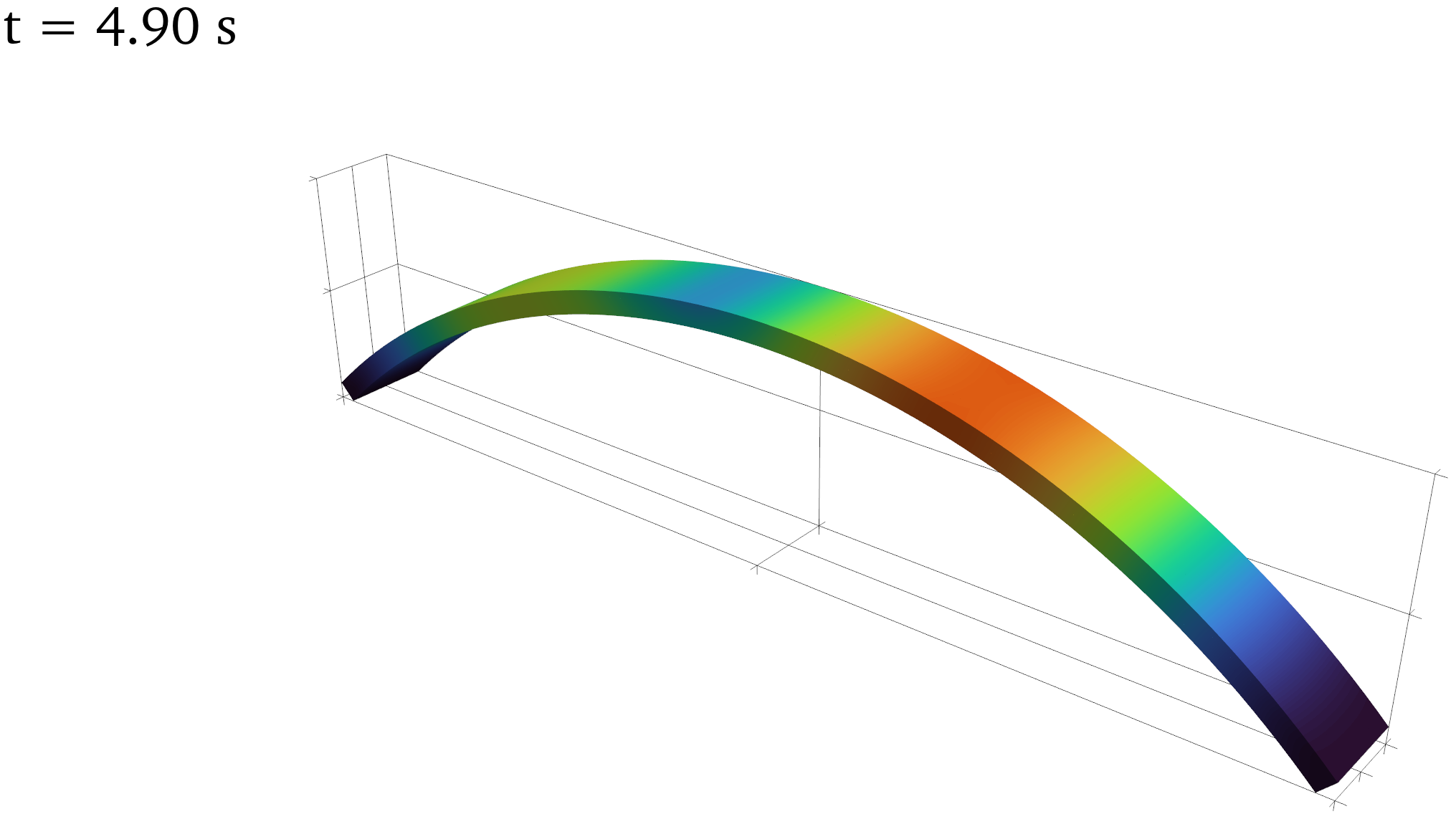}\hfill
\includegraphics[width=0.44\textwidth]{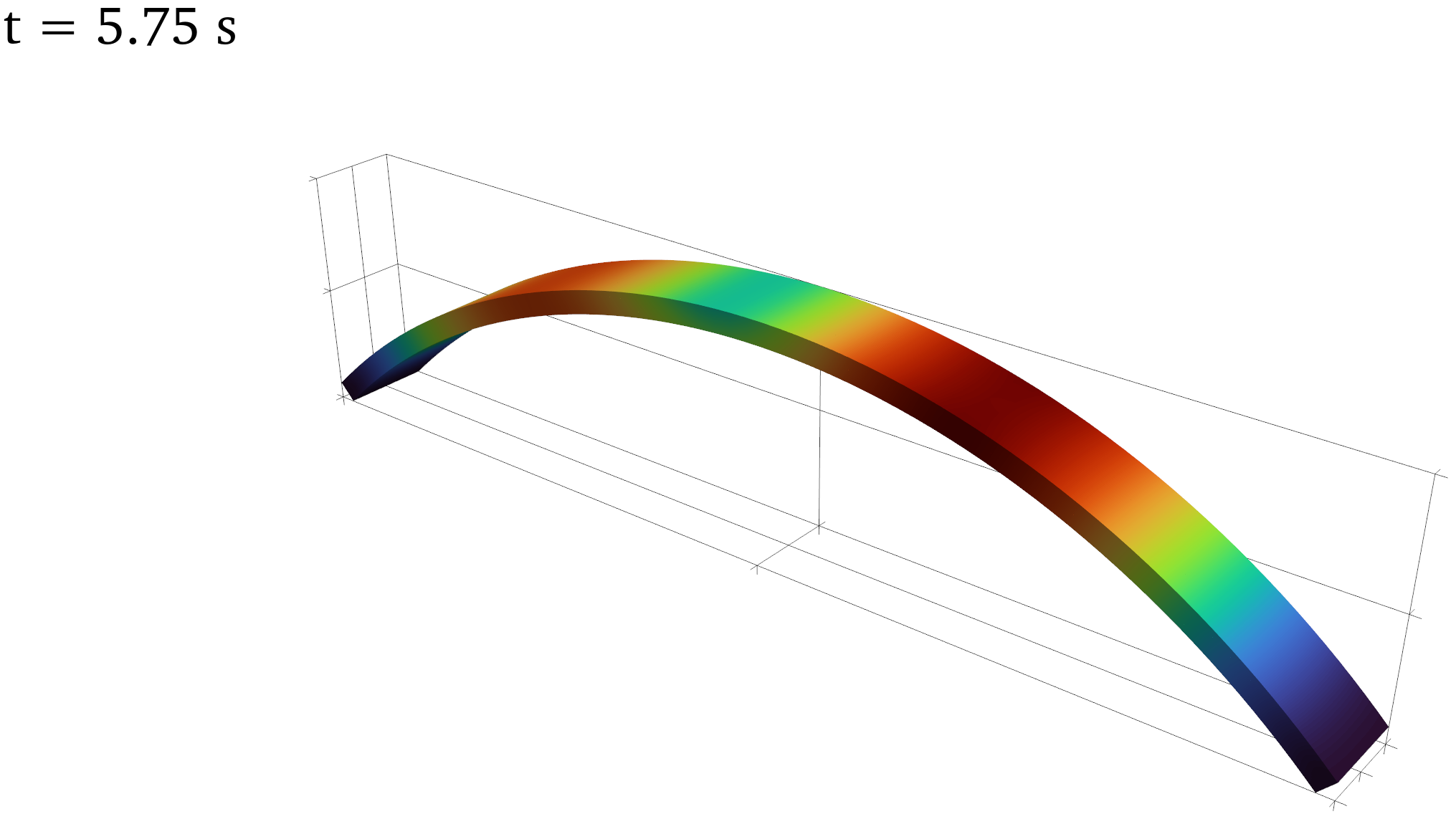}
\end{figure}
\begin{figure}[H]
\centering
\includegraphics[width=0.44\textwidth]{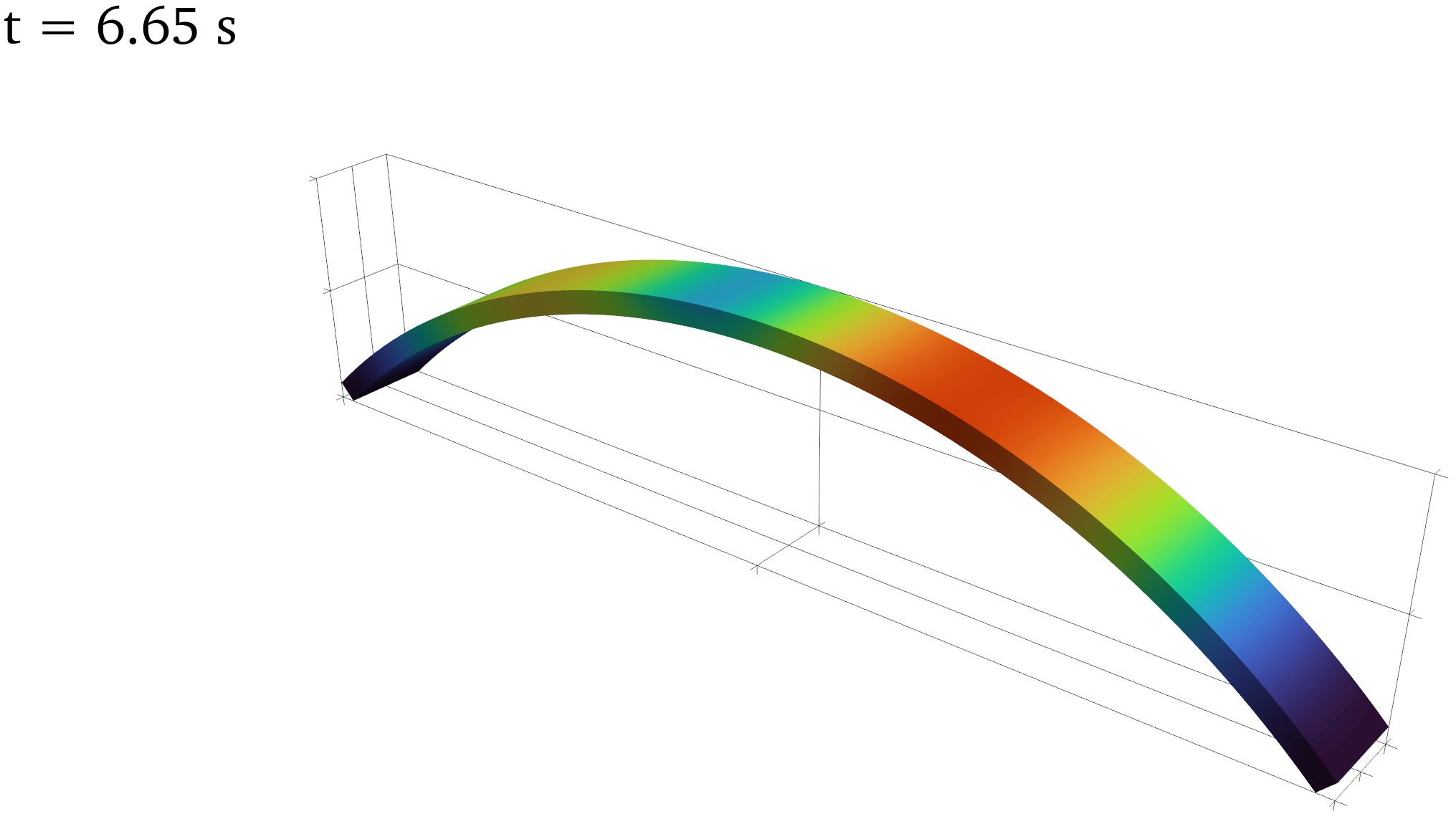}\hfill
\includegraphics[width=0.44\textwidth]{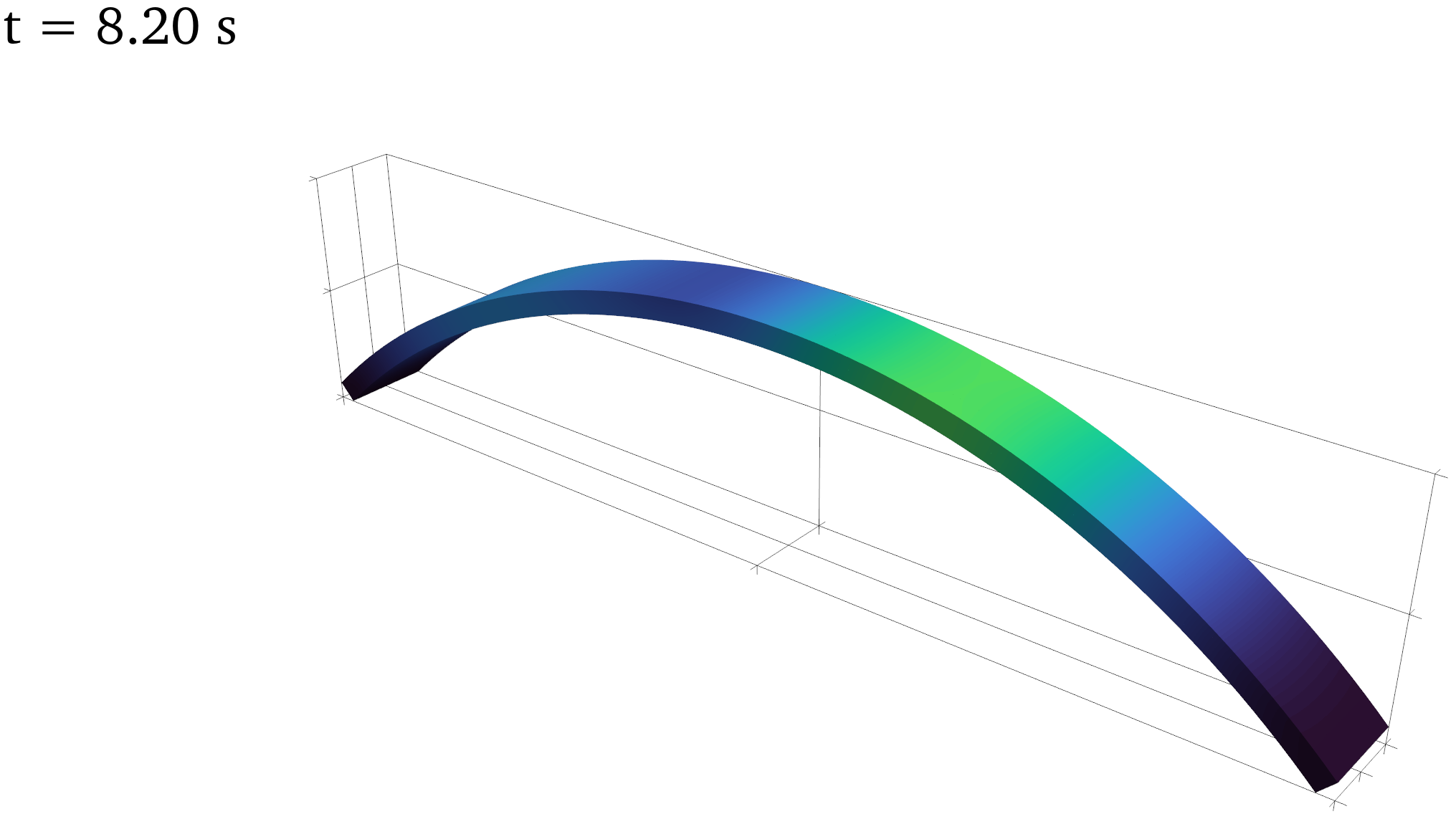}
\end{figure}
\begin{figure}[H]
\centering
\includegraphics[width=0.44\textwidth]{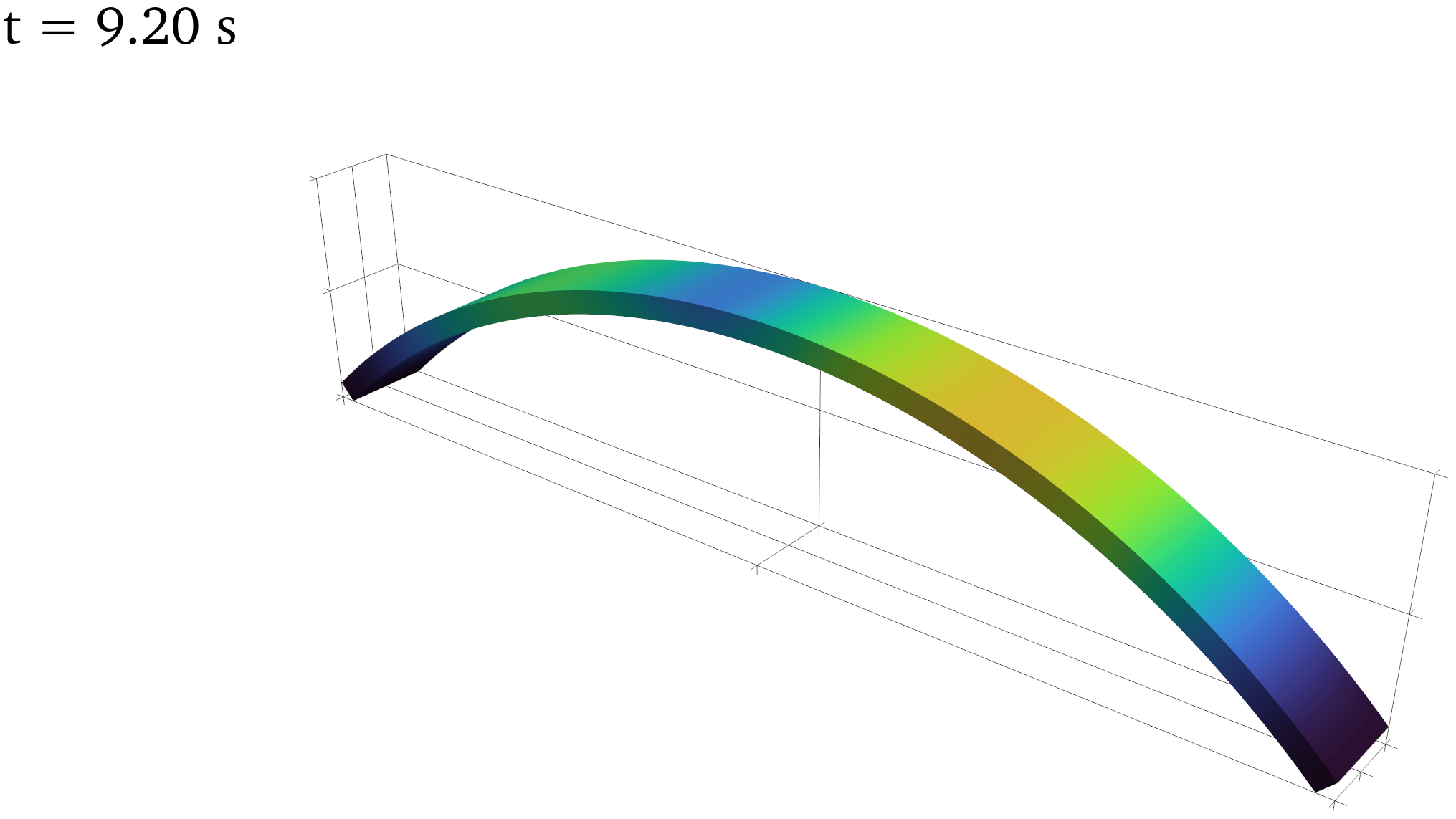}\hfill
\includegraphics[width=0.44\textwidth]{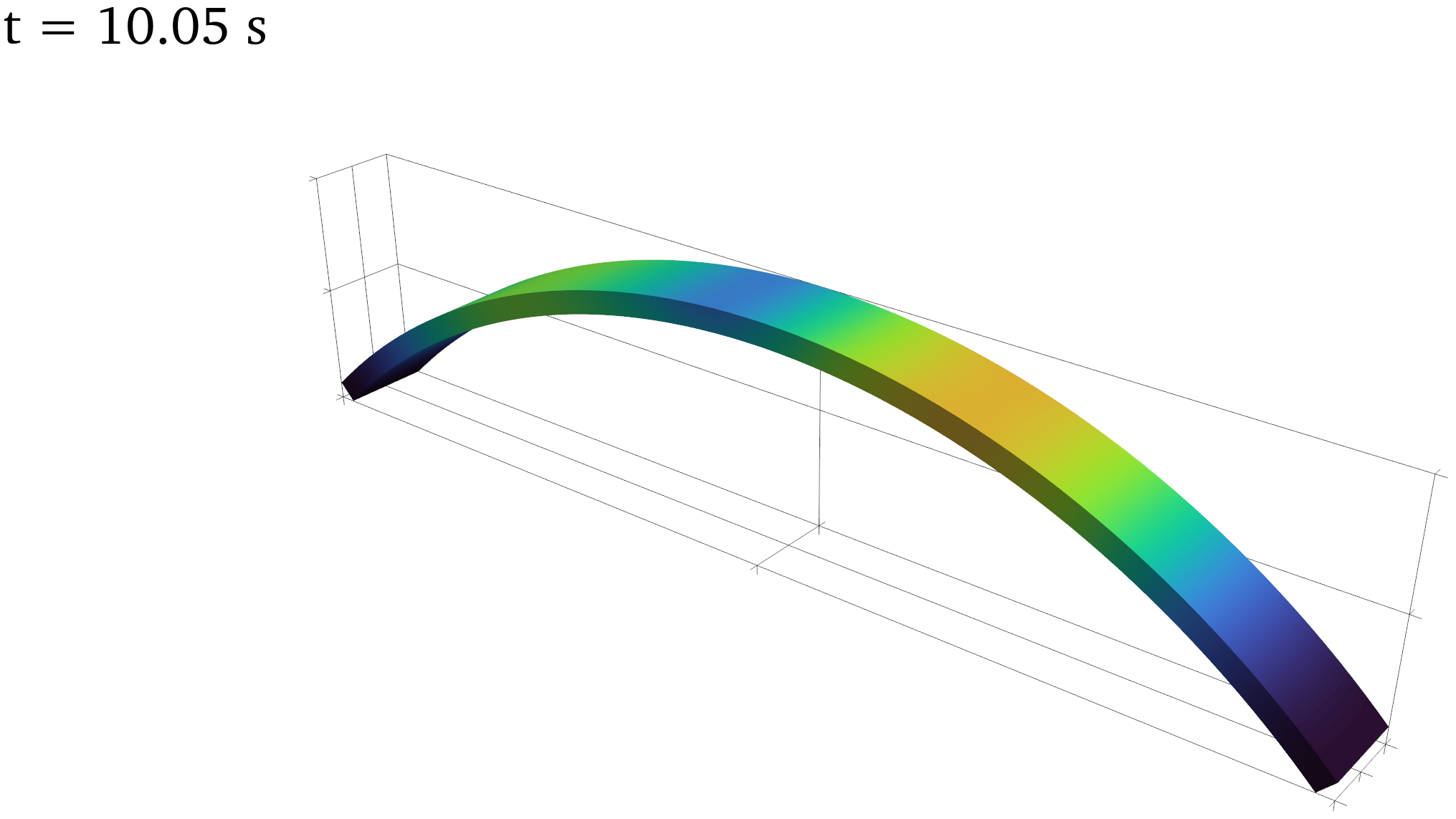}
\end{figure}
\begin{figure}[H]
\centering
\includegraphics[width=0.44\textwidth]{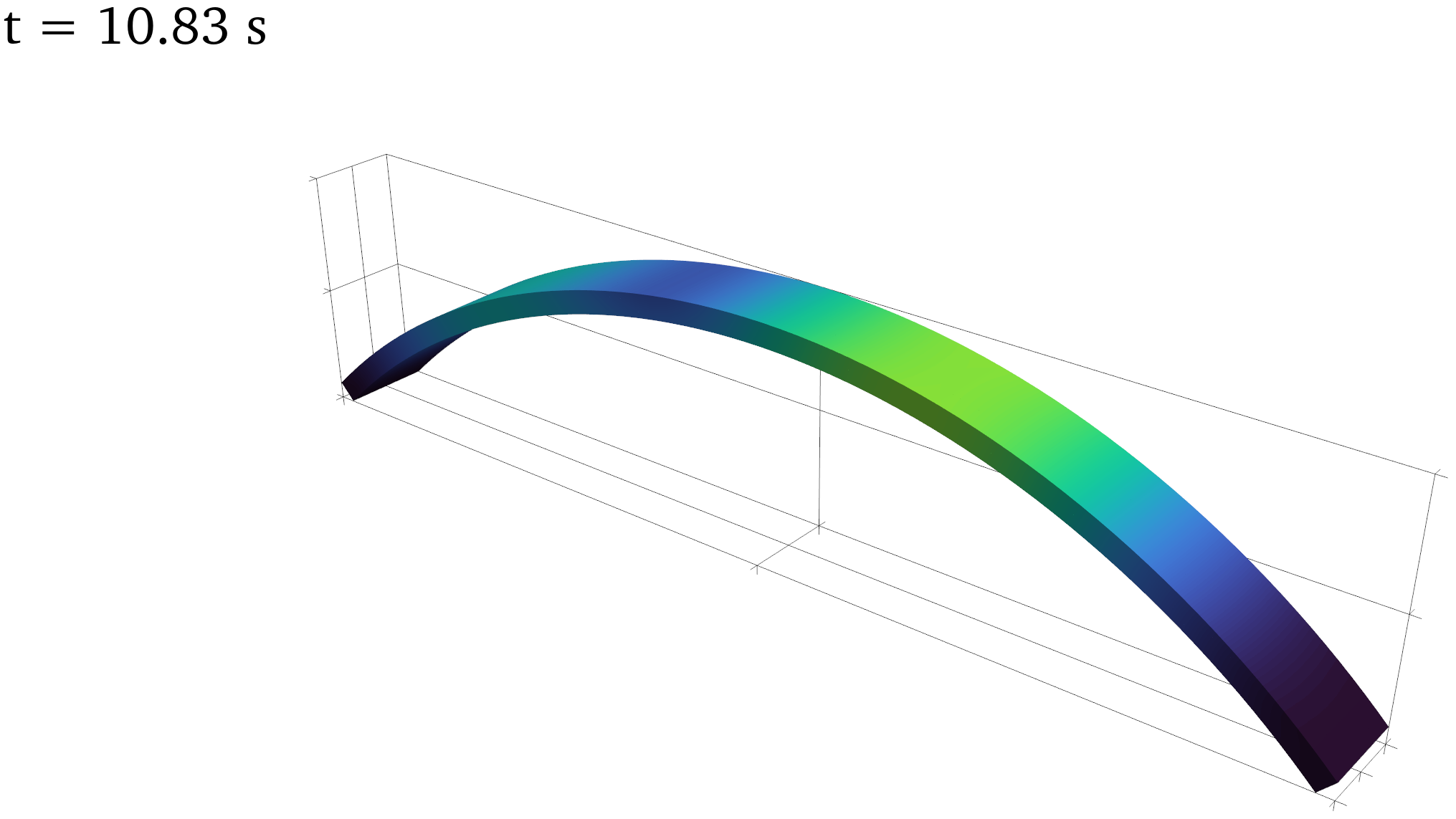}\hfill
\includegraphics[width=0.44\textwidth]{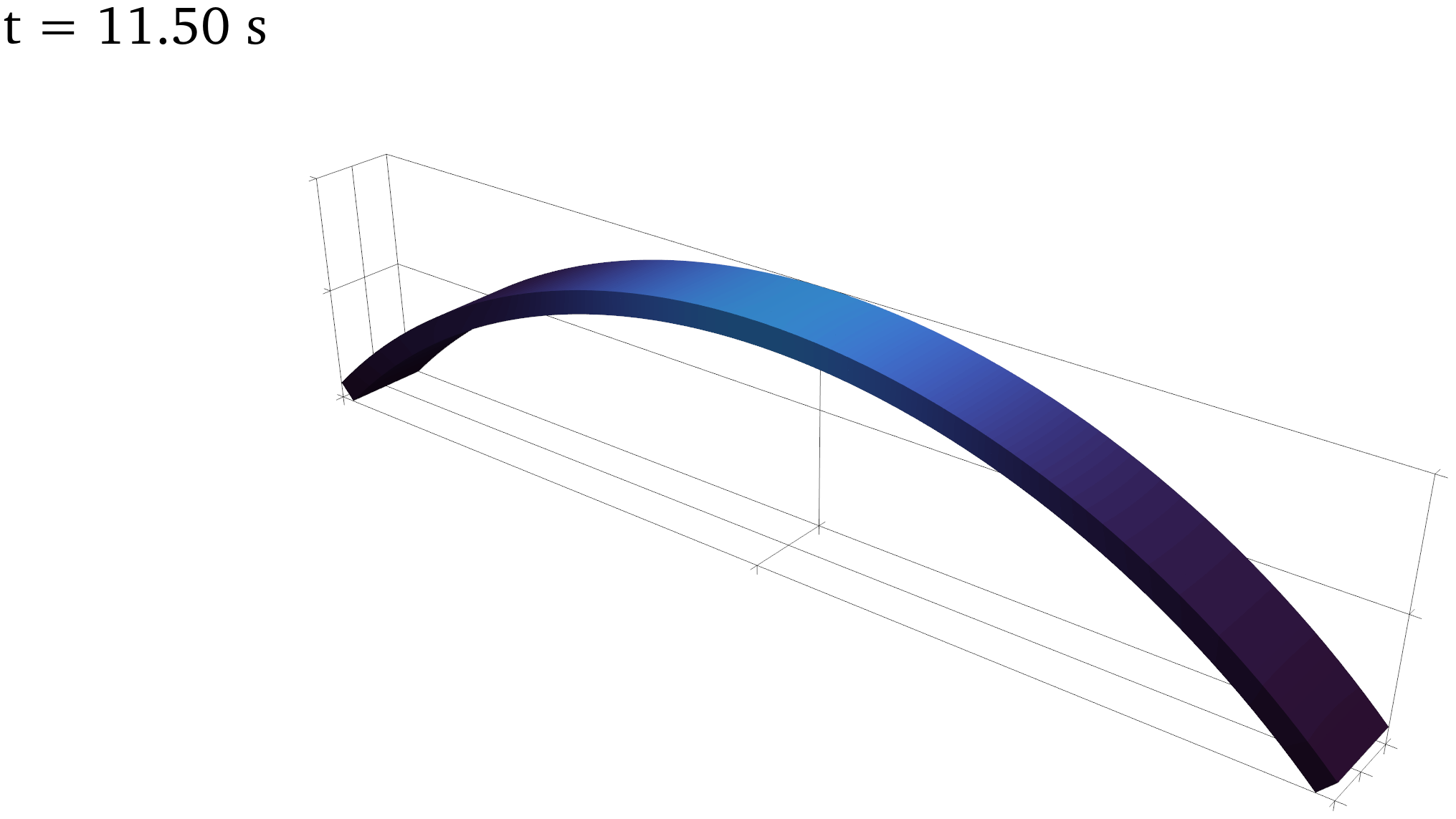}
\\[0.3em]\includegraphics[width=0.80\textwidth]{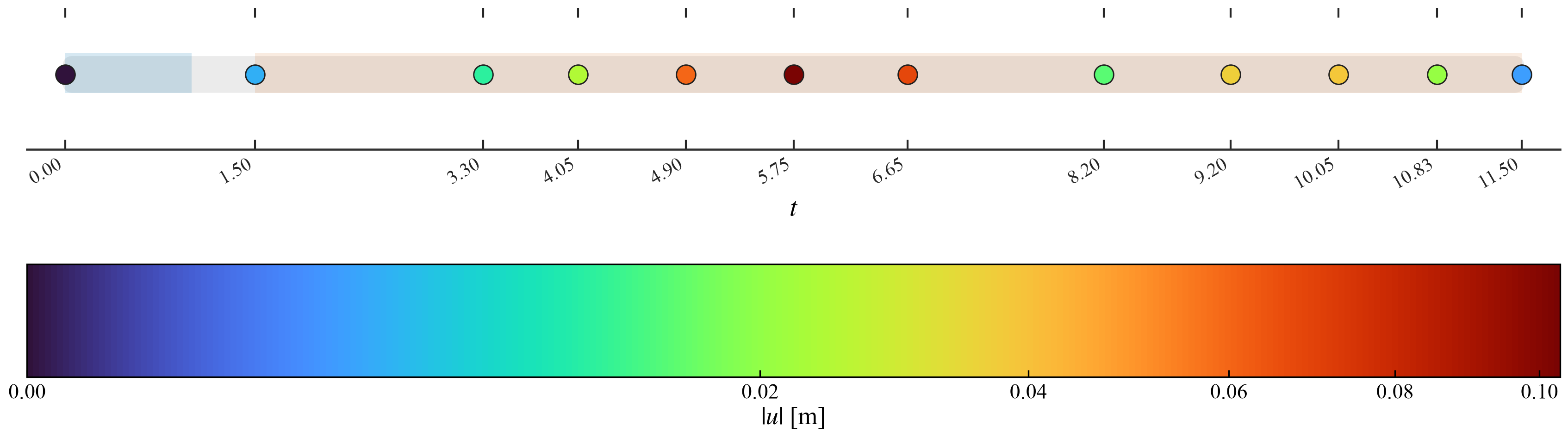}
\caption{Displacement magnitude~$|u|$ on the arch bridge with
MIISO4 at the given times.}
\label{fig:earthquake}
\end{figure}

\begin{note}[Scope of the arch demonstration]
\label{note:arch-scope}
At $N=10200$ on a $32\times6\times6$ mesh, this run is a high-fidelity
continuum model of the arch.
The mesh is fine enough to resolve the spatial variation of~$|u|$ along the span
and across the cross-section during strong shaking.
Section~\ref{sec:results-scaling} already serves the purpose of demonstrating scaling.
This section checks that MIISO runs on a three-dimensional
nonlinear continuum problem with ground motion, hyperelastic material
behavior, and a full 3D continuum displacement field.
\end{note}

\section{Conclusions}
\label{sec:conclusions}

MIISO is a family of exponential Rosenbrock--Krylov integrators that
suppress isogeometric outlier modes in the time integrator and leave
the spatial discretization unchanged.
The construction blends a diagonal Pad\'e approximant with a repaired
subdiagonal Pad\'e approximant at a cutoff frequency placed by the a~priori outlier
count. The resulting amplification map is unconditionally $A$-stable
for all six members of the family, orders two through six. Physical
modes receive no dissipation and outlier modes are driven to zero.
The method also supplies a matrix-free outlier-tracking procedure.
The cutoff is located once by an Arnoldi setup and, for nonlinear and
nonautonomous problems, updated each step by a Rayleigh quotient that
uses one Jacobian--vector product. The accuracy of that update and a
sufficient condition for its validity are established in the appendix.
Numerical checks show that the update remains reliable well past the
certified condition.
Every stage of the method uses the Jacobian only through matrix-vector
products. Independent Krylov evaluations of $\phit_j(hJ)v$ can run
concurrently, so the full algorithm is matrix-free and parallelizable.
Numerical results confirm the predicted convergence orders, the
resolved and outlier mode behavior against spatial subspace removal and
other established methods, a cost advantage over generalized-$\alpha$
at matched accuracy and in scaling, and a practical run on a
three-dimensional continuum earthquake simulation.

\section{Future work}
\label{sec:future}

The same idea, namely identifying a spectral pathology a~priori and modifying
a time integrator's amplification map to damp it selectively, applies
beyond isogeometric analysis. Notably, the non-holomorphic structure of the
amplification map is not specific to isogeometric outliers and suggests a
general framework for mode-selective damping.
A natural next step is to extend it to
other semidiscrete PDE methods that produce analogous outlier eigenvalues.
Interior penalty discontinuous Galerkin methods are a concrete first
case. The penalty parameter is known to introduce a band of
eigenvalues at the top of the discrete spectrum, analogous to the
isogeometric outliers treated here~\citep{chanwarburton2016}.

A second extension concerns the resolved band.
The Gauss parent leaves modes below~$\omega_0$ with exactly zero
dissipation.
However, an ideal setup would apply some amount of dissipation to the upper
end of the resolved band as well.
A parameter we can potentially call~$\rho_c$, named to distinguish
it from the asymptotic spectral radius~$\rho_\infty$ of
generalized-$\alpha$, would set a short interval below~$\omega_c$ over
which a small, smoothly vanishing amount of damping is blended into the
Gauss parent.
This would give an extension of MIISO that can apply controlled dissipation to the resolved band when needed.

\appendix
\clearpage
\phantomsection
\addcontentsline{toc}{part}{Appendices}
\makeatletter
\renewcommand{\toclevel@section}{1}
\renewcommand{\toclevel@subsection}{2}
\renewcommand{\toclevel@subsubsection}{3}
\renewcommand{\thesection}{\Alph{section}}
\renewcommand{\@seccntformat}[1]{%
  \ifnum\pdfstrcmp{#1}{section}=0
    Appendix~\csname the#1\endcsname\quad
  \else
    \csname the#1\endcsname\quad
  \fi
}
\makeatother

\section{Algorithms}
\label{app:algorithms}

\subsection{Cutoff setup and tracking}
\label{app:cutoff}

\begin{algorithm}[H]
\caption{Setup of the cutoff (once, at $t_0$)}
\label{alg:cutoff-setup}
\begin{algorithmic}[1]
\Require Jacobian $J_0$ at the initial state, outlier count $k$, system
size $n$, step size $h$
\If{the state splits into a second-order block $J=\begin{psmallmatrix}0&I\\-A&0\end{psmallmatrix}$}
\State Detect the block ordering from a probe vector, and read off the
spatial operator~$A$ through matrix-free products with~$J_0$
\State Run Arnoldi (IRAM)~\citep{sorensen1992,lehoucq1998} on~$A$ for $k+1$ eigenpairs, and take the
eigenvalue at the resolved--outlier boundary, $\omega\leftarrow\sqrt{\lambda}$
\State Form the right eigenvector~$\varphi$ and the left
eigenvector~$\psi$ of the full $J_0$ from the spatial eigenvector and
$\omega$ in closed form
\State $\lambda_{\mathrm{ref}}\leftarrow ih\omega$
\Else
\State Run Arnoldi (IRAM)~\citep{sorensen1992,lehoucq1998} directly on~$J_0$ for the largest-magnitude
$2k$ eigenpairs, and take the pair at the resolved--outlier boundary
\State $\varphi\leftarrow$ the corresponding eigenvector, $\psi\leftarrow\varphi$
\State $\lambda_{\mathrm{ref}}\leftarrow h\lambda$
\EndIf
\State Cache the one-time reference term
$J_0\varphi$-term $\leftarrow h\,\dfrac{\psi^{*}(J_0\varphi)}{\psi^{*}\varphi}$
\State \Return $\lambda_{\mathrm{ref}}$, $\varphi$, $\psi$, and the cached
reference term
\end{algorithmic}
\end{algorithm}

\begin{algorithm}[H]
\caption{Per-step Rayleigh quotient update of the cutoff}
\label{alg:cutoff-track}
\begin{algorithmic}[1]
\Require Cache from Algorithm~\ref{alg:cutoff-setup}
($\lambda_{\mathrm{ref}}$, $\varphi$, $\psi$, reference term), current
Jacobian--vector product with $J(t,y)$
\State Apply $J$ to the fixed eigenvector~$\varphi$, one matrix-free
Jacobian--vector product
\State $\hat\lambda\leftarrow\dfrac{\psi^{*}(J\varphi)}{\psi^{*}\varphi}$
\State $\omega_c\leftarrow\bigl|\lambda_{\mathrm{ref}}+h\hat\lambda-\text{(reference term)}\bigr|$
\State \Return $\omega_c$
\end{algorithmic}
\end{algorithm}

\subsection{Adaptive Krylov dimension}
\label{app:adaptive-m}

\begin{algorithm}[H]
\caption{Adaptive-dimension Arnoldi for $\phit_j(hJ)v$}
\label{alg:adaptive-arnoldi}
\begin{algorithmic}[1]
\Require Matrix-free product with $hJ$, starting vector~$v$, maximum
dimension~$m_{\max}$, minimum check dimension~$j_{\min}$, check
interval~$d$, safety factor and tolerance (safety, $\mathrm{tol}$), a
map $\mathrm{coefs}(H_j)$ returning the requested coefficients of
$\phit_j(hJ)v$ from a $j\times j$ Hessenberg matrix
\State $\beta\leftarrow\|v\|$, $V_1\leftarrow v/\beta$
\State $C_{\mathrm{prev}}\leftarrow$ undefined
\For{$j=1,\dots,m_{\max}$}
\State $w\leftarrow (hJ)V_j$
\State Orthogonalize~$w$ against $V_1,\dots,V_j$ by classical
Gram--Schmidt, two passes, recording the $j$-th Hessenberg column
\State $h_{j+1,j}\leftarrow\|w\|$ after orthogonalization
\If{$h_{j+1,j}$ is near zero (breakdown), or $j=m_{\max}$, or
($j\ge j_{\min}$ and $(j-j_{\min})\bmod d=0$)}
\State $C\leftarrow\mathrm{coefs}(H_{1:j,1:j})$
\If{breakdown, or $j=m_{\max}$, or ($C_{\mathrm{prev}}$ defined and
$\beta\,\|C-C_{\mathrm{prev}}\|<\mathrm{safety}\cdot\mathrm{tol}$)}
\State \Return $V_{1:j}$, $C$, $\beta$, $j$
\EndIf
\State $C_{\mathrm{prev}}\leftarrow C$
\EndIf
\State $V_{j+1}\leftarrow w/h_{j+1,j}$
\EndFor
\end{algorithmic}
\end{algorithm}

\subsection{A general MIISO step}
\label{app:miiso-step}

\begin{algorithm}[H]
\caption{General MIISO step (order~$N$)}
\label{alg:miiso-general}
\begin{algorithmic}[1]
\Require $f$, $t_n$, $y_n$, $h$, order~$N$ (nodes $c_i$ and
coefficients from Table~\ref{tab:parents} and the corresponding
equation among~\eqref{eq:miiso2}--\eqref{eq:miiso6}), cutoff state
from Section~\ref{sec:cutoff-tracking}, Krylov workspace, a defect
threshold~$\tau$
\State $f_n\leftarrow f(t_n,y_n)$, $f_{t,n}\leftarrow\partial_tf(t_n,y_n)$,
$J_n\leftarrow$ matrix-free Jacobian--vector product at~$(t_n,y_n)$
\If{the cutoff is tracked}
\State $\omega_c\leftarrow$ Rayleigh quotient update of
Algorithm~\ref{alg:cutoff-track}
\EndIf
\State Evaluate $\phit_1(hJ_n)f_n$ and $\phit_2(hJ_n)f_{t,n}$, and any
half- or fractional-step scalings these require, by
Algorithm~\ref{alg:adaptive-arnoldi}, diagonalizing the resulting
Hessenberg matrix and applying $\phit_j$ to its Ritz values
\State Run the Krylov evaluations that do not depend on one another concurrently
across threads
\For{each internal stage $i=2,\dots,$ the stage count of order~$N$}
\State Form~$Y_i$ from the coefficients of order~$N$, as a linear
combination of products $\phit_j(hJ_n)v$ already computed on~$f_n$,
$f_{t,n}$, and the defects at earlier stages
\State $D_i\leftarrow f(t_n+c_ih,Y_i)-f_n-J_n(Y_i-y_n)-c_ih f_{t,n}$
\If{$\|D_i\|<\tau\,(1+\|f_n\|)$}
\State treat $D_i$ as zero and skip its Krylov evaluation
\EndIf
\EndFor
\State Form the final weighted combinations $w_k$ of the stage defects,
from the coefficients of order~$N$
\State Evaluate $\phit_k(hJ_n)w_k$ for each~$k$ by
Algorithm~\ref{alg:adaptive-arnoldi}, run concurrently across threads
where those evaluations are independent
\State $y_{n+1}\leftarrow y_n+h\,\phit_1(hJ_n)f_n+h^2\phit_2(hJ_n)f_{t,n}
+h\sum_k\phit_k(hJ_n)w_k$
\State \Return $y_{n+1}$
\end{algorithmic}
\end{algorithm}

\section{Proofs and analysis}
\label{app:proofs}

\subsection{Global damping on oscillatory modes}
\label{app:global-dissipation}

\begin{proposition}[Dissipation on every oscillatory mode]
\label{prop:global-dissipation}
For any fixed $\rho_\infty\in[0,1)$, let $\mathbf{A}(\Omega)$ be the
generalized-$\alpha$ amplification matrix and $\rho(\Omega)$ its spectral
radius.
HHT-$\alpha$~\citep{hilberhughestaylor1977} and
WBZ-$\alpha$~\citep{wood1980} are the special cases $\alpha_m=0$ and
$\alpha_f=0$ of the same matrix.
Let $R_{\mathrm{Song}}$ be the mixed-Pad\'e map of~\citet{song2023}.
Then $\rho(\Omega)<1$ and $|R_{\mathrm{Song}}(i\Omega)|<1$ for every
$\Omega=\omega\Delta t>0$.
Every oscillatory mode therefore receives some dissipation at every
step size.
\end{proposition}
\begin{proof}
For generalized-$\alpha$, $\rho(\Omega)$ is a rational function
of~$\Omega^2$ with $\rho(0)=1$ and $\rho(\Omega)\to\rho_\infty$ as
$\Omega\to\infty$~\citep{chunghulbert1993}.
Since $\rho_\infty<1$, we have $\rho(\Omega)<1$ for every $\Omega>0$.
The same bound holds at $\alpha_m=0$ and at $\alpha_f=0$, so it applies
to HHT-$\alpha$~\citep{hilberhughestaylor1977} and
WBZ-$\alpha$~\citep{wood1980} with no further argument.
For~$R_{\mathrm{Song}}$, the spectral radius tends to~$\rho_\infty$ only
as $\Omega\to\infty$, and stays below one for every finite $\Omega>0$
when $\rho_\infty<1$~\citep{song2023}.
\end{proof}

\subsection{Repair of the subdiagonal column}
\label{app:repair}

For $N=3,4,5,6$ write
\begin{equation}
Q_N(z)=\sum_{j=0}^{N}\frac{(-z)^{j}}{j!},
\qquad
\pade{0}{N}(z)=\frac{1}{Q_N(z)},
\label{eq:QN}
\end{equation}
the raw subdiagonal column of Section~\ref{sec:dissipation}.
Since $|\pade{0}{N}(z)|>1$ exactly when $|Q_N(z)|<1$, define the
indicator
\begin{equation}
F_N(z)=|Q_N(z)|^2-1,
\label{eq:FN}
\end{equation}
so the instability region of~\pade{0}{N} in the closed left half-plane is
\begin{equation}
\mathcal{W}_N=\{z:\mathrm{Re}(z)\le 0,\ F_N(z)<0\}.
\label{eq:WN-general}
\end{equation}
For each order we choose a center~$c_N$ inside~$\mathcal{W}_N$ so that
every direction~$\alpha$ from~$c_N$ meets the boundary of~$\mathcal{W}_N$
at exactly one distance.
As shown below, $\mathcal{W}_N$ is also bounded.
Table~\ref{tab:region-data} lists~$c_N$, the shape of~$\mathcal{W}_N$,
and the repair parent for each order.

\begin{table}[H]
\centering
\caption{Instability region data for the raw subdiagonal column.}
\label{tab:region-data}
\begin{tabular}{@{}llll@{}}
\toprule
$N$ & Center $c_N$ & Shape of $\mathcal{W}_N$ & Repair parent \\
\midrule
3 & $0$ & wedge touching the origin & \pade{1}{2} \\
4 & $0$ & wedge touching the origin & \pade{1}{3} \\
5 & $i\,\omega_c^{(5)}$ & lobes on the imaginary axis & \pade{2}{3} \\
6 & $p_6$, $\overline{p_6}$ & pair of islands in the left half-plane & \pade{2}{4} \\
\bottomrule
\end{tabular}
\end{table}

For $N=3,4$ the center is the origin, and on the imaginary axis
\begin{equation}
F_3(i\omega)=\frac{\omega^4}{36}(\omega^2-3),
\qquad
F_4(i\omega)=\frac{\omega^6}{576}(\omega^2-8),
\label{eq:F34-axis}
\end{equation}
so $\mathcal{W}_3$ and $\mathcal{W}_4$ meet the imaginary axis on
$0<|\omega|<\sqrt{3}$ and $0<|\omega|<\sqrt{8}$, respectively, and extend
into the left half-plane to
$\min\{\mathrm{Re}(z):z\in\mathcal{W}_3\}\approx-0.040$,
$\min\{\mathrm{Re}(z):z\in\mathcal{W}_4\}\approx-0.237$.

For $N=5$, $F_5$ on the imaginary axis depends only on~$\omega^2$ and
vanishes where $\omega^4-15\omega^2+40=0$.
The two frequencies are
\begin{equation}
\omega_{\mathrm{lo}}=\sqrt{\tfrac{15-\sqrt{65}}{2}}\approx 1.8625,
\qquad
\omega_{\mathrm{hi}}=\sqrt{\tfrac{15+\sqrt{65}}{2}}\approx 3.3958,
\label{eq:F5-roots}
\end{equation}
so $\mathcal{W}_5$ meets the imaginary axis on
$\omega_{\mathrm{lo}}<|\omega|<\omega_{\mathrm{hi}}$, a pair of lobes that
do not touch the origin.
The center is the midpoint of that interval,
$c_5=i\,\omega_c^{(5)}$ with
$\omega_c^{(5)}=\tfrac12(\omega_{\mathrm{lo}}+\omega_{\mathrm{hi}})$.

For $N=6$, $Q_6'(z)=-Q_5(z)$.
Newton's method on~$Q_6$ with this derivative yields a zero~$p_6$ in the
open left half-plane,
$p_6\approx-0.8036+3.6977i$, and its conjugate~$\overline{p_6}$.
At these points $\pade{0}{6}(z)=1/Q_6(z)$ is undefined, and
$\mathcal{W}_6$ is a pair of islands about~$p_6$ and~$\overline{p_6}$ that
do not meet the imaginary axis.

For each order, the boundary distance
\begin{equation}
\rho_N(\alpha)=\inf\{r>0:F_N(c_N+re^{i\alpha})=0\}
\label{eq:rho-ray}
\end{equation}
is the distance from~$c_N$ to the outer edge of~$\mathcal{W}_N$ in
direction~$\alpha$.
It is finite for every~$\alpha$ because $\mathcal{W}_N$ is bounded and
meets each direction from~$c_N$ once, as shown in the proof below.

\paragraph{The repair}
\label{par:repair}
Fix a relative pad width~$\delta=0.10$.
Reflect~$z$ across the real axis of~$c_N$ into the half-plane where
$\rho_N$ is defined by setting
$z'=c_N+(\mathrm{Re}(z)-\mathrm{Re}(c_N),|\mathrm{Im}(z)-\mathrm{Im}(c_N)|)$,
and write $r=|z'-c_N|$ and $\alpha=\arg(z'-c_N)$.
The taper weight is
\begin{equation}
\theta_N(z)=
\begin{cases}
1, & r\le\rho_N(\alpha),\\[2pt]
\dfrac{1+\delta-r/\rho_N(\alpha)}{\delta},
& \rho_N(\alpha)<r<(1+\delta)\rho_N(\alpha),\\[6pt]
0, & r\ge(1+\delta)\rho_N(\alpha).
\end{cases}
\label{eq:theta-general}
\end{equation}
Let $R_{\mathrm{raw}}=\pade{0}{N}$ and let $R_{\mathrm{fix}}$ be the
repair parent of Table~\ref{tab:region-data}.
This parent satisfies Ehle's criterion~\eqref{eq:ehle} and agrees
with~$e^{z}$ through order~$N$.
The repaired map is
\begin{equation}
R_D(z)=
|R_{\mathrm{fix}}(z)|^{\theta_N(z)}\,|R_{\mathrm{raw}}(z)|^{1-\theta_N(z)}\,
\exp\!\Bigl(i\bigl[(1-\theta_N(z))\arg R_{\mathrm{raw}}(z)
+\theta_N(z)\arg R_{\mathrm{fix}}(z)\bigr]\Bigr),
\label{eq:wedge-blend-general}
\end{equation}
where $\arg R_{\mathrm{raw}}$ and $\arg R_{\mathrm{fix}}$ are continuous
and vanish at~$c_N$.

\begin{proposition}[$A$-acceptability and $L$-stability of the repaired map]
\label{prop:repair-general}
For each $N=3,4,5,6$, the map~$R_D$ of~\eqref{eq:wedge-blend-general}
satisfies $|R_D(z)|\le 1$ for all $\mathrm{Re}(z)\le 0$, and
$R_D(z)\to 0$ as $|z|\to\infty$ in the left half-plane.
\end{proposition}
\begin{proof}
$Q_N$ has degree~$N$, so for $|z|$ large enough,
\begin{equation}
|Q_N(z)|\ge\frac{|z|^N}{N!}-\sum_{j=0}^{N-1}\frac{|z|^j}{j!}>1,
\label{eq:QN-growth}
\end{equation}
since the leading term grows past every lower-order term and past the
constant~$1$.
Thus $F_N(z)$ cannot stay negative, $\mathcal{W}_N$ is bounded, and
$\rho_N(\alpha)$ in~\eqref{eq:rho-ray} is finite for every
direction~$\alpha$.

When $\theta_N=1$, $R_D=R_{\mathrm{fix}}$, which is $A$-acceptable by
Ehle's criterion~\eqref{eq:ehle}.
When $\theta_N=0$, we have $r\ge(1+\delta)\rho_N(\alpha)$, so~$z$ lies
outside~$\mathcal{W}_N$ by definition of~$\rho_N$.
Then $F_N(z)\ge 0$ and $R_D=R_{\mathrm{raw}}$ satisfies
$|R_{\mathrm{raw}}(z)|\le 1$ by~\eqref{eq:FN} and~\eqref{eq:WN-general}.
When $0<\theta_N<1$, both parents have modulus at most one at~$z$,
$R_{\mathrm{fix}}$ everywhere and $R_{\mathrm{raw}}$ because
$r>\rho_N(\alpha)$ places~$z$ outside~$\mathcal{W}_N$.
Hence
$|R_{\mathrm{fix}}|^{\theta}|R_{\mathrm{raw}}|^{1-\theta}\le 1$ for every
$\theta\in[0,1]$, and $|R_D(z)|\le 1$ for all $\mathrm{Re}(z)\le 0$.
For $L$-stability, boundedness of~$\mathcal{W}_N$ implies that for $|z|$
large enough we have $\theta_N(z)=0$, so
$R_D(z)=R_{\mathrm{raw}}(z)=\pade{0}{N}(z)\to 0$ as $|z|\to\infty$ in the
left half-plane, by~\eqref{eq:QN} and~\eqref{eq:QN-growth}.
\end{proof}

\subsection{Resolved and outlier mode behavior}
\label{app:mode-behavior}

Throughout, $R_C$ and $R_D$ are the parents of
Table~\ref{tab:parents}, $R$ is the blend~\eqref{eq:rblend}, and
$\theta$ is the taper~\eqref{eq:taper}.

\begin{proposition}[Exact non-dissipation below the taper]
\label{prop:zero-dissipation}
If $\omega=|z|\le\omega_0$, then $|R(i\omega)|=1$.
\end{proposition}
\begin{proof}
For $\omega\le\omega_0$, $\theta(\omega)=0$, so $R=R_C$ by~\eqref{eq:rblend}.
Every Gauss parent satisfies $|R_C(i\omega)|=1$ on the imaginary axis
(Section~\ref{sec:gauss}).
\end{proof}

\begin{proposition}[Annihilation above the cutoff]
\label{prop:annihilation-general}
If $\omega=|z|\ge\omega_c$, then $R(z)=R_D(z)$, and $R(z)\to 0$ as
$|z|\to\infty$ in the left half-plane, for every member from MIISO2 through
MIISO6.
\end{proposition}
\begin{proof}
For $\omega\ge\omega_c$, $\theta(\omega)=1$, so $R=R_D$
by~\eqref{eq:rblend}.
For MIISO2, $R_D=\pade{0}{2}$ already satisfies Ehle's
criterion~\eqref{eq:ehle} and is $L$-stable.
For MIISO3 through MIISO6, $R_D$ is the repaired map of
Proposition~\ref{prop:repair-general}, which is $L$-stable.
\end{proof}

\begin{proposition}[Withdrawal rate of dissipation]
\label{prop:withdrawal-general}
Fix a physical frequency~$\varpi$ with $h\varpi\ge\omega_c$, such that the mode
is an outlier at step size~$h$.
Let~$p$ be the order of the Pad\'e branch that governs~$R_D$ on the
imaginary axis near the origin, either the raw column~\pade{0}{N} when its
instability region does not meet the imaginary axis at the origin or
the repair parent otherwise as listed in Table~\ref{tab:withdrawal}.
Then
\begin{equation}
1-|R_D(ih\varpi)|=c\,(h\varpi)^q+O\bigl((h\varpi)^{q+2}\bigr)
\qquad\text{as }h\to 0,
\label{eq:withdrawal-rate}
\end{equation}
for a positive constant~$c$ depending on the branch, where
$q=p+1$ if~$p$ is odd and $q=p+2$ if~$p$ is even.
Consequently, over a fixed interval $[0,T]$ advanced in $T/h$ steps,
the retained amplitude of the mode satisfies
\begin{equation}
\bigl|R_D(ih\varpi)\bigr|^{T/h}
=\exp\!\Bigl(-\tfrac{c}{2}\,T\,h^{q-1}\varpi^{q}\Bigr)
\bigl(1+o(1)\bigr)
\qquad\text{as }h\to 0,
\label{eq:withdrawal-window}
\end{equation}
which tends to~$1$ as $h\to 0$, at rate~$h^{q-1}$.
Table~\ref{tab:withdrawal} gives~$q$ for each member.
\end{proposition}
\begin{proof}
Write $Q(z)$ for the denominator of whichever branch governs~$R_D$ near
the origin, normalized so the branch is $1/Q(z)$ or $P(z)/Q(z)$ in
lowest terms, and let the branch agree with~$e^z$ through
order~$p$, so $R_D(z)e^{-z}=1+cz^{p+1}+O(z^{p+2})$ for some real
constant~$c$, since the branch has real coefficients.
On the imaginary axis,
\begin{equation}
|R_D(i\omega)|^2
=R_D(i\omega)R_D(-i\omega)
=1+c(i\omega)^{p+1}+c(-i\omega)^{p+1}+O(\omega^{p+2}).
\end{equation}
If~$p$ is odd, $p+1$ is even, $(i\omega)^{p+1}=(-i\omega)^{p+1}$ is real,
and the two terms add, giving $|R_D(i\omega)|^2-1=O(\omega^{p+1})$.
If~$p$ is even, $p+1$ is odd, the two terms cancel, and the next order,
$O(\omega^{p+2})$, dominates.
In either case $1-|R_D(i\omega)|=\tfrac{c}{2}\omega^q+O(\omega^{q+2})$
with~$q$ as stated, by expanding $\sqrt{1-x}=1-x/2+O(x^2)$.
Equation~\eqref{eq:withdrawal-window} follows by taking the logarithm of
$|R_D(ih\varpi)|^{T/h}$, substituting~\eqref{eq:withdrawal-rate}, and
exponentiating.
\end{proof}

\begin{table}[H]
\centering
\caption{Governing branch near the origin and withdrawal exponent~$q$
for each MIISO order.}
\label{tab:withdrawal}
\begin{tabular}{@{}llc@{}}
\toprule
Method & Governing branch near~$z=0$ & $q$ \\
\midrule
MIISO2 & \pade{0}{2} (raw, no repair) & 4 \\
MIISO3 & \pade{1}{2} (repair touches the origin) & 4 \\
MIISO4 & \pade{1}{3} (repair touches the origin) & 6 \\
MIISO5 & \pade{0}{5} (raw, repair region off-axis) & 6 \\
MIISO6 & \pade{0}{6} (raw, repair region off-axis) & 8 \\
\bottomrule
\end{tabular}
\end{table}

\subsection{Stability and convergence}
\label{app:stability-convergence}

\begin{assumption}[Smooth flow and a well-separated cutoff]
\label{ass:smooth}
On a neighborhood of the exact trajectory, $f(t,y)$ has bounded
derivatives through the order the method requires, each
Jacobian~$J(t,y)$ is diagonalizable with uniformly bounded eigenvector
conditioning, the cutoff~$\omega_c$ varies in a locally Lipschitz way
with~$(t,y)$, no eigenvalue of~$hJ$ coincides with a pole of~$R_C$
or~$R_D$, and the Arnoldi error in a one-step increment is of the same
or higher local order as the step.
\end{assumption}

\begin{theorem}[Unconditional $A$-stability]
\label{thm:A-stable-general}
Under Table~\ref{tab:parents} and Proposition~\ref{prop:repair-general},
the blend~$R$ of~\eqref{eq:rblend} satisfies $|R(z)|\le 1$ for all
$\mathrm{Re}(z)\le 0$, independently of the step size~$h$, for each of
MIISO2 through MIISO6.
\end{theorem}
\begin{proof}
For every order, $R_C$ is a diagonal Gauss parent, $A$-acceptable with
modulus one on the imaginary axis by~\eqref{eq:ehle} at $m=n$.
$R_D$ is $A$-acceptable for every order, directly for MIISO2, whose raw
column already satisfies~\eqref{eq:ehle}, and by
Proposition~\ref{prop:repair-general} for MIISO3 through MIISO6.
For $0<\theta(\omega)<1$, both parents have modulus at most one, so
$|R|=|R_D|^{\theta}|R_C|^{1-\theta}\le 1$ by the same magnitude-blend
inequality used in the proof of Proposition~\ref{prop:repair-general}.
For $\theta=0$ or $\theta=1$, $R$ equals~$R_C$ or~$R_D$ directly, each
$A$-acceptable.
Hence $|R(z)|\le 1$ on $\mathrm{Re}(z)\le 0$ for every order, with no
restriction on~$h$.
\end{proof}

Figure~\ref{fig:stability-surfaces} shows $|R(z)|$ over the left half-plane
for each member of the family, with the vertical axis clipped at~$1$.
Each amplification factor is $A$-stable.

\begin{figure}[H]
\centering
\includegraphics[width=\textwidth]{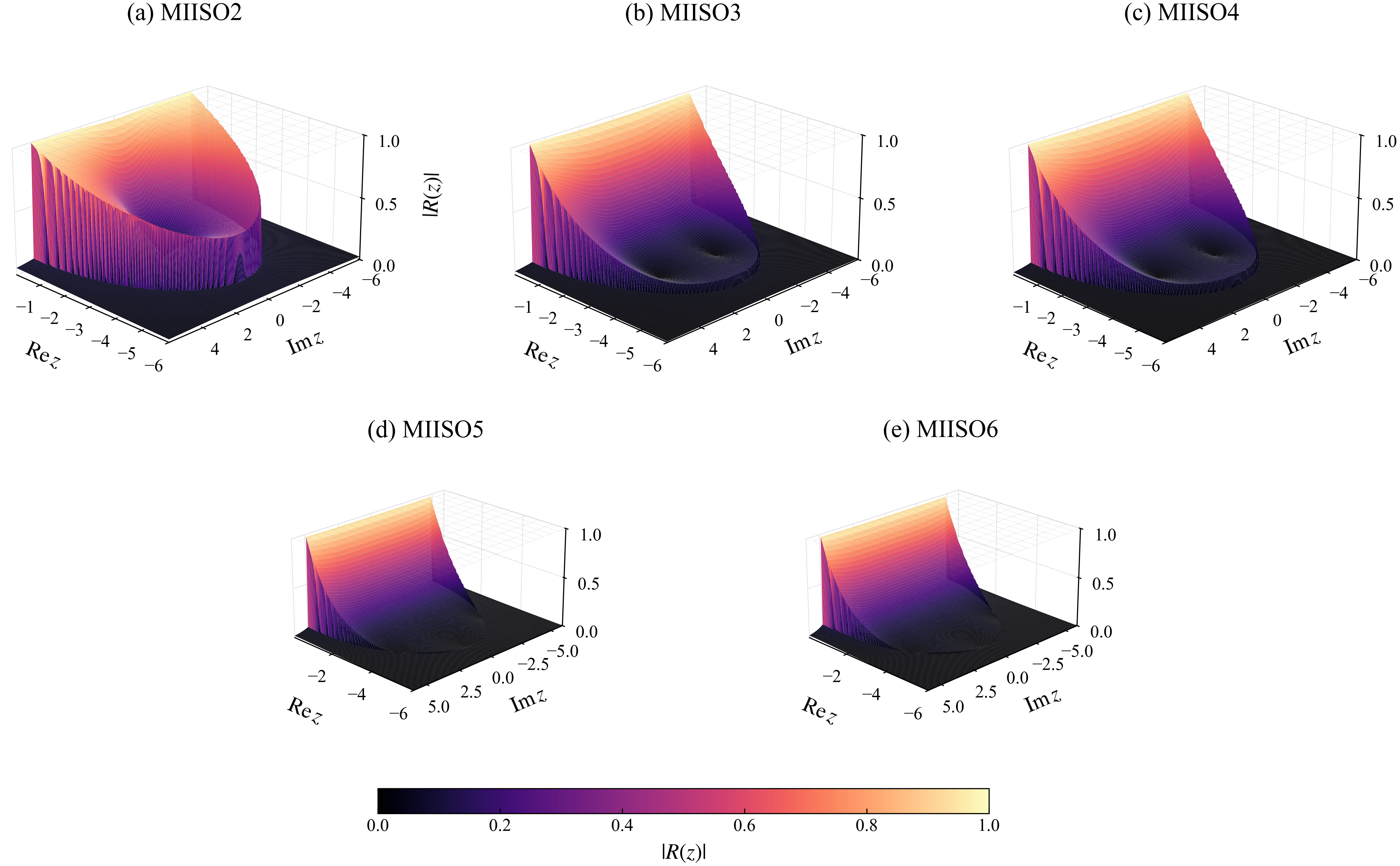}
\caption{$|R(z)|$ for MIISO2 through MIISO6 from the fixed--fixed degree-$4$ bar.}
\label{fig:stability-surfaces}
\end{figure}

\begin{lemma}[Order inherited by the blend]
\label{lemma:order-general}
For a method of order~$N$, $R(z)-e^{z}=O(z^{N+1})$, uniformly in the
taper weight~$\theta$, and consequently
$\phit_j(z)-\varphi_j(z)=O(z^{N+1-j})$ for $1\le j\le N$.
\end{lemma}
\begin{proof}
Both~$R_C$ and~$R_D$ agree with~$e^z$ through order~$N$, $R_C$ by
construction of the Gauss parent at the matching order and~$R_D$ by
construction of the raw or repaired column at order~$N$
(Section~\ref{sec:pade}, Appendix~\ref{app:repair}).
Writing $\log R_C(z)=z+O(z^{N+1})$ and $\log R_D(z)=z+O(z^{N+1})$ near
the origin, the magnitude-phase blend
$R=|R_D|^{\theta}|R_C|^{1-\theta}\exp(i[(1-\theta)\arg R_C+\theta\arg R_D])$
satisfies $\log R(z)=(1-\theta)\log R_C(z)+\theta\log R_D(z)=z+O(z^{N+1})$
for every $\theta\in[0,1]$, since the exponent is a convex combination
of two expansions that agree through the same order.
Exponentiating gives $R(z)-e^z=O(z^{N+1})$.
The recurrence~\eqref{eq:phirec} divides by~$z$ at each step, lowering
the order by one power of~$z$ relative to the exact recurrence at each
application, giving $\phit_j(z)-\varphi_j(z)=O(z^{N+1-j})$.
\end{proof}

\begin{theorem}[Global convergence]
\label{thm:convergence-general}
Under Assumption~\ref{ass:smooth}, on a fixed interval~$[t_0,T]$, the
MIISO method of order~$N$ converges at global order~$N$, for
$N=2,3,4,5,6$.
\end{theorem}
\begin{proof}
By Lemma~\ref{lemma:order-general}, the modified $\phit_j$ functions
match the exact~$\varphi_j$ through the order the underlying scheme's
local truncation error analysis requires. This implies that each member inherits the same
local order of its underlying exact-exponential scheme, the
exponential Rosenbrock--Euler, exprb32, and exprb43
methods~\citep{hochbruckostermann2009} for $N=2,3,4$, and pexprb54s5 and
pexprb65s7~\citep{luanostermann2016} for $N=5,6$, giving a local defect
$O(h^{N+1})$.
Under Assumption~\ref{ass:smooth}, the step map is locally Lipschitz
with constant $1+Ch$, and writing $e_n=y_n-y(t_n)$,
$\|e_{n+1}\|\le(1+Ch)\|e_n\|+Ch^{N+1}$.
Discrete Gronwall gives $\max_n\|e_n\|\le C_T h^{N}$, with the Arnoldi
error absorbed into the Lipschitz constant by
Assumption~\ref{ass:smooth}.
\end{proof}

\begin{corollary}[Extension to nonlinear and nonautonomous problems]
\label{cor:nonlinear}
Suppose the outlier band of~$J(t,y)$ remains separated from the
resolved band throughout~$[t_0,T]$, either by
Proposition~\ref{prop:mode-tracking} or empirically.
Then Theorem~\ref{thm:A-stable-general} and
Theorem~\ref{thm:convergence-general} hold with~$J(t,y)$ in place of a
fixed~$J$, for general nonlinear, nonautonomous~$f$ satisfying
Assumption~\ref{ass:smooth}.
\end{corollary}
\begin{proof}
Each MIISO step is built on the frozen local linearization
$\dot\delta=J(t,y)\delta$, with the amplification map~$R$ applied mode
by mode at $z=h\lambda$.
That construction extends directly to nonlinear problems and is
trivial on linear problems, where the defects vanish and the step reduces
to multiplication by~$R$.
Theorem~\ref{thm:A-stable-general} concerns the scalar map~$R$ at the
frozen~$z$, meaning that it holds at every such step.
Theorem~\ref{thm:convergence-general} uses
Lemma~\ref{lemma:order-general} and Assumption~\ref{ass:smooth}.
The frozen-Jacobian defects are already part of the local order of the
underlying exact
schemes~\citep{hochbruckostermann2009,luanostermann2016}, and remain
so when~$f$ is nonlinear.
The additional requirement is that the outlier band stay correctly
separated as~$J(t,y)$ evolves, which is the stated hypothesis.
\end{proof}

\subsection{Validity of the tracked cutoff}
\label{app:cutoff-bound}

Let~$J_0$ be the Jacobian at setup, with a simple eigenvalue~$\lambda_0$
separated from the rest of~$\mathrm{spec}(J_0)$ by the gap
\begin{equation}
\gamma=\min_{\mu\in\mathrm{spec}(J_0)\setminus\{\lambda_0\}}|\lambda_0-\mu|.
\label{eq:gap}
\end{equation}
Let~$\varphi$ and~$\psi$ be the associated right and left eigenvectors,
normalized so that $\psi^{*}\varphi=1$, as returned by
Algorithm~\ref{alg:cutoff-setup}.
Write $J(t)=J_0+\Delta(t)$ for the Jacobian at a later state, and let
$\lambda(t)$ be the eigenvalue of~$J(t)$ nearest~$\lambda_0$.

\begin{proposition}[Accuracy of the Rayleigh quotient update]
\label{prop:rayleigh-bound}
The two-sided Rayleigh quotient~\citep{ostrowski1958,saad2011} of Algorithm~\ref{alg:cutoff-track},
$\hat\lambda(t)=\psi^{*}J(t)\varphi/\psi^{*}\varphi$, satisfies
\begin{equation}
\hat\lambda(t)=\lambda_0+\psi^{*}\Delta(t)\varphi
\label{eq:rayleigh-exact}
\end{equation}
exactly, and
\begin{equation}
|\lambda(t)-\hat\lambda(t)|\le C\,\frac{\|\Delta(t)\|^2}{\gamma}
\label{eq:rayleigh-error}
\end{equation}
whenever $\|\Delta(t)\|$ is small enough relative to~$\gamma$.
The constant~$C$ depends only on the conditioning of the eigenvectors
of~$J_0$.
\end{proposition}
\begin{proof}
Since $\psi^{*}\varphi=1$ and $J_0\varphi=\lambda_0\varphi$,
\begin{equation}
\hat\lambda(t)
=\frac{\psi^{*}J_0\varphi+\psi^{*}\Delta(t)\varphi}{\psi^{*}\varphi}
=\lambda_0+\psi^{*}\Delta(t)\varphi,
\end{equation}
which is~\eqref{eq:rayleigh-exact}.
The right-hand side is the classical first-order eigenvalue
perturbation.
The Rayleigh quotient is stationary at the exact eigenvector of a
simple eigenvalue, so its value at the fixed pair~$(\varphi,\psi)$
agrees with~$\lambda(t)$ through first order
in~$\Delta(t)$~\citep{ostrowski1958,parlett1998,stewartsun1990}.
By the Davis--Kahan bound, the eigenvector of~$J(t)$ for~$\lambda(t)$
differs from~$\varphi$ by an angle $O(\|\Delta(t)\|/\gamma)$ as long as
$\|\Delta(t)\|$ is small relative to~$\gamma$ such that the gap has not
closed~\citep{daviskahan1970}.
The Rayleigh quotient error is second order in that angle, which
gives~\eqref{eq:rayleigh-error}.
\end{proof}

\begin{proposition}[Validity of the tracked cutoff under perturbation]
\label{prop:mode-tracking}
Let~$\lambda_0$ be the outlier eigenvalue of~$J_0$ located by
Algorithm~\ref{alg:cutoff-setup}, separated from the nearest resolved
eigenvalue by the gap~$\gamma$ of~\eqref{eq:gap}, and let~$\kappa(X)$
be the condition number of the eigenvector matrix of~$J_0$.
If $\|\Delta(t)\|\le\varepsilon$ with
\begin{equation}
2\,\kappa(X)\,\varepsilon<\gamma,
\label{eq:bauerfike-cond}
\end{equation}
then the outlier eigenvalue of~$J(t)=J_0+\Delta(t)$ remains separated
from every resolved eigenvalue of~$J(t)$ by a positive gap, and the
cutoff~$\omega_c$ continues to identify the same~$k$ eigenvalues as
outliers.
\end{proposition}
\begin{proof}
By the Bauer--Fike theorem~\citep{bauerfike1960}, every eigenvalue of~$J(t)$ lies within
$\kappa(X)\|\Delta(t)\|$ of some eigenvalue of~$J_0$.
Hence the outlier eigenvalue moves by at most~$\kappa(X)\varepsilon$, and
every resolved eigenvalue moves by at most~$\kappa(X)\varepsilon$ as
well.
Under~\eqref{eq:bauerfike-cond} these two displaced sets cannot meet, so
the gap stays open and the ordering of outlier against resolved
eigenvalues is unchanged.
\end{proof}

The Rayleigh error~\eqref{eq:rayleigh-error} assumes the gap has not
closed.
Proposition~\ref{prop:mode-tracking} gives a sufficient condition for
that through~\eqref{eq:bauerfike-cond}, under which
Algorithm~\ref{alg:cutoff-track} continues to track the same outlier
identified at setup.

\begin{remark}
\label{rem:cutoff-bounds-conservative}
Note that the bounds of Propositions~\ref{prop:rayleigh-bound}
and~\ref{prop:mode-tracking} are formal mathematical guarantees that may be treated almost as worst-case estimates.
Proposition~\ref{prop:rayleigh-bound} says that the error of the Rayleigh quotient estimate is
second order in the perturbation when the first outlier eigenvalue is separated from the nearest resolved eigenvalue by the gap~$\gamma$.
Proposition~\ref{prop:mode-tracking} supplies a sufficient gap condition
through Bauer--Fike, which bounds every eigenvalue by the same
displacement~$\kappa(X)\|\Delta(t)\|$, independent of how the
perturbation aligns with the spectrum.
In the runs of Appendix~\ref{app:tracking-validity}, the tracker remains accurate at
perturbations well past what~\eqref{eq:rayleigh-error}
and~\eqref{eq:bauerfike-cond} certify, since the true
eigenvalue displacement is typically significantly smaller than the
given bounds allow.
\end{remark}

\subsection{Validity of the adaptive Arnoldi procedure}
\label{app:adaptive-bound}

\begin{proposition}[Ritz-value accuracy transfers to \texorpdfstring{$\phit_j$}{phi\_j}]
\label{prop:ritz-transfer}
Let~$\hat\lambda$ be a Ritz value of the $m$-step Arnoldi iteration
on~$hJ$ started from~$v$, with true eigenvalue~$\lambda$ nearest to
it, and suppose~$\hat\lambda$ does not lie on the taper boundaries
$\omega=\omega_0$ or $\omega=\omega_c$ of~\eqref{eq:taper}.
Then
\begin{equation}
|\phit_j(\hat\lambda)-\phit_j(\lambda)|\le L\,|\hat\lambda-\lambda|,
\label{eq:phi-lipschitz}
\end{equation}
where~$L$ is a local Lipschitz constant for~$\phit_j$ on a neighborhood
of~$\lambda$ that does not cross either boundary.
\end{proposition}
\begin{proof}
Away from $\omega=\omega_0$ and $\omega=\omega_c$, the map~$R$ coincides
with either the holomorphic parent~$R_C$ or the holomorphic
parent~$R_D$.
Hence $\phit_j$, generated from~$R$ by~\eqref{eq:phirec}, is smooth on a
neighborhood of~$\lambda$ there and admits a local Lipschitz
constant~$L$.
The bound then follows from the mean value inequality applied
to~$\phit_j$ on the segment between~$\hat\lambda$ and~$\lambda$.
\end{proof}

\begin{remark}
\label{rem:adaptive-stopping-honest}
Note that Proposition~\ref{prop:ritz-transfer} transfers Ritz accuracy
to~$\phit_j$ once the Ritz values are close.
The stopping rule of Algorithm~\ref{alg:adaptive-arnoldi} is a
computational check on coefficient change.
It stops once that change falls below
$\mathrm{safety}\cdot\mathrm{tol}$ over the last~$d$ vectors, and
gives no guarantee that the Ritz values have entered the regime of
the proposition.
The a~priori rates of~\citep{saad1992,hochbrucklubich1997} for the
exact~$\varphi_j$ do not transfer to~$\phit_j$, for the reason in
Remark~\ref{rem:phi-analysis}.
\end{remark}

\section{Additional numerical results}
\label{app:experiments}

\subsection{Modal gallery}
\label{app:modal-gallery}

This is the full modal gallery referred to from
Section~\ref{sec:results-modal}, which shows only one case per
dimension.

\begin{figure}[H]
\centering
\includegraphics[width=\textwidth]{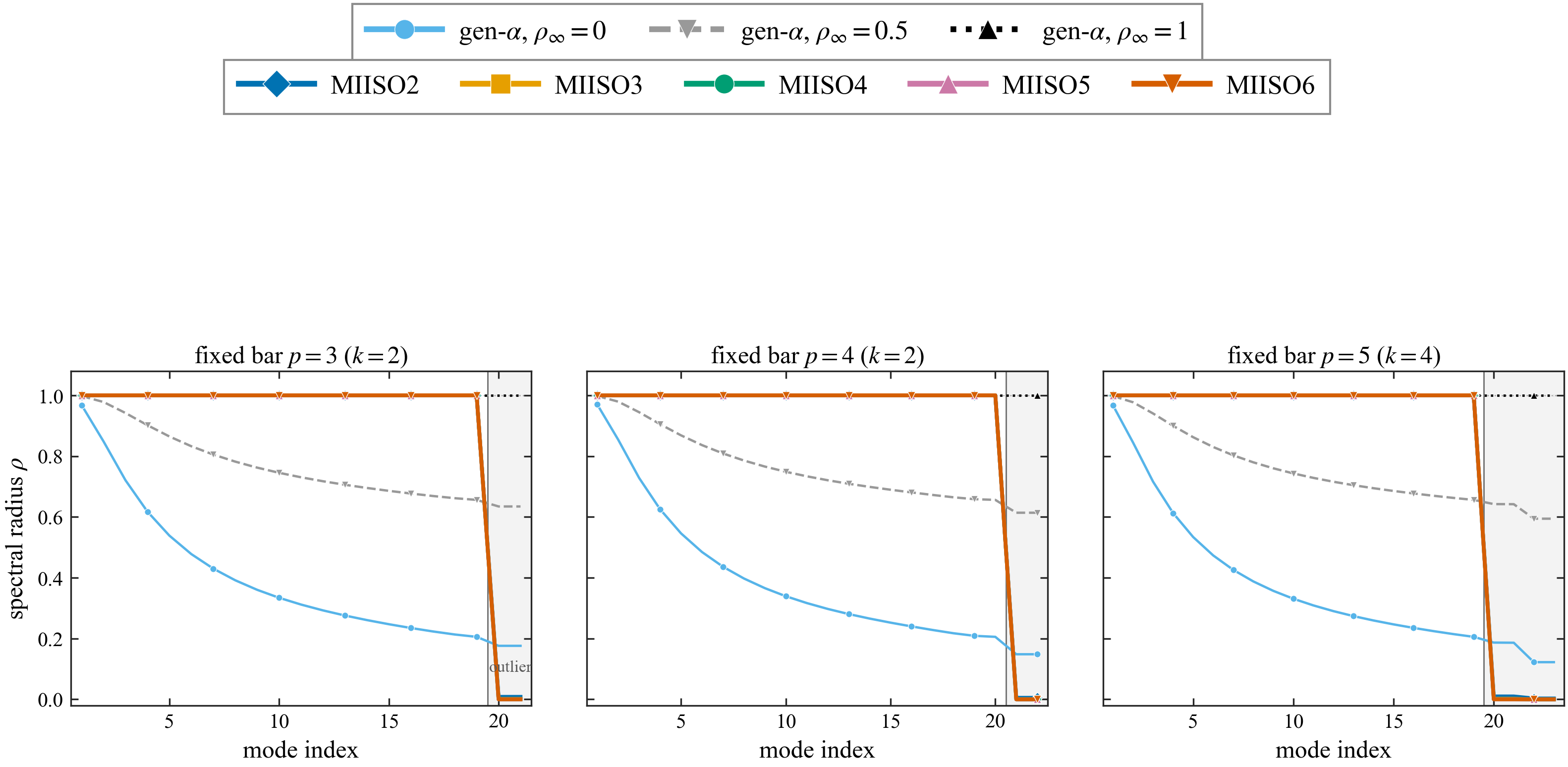}
\caption{Fully discrete modal amplification for the fixed--fixed bar at
degrees~$p=3$, $4$, and~$5$.}
\label{fig:modal-fig7}
\end{figure}

\begin{figure}[H]
\centering
\includegraphics[width=\textwidth]{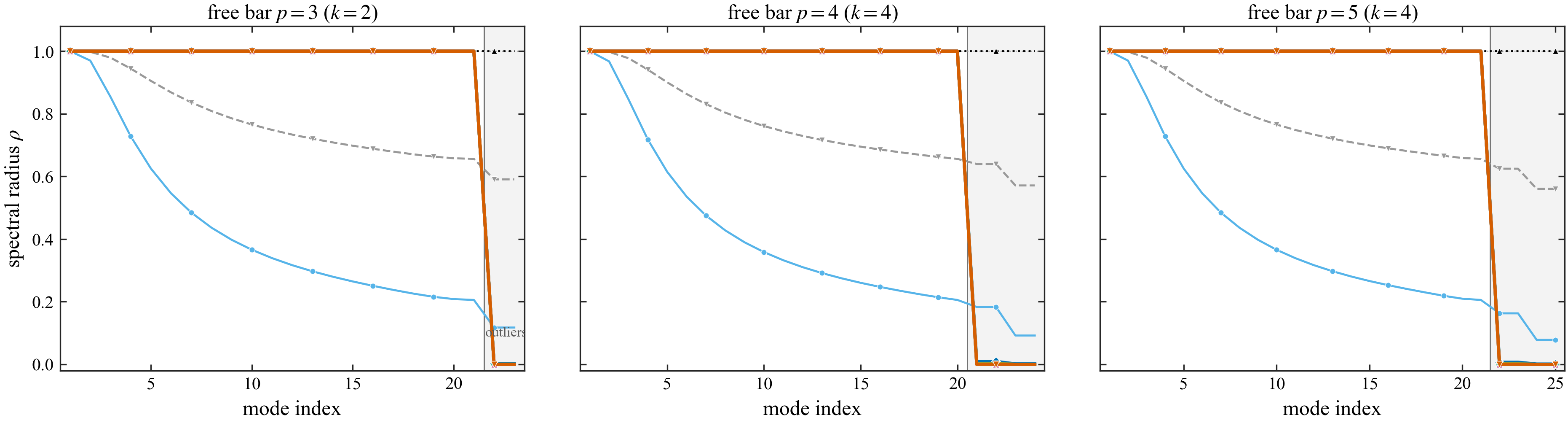}
\caption{Fully discrete modal amplification for the free--free bar at
degrees~$p=3$, $4$, and~$5$.}
\label{fig:modal-fig8}
\end{figure}

\begin{figure}[H]
\centering
\includegraphics[width=\textwidth]{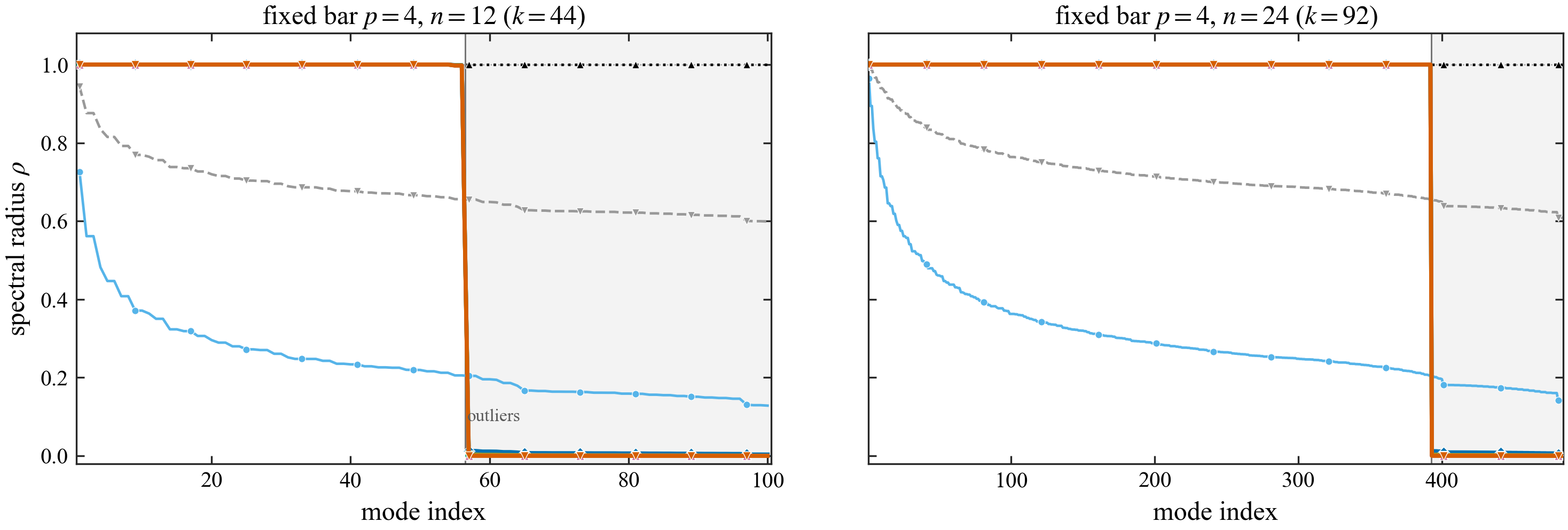}
\caption{Fully discrete modal amplification for the fixed--fixed bar in
2D at degree~$p=4$ and two mesh sizes.}
\label{fig:modal-fig9}
\end{figure}

\begin{figure}[H]
\centering
\includegraphics[width=0.55\textwidth]{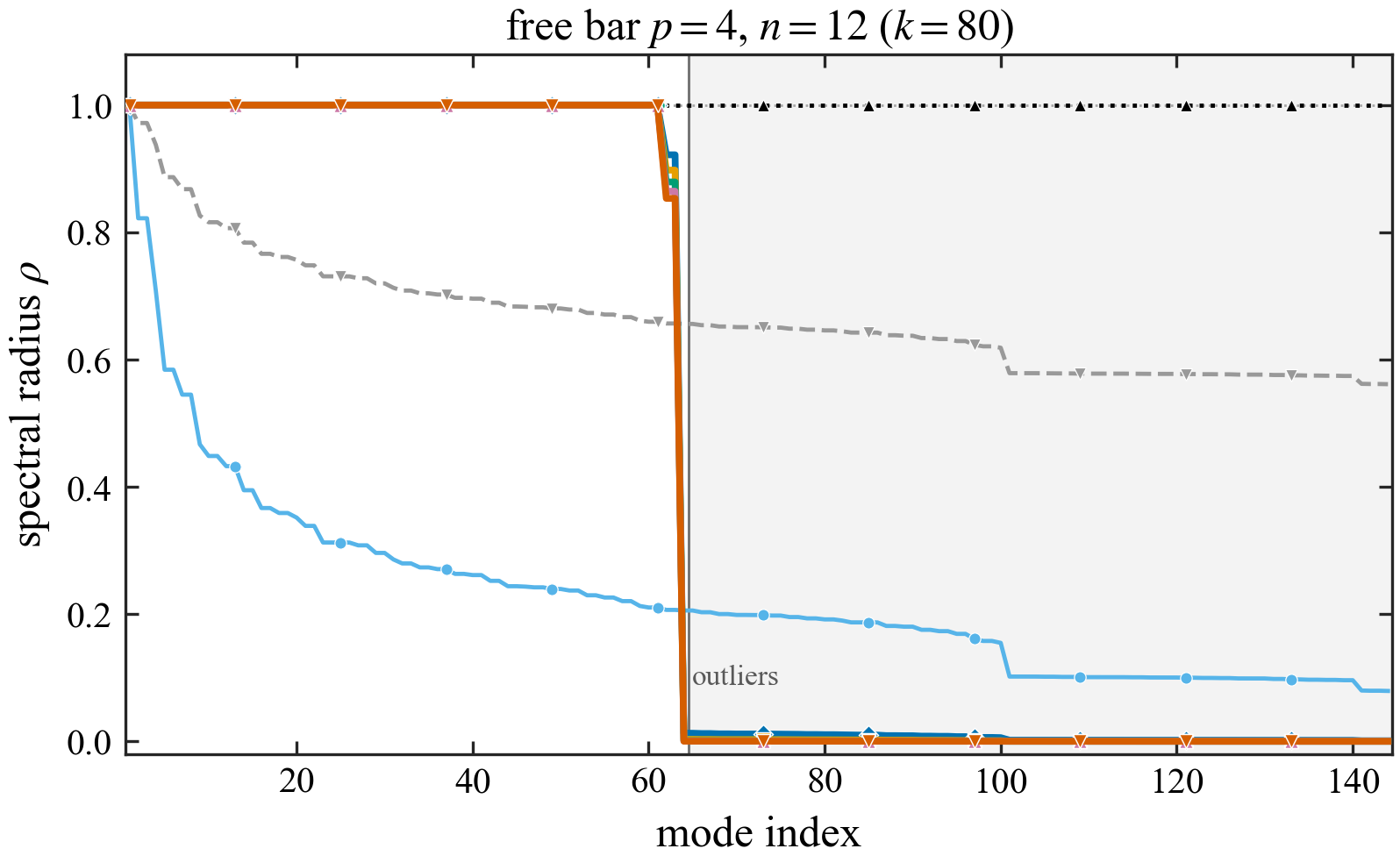}
\caption{Fully discrete modal amplification for the free--free bar in
2D at degree~$p=4$.}
\label{fig:modal-fig10}
\end{figure}

\begin{figure}[H]
\centering
\includegraphics[width=\textwidth]{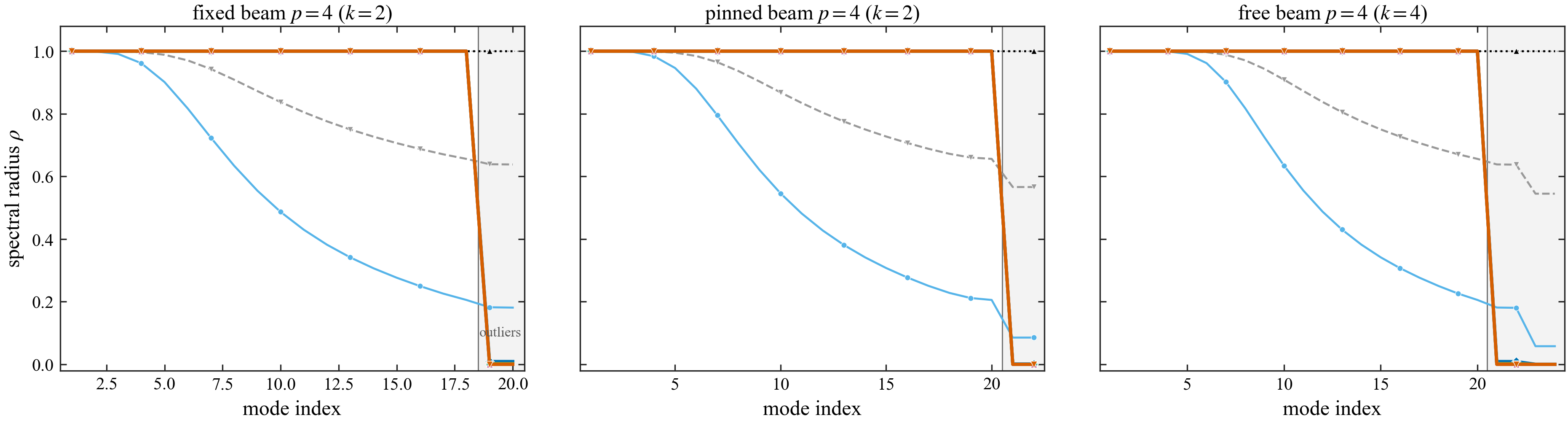}
\caption{Fully discrete modal amplification for fixed, pinned, and free
beams at degree~$p=4$.}
\label{fig:modal-fig11}
\end{figure}

\begin{figure}[H]
\centering
\includegraphics[width=\textwidth]{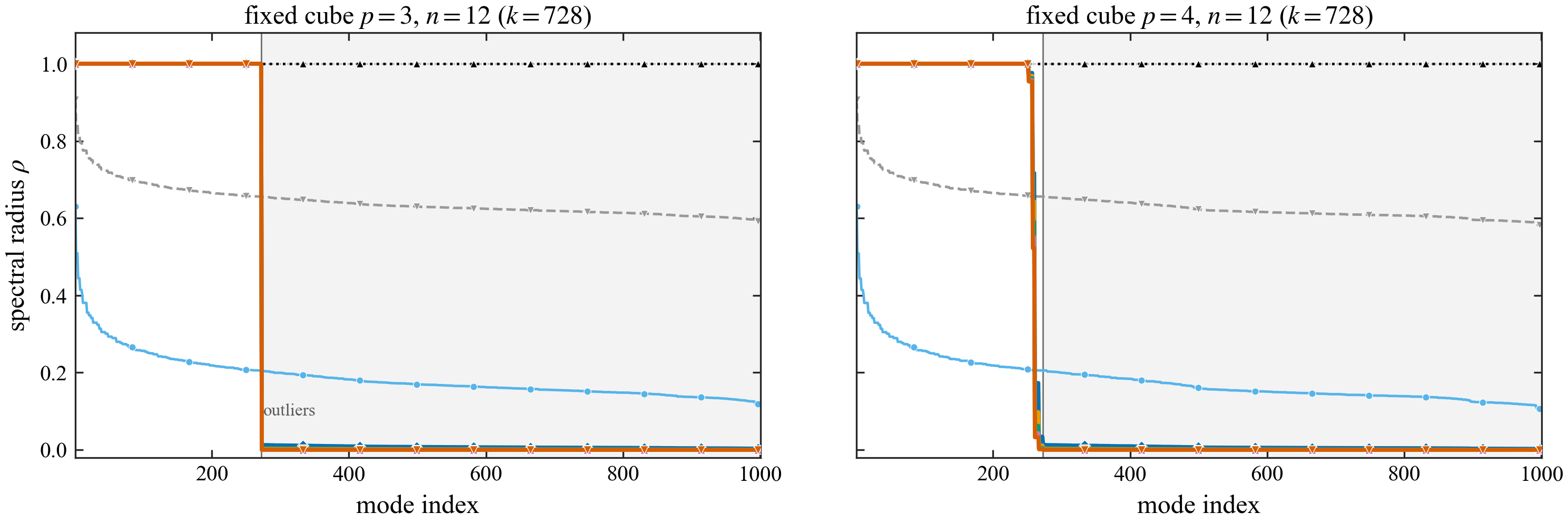}
\caption{Fully discrete modal amplification for the fixed--fixed unit
cube at degrees~$p=3$ and~$4$.}
\label{fig:modal-fig12}
\end{figure}

\subsection{Temporal convergence table}
\label{app:convergence-table}

Table~\ref{tab:convergence} records the relative~$L^2(\Omega)$ errors
and observed orders for the nonlinear membrane run of
Section~\ref{sec:results-convergence}.
Observed order is computed between successive step sizes as
$p_{\mathrm{obs}}=\log_2(e_h/e_{h/2})$.

\begin{table}[H]
\centering
\caption{Temporal convergence on~\eqref{eq:membrane-nonlinear}.}
\label{tab:convergence}
\begin{tabular}{@{}llcc@{}}
\toprule
Method & $h$ & relative $L^2$ error & $p_{\mathrm{obs}}$ \\
\midrule
MIISO2 & $5.000\times10^{-2}$ & $3.921\times10^{-1}$ & --- \\
 & $2.500\times10^{-2}$ & $9.824\times10^{-2}$ & $2.00$ \\
 & $1.250\times10^{-2}$ & $2.453\times10^{-2}$ & $2.00$ \\
 & $6.250\times10^{-3}$ & $6.124\times10^{-3}$ & $2.00$ \\
\midrule
MIISO3 & $5.000\times10^{-2}$ & $1.476\times10^{-3}$ & --- \\
 & $2.500\times10^{-2}$ & $1.758\times10^{-4}$ & $3.07$ \\
 & $1.250\times10^{-2}$ & $2.240\times10^{-5}$ & $2.97$ \\
 & $6.250\times10^{-3}$ & $2.867\times10^{-6}$ & $2.97$ \\
\midrule
MIISO4 & $5.000\times10^{-2}$ & $1.114\times10^{-3}$ & --- \\
 & $2.500\times10^{-2}$ & $7.396\times10^{-5}$ & $3.91$ \\
 & $1.250\times10^{-2}$ & $4.680\times10^{-6}$ & $3.98$ \\
 & $6.250\times10^{-3}$ & $2.932\times10^{-7}$ & $4.00$ \\
\midrule
MIISO5 & $5.000\times10^{-2}$ & $1.045\times10^{-5}$ & --- \\
 & $2.500\times10^{-2}$ & $2.386\times10^{-7}$ & $5.45$ \\
 & $1.250\times10^{-2}$ & $6.810\times10^{-9}$ & $5.13$ \\
 & $6.250\times10^{-3}$ & $2.174\times10^{-10}$ & $4.97$ \\
\midrule
MIISO6 & $5.000\times10^{-2}$ & $9.844\times10^{-6}$ & --- \\
 & $2.500\times10^{-2}$ & $1.640\times10^{-7}$ & $5.91$ \\
 & $1.250\times10^{-2}$ & $2.607\times10^{-9}$ & $5.98$ \\
 & $6.250\times10^{-3}$ & $4.132\times10^{-11}$ & $5.98$ \\
\bottomrule
\end{tabular}
\end{table}

\subsection{Cost to accuracy tables}
\label{app:cost-table}

Table~\ref{tab:cost-raw} records wall time and $L^2(\Omega)$
error at each step size on the ladder of Section~\ref{sec:results-cost}.
Table~\ref{tab:cost-accuracy} records the cheapest wall time from Table~\ref{tab:cost-raw} at each target tolerance.
An empty entry in Table~\ref{tab:cost-accuracy} means the method did
not reach that tolerance within the tested step-size range.

\begin{table}[H]
\centering
\footnotesize
\setlength{\tabcolsep}{5pt}
\caption{Step size, wall time, and relative $L^2(\Omega)$ error.}
\label{tab:cost-raw}
\begin{tabular}{@{}llrrr@{}}
\toprule
Method & $h$ & Steps & Rel.\ $L^2$ err.\ & Wall (s) \\
\midrule
MIISO2 & $0.05$ & $20$ & $3.073\times10^{-1}$ & $0.243$ \\
 & $0.025$ & $40$ & $7.688\times10^{-2}$ & $0.339$ \\
 & $0.0125$ & $80$ & $1.918\times10^{-2}$ & $0.624$ \\
 & $0.00625$ & $160$ & $4.787\times10^{-3}$ & $0.968$ \\
 & $0.003125$ & $320$ & $1.195\times10^{-3}$ & $1.228$ \\
 & $0.0015625$ & $640$ & $2.987\times10^{-4}$ & $0.979$ \\
\midrule
MIISO3 & $0.05$ & $20$ & $1.363\times10^{-3}$ & $0.328$ \\
 & $0.025$ & $40$ & $1.575\times10^{-4}$ & $0.611$ \\
 & $0.0125$ & $80$ & $1.941\times10^{-5}$ & $0.980$ \\
 & $0.00625$ & $160$ & $2.446\times10^{-6}$ & $1.185$ \\
 & $0.003125$ & $320$ & $3.081\times10^{-7}$ & $0.962$ \\
 & $0.0015625$ & $640$ & $3.870\times10^{-8}$ & $1.299$ \\
\midrule
MIISO4 & $0.05$ & $20$ & $9.824\times10^{-4}$ & $0.364$ \\
 & $0.025$ & $40$ & $6.766\times10^{-5}$ & $0.677$ \\
 & $0.0125$ & $80$ & $4.326\times10^{-6}$ & $1.001$ \\
 & $0.00625$ & $160$ & $2.717\times10^{-7}$ & $1.136$ \\
 & $0.003125$ & $320$ & $1.700\times10^{-8}$ & $1.004$ \\
 & $0.0015625$ & $640$ & $1.065\times10^{-9}$ & $1.632$ \\
\midrule
MIISO5 & $0.05$ & $20$ & $1.560\times10^{-5}$ & $0.558$ \\
 & $0.025$ & $40$ & $3.040\times10^{-7}$ & $0.865$ \\
 & $0.0125$ & $80$ & $7.828\times10^{-9}$ & $1.290$ \\
 & $0.00625$ & $160$ & $2.160\times10^{-10}$ & $1.316$ \\
 & $0.003125$ & $320$ & $1.618\times10^{-11}$ & $1.314$ \\
 & $0.0015625$ & $640$ & $6.578\times10^{-12}$ & $2.155$ \\
\midrule
MIISO6 & $0.05$ & $20$ & $1.656\times10^{-5}$ & $0.572$ \\
 & $0.025$ & $40$ & $3.112\times10^{-7}$ & $1.070$ \\
 & $0.0125$ & $80$ & $5.057\times10^{-9}$ & $1.582$ \\
 & $0.00625$ & $160$ & $8.854\times10^{-11}$ & $1.679$ \\
 & $0.003125$ & $320$ & $3.835\times10^{-12}$ & $2.247$ \\
 & $0.0015625$ & $640$ & $4.448\times10^{-12}$ & $2.756$ \\
\midrule
gen-$\alpha$ & $0.05$ & $20$ & $3.570\times10^{-1}$ & $0.526$ \\
 & $0.025$ & $40$ & $8.950\times10^{-2}$ & $1.004$ \\
 & $0.0125$ & $80$ & $2.235\times10^{-2}$ & $1.858$ \\
 & $0.00625$ & $160$ & $5.578\times10^{-3}$ & $2.738$ \\
 & $0.003125$ & $320$ & $1.393\times10^{-3}$ & $5.465$ \\
 & $0.0015625$ & $640$ & $3.481\times10^{-4}$ & $11.089$ \\
\bottomrule
\end{tabular}
\end{table}

\begin{table}[H]
\centering
\caption{Cheapest wall time (s) at each tolerance.}
\label{tab:cost-accuracy}
\begin{tabular}{@{}lcccccc@{}}
\toprule
Tol. & MIISO2 & MIISO3 & MIISO4 & MIISO5 & MIISO6 & gen-$\alpha$ \\
\midrule
$10^{-2}$ & 0.968 & 0.328 & 0.364 & 0.558 & 0.572 & 2.738 \\
$10^{-3}$ & 0.979 & 0.611 & 0.364 & 0.558 & 0.572 & 11.089 \\
$10^{-4}$ & --- & 0.962 & 0.677 & 0.558 & 0.572 & --- \\
$10^{-5}$ & --- & 0.962 & 1.001 & 0.865 & 1.070 & --- \\
$10^{-6}$ & --- & 0.962 & 1.004 & 0.865 & 1.070 & --- \\
$10^{-7}$ & --- & 1.299 & 1.004 & 1.290 & 1.582 & --- \\
$10^{-8}$ & --- & --- & 1.632 & 1.290 & 1.582 & --- \\
$10^{-9}$ & --- & --- & --- & 1.314 & 1.679 & --- \\
$10^{-10}$ & --- & --- & --- & 1.314 & 1.679 & --- \\
\bottomrule
\end{tabular}
\end{table}

\subsection{Scaling table}
\label{app:scaling-table}

Table~\ref{tab:scaling} records wall time and relative~$L^2(\Omega)$
error at fixed $h=0.00625$ across mesh sizes for the run of
Section~\ref{sec:results-cost}.
For each method the first row is wall time in seconds and the second
row is relative error against the eighth-order Runge--Kutta reference.

\begin{table}[H]
\centering
\scriptsize
\setlength{\tabcolsep}{3.5pt}
\caption{Wall time (s) and relative $L^2(\Omega)$ error at fixed
$h=0.00625$ by problem size~$N$.}
\label{tab:scaling}
\begin{tabular}{@{}lcccccccc@{}}
\toprule
Method & $64$ & $196$ & $484$ & $900$ & $1764$ & $3364$ & $5476$ & $9604$ \\
\midrule
MIISO2 & 0.091 & 0.569 & 0.949 & 1.391 & 2.084 & 3.157 & 7.057 & 18.093 \\
 & $6.12{\times}10^{-3}$ & $4.98{\times}10^{-3}$ & $4.79{\times}10^{-3}$ & $4.74{\times}10^{-3}$ & $4.71{\times}10^{-3}$ & $4.70{\times}10^{-3}$ & $4.69{\times}10^{-3}$ & $4.69{\times}10^{-3}$ \\
MIISO3 & 0.163 & 0.779 & 1.167 & 1.418 & 1.853 & 3.807 & 7.251 & 18.822 \\
 & $3.16{\times}10^{-6}$ & $2.51{\times}10^{-6}$ & $2.45{\times}10^{-6}$ & $2.43{\times}10^{-6}$ & $2.42{\times}10^{-6}$ & $2.42{\times}10^{-6}$ & $2.41{\times}10^{-6}$ & $2.41{\times}10^{-6}$ \\
MIISO4 & 0.227 & 0.767 & 1.127 & 1.255 & 2.167 & 4.384 & 8.537 & 21.889 \\
 & $2.95{\times}10^{-7}$ & $2.76{\times}10^{-7}$ & $2.72{\times}10^{-7}$ & $2.71{\times}10^{-7}$ & $2.70{\times}10^{-7}$ & $2.70{\times}10^{-7}$ & $2.69{\times}10^{-7}$ & $2.69{\times}10^{-7}$ \\
MIISO5 & 0.251 & 0.729 & 1.287 & 1.528 & 2.889 & 5.839 & 11.116 & 27.662 \\
 & $8.20{\times}10^{-9}$ & $1.92{\times}10^{-9}$ & $2.14{\times}10^{-10}$ & $2.55{\times}10^{-10}$ & $2.12{\times}10^{-10}$ & $1.91{\times}10^{-10}$ & $1.97{\times}10^{-10}$ & $1.91{\times}10^{-10}$ \\
MIISO6 & 0.341 & 1.015 & 1.643 & 2.555 & 5.132 & 8.543 & 16.436 & 39.895 \\
 & $5.47{\times}10^{-10}$ & $2.15{\times}10^{-10}$ & $8.97{\times}10^{-11}$ & $8.61{\times}10^{-11}$ & $9.37{\times}10^{-11}$ & $8.22{\times}10^{-11}$ & $8.43{\times}10^{-11}$ & $1.16{\times}10^{-10}$ \\
gen-$\alpha$ & 0.136 & 0.829 & 2.686 & 5.737 & 13.295 & 31.353 & 61.935 & 140.368 \\
 & $7.26{\times}10^{-3}$ & $5.82{\times}10^{-3}$ & $5.58{\times}10^{-3}$ & $5.51{\times}10^{-3}$ & $5.48{\times}10^{-3}$ & $5.47{\times}10^{-3}$ & $5.46{\times}10^{-3}$ & $5.46{\times}10^{-3}$ \\
\bottomrule
\end{tabular}
\end{table}

\subsection{Empirical validity of the outlier tracking}
\label{app:tracking-validity}

In this section, we perform an empirical validation of the outlier tracking mechanism of Section~\ref{sec:cutoff}.
Recall that mode mistracking refers to the failure mechanism described in Section~\ref{sec:discuss-tracking}.
We define two quantities to measure the separation of the outlier and resolved bands, recalling
the nonlinear membrane problem of Section~\ref{sec:results-convergence}.

First, we define the outlier gap as the ratio of the first outlier frequency to the
last resolved frequency,
\begin{equation}
g(q) = \frac{\omega_{N-k+1}(q)}{\omega_{N-k}(q)},
\label{eq:outlier-gap}
\end{equation}
evaluated on the live generalized eigenvalues of the current
configuration~$q$.
Values $g>1$ mean the outlier and resolved bands remain separated.
The limit $g\to1$ means the gap has closed and there is no
separation between the outlier and resolved bands.

Second, we define the subspace leakage as the sum of squared principal
angles~\citep{bjorckgolub1973} between the current outlier subspace at~$q$ and the outlier
subspace at the reference configuration.
\begin{equation}
L(q) = \sum_j \sin^2\theta_j(q).
\label{eq:subspace-leakage}
\end{equation}
This quantity is an indicator of how far the two subspaces have rotated apart.
$L=0$ means the subspaces coincide and there is no current failure of the outlier tracking mechanism.
As~$L$ rises toward an integer~$j$, roughly~$j$ outlier modes and
resolved modes have exchanged positions and are being mistracked.

Figure~\ref{fig:tracking-validity} shows both quantities on the
nonlinear membrane of Section~\ref{sec:results-convergence}, with the setup $N=64$, $p=4$,
$k=28$, at peak amplitudes $A=0.20$, $0.30$, and~$0.35$, integrated to
$T=2.5$.
At $A=0.20$, the amplitude used in the main results, the gap stays
above one and leakage stays small throughout, indicating that tracking remains valid
for the entire run.
At $A=0.30$ the gap reaches one near the end of the window and leakage
rises past one, so one pair of modes is mistracked for a short interval.
At $A=0.35$ the gap reaches one and leakage exceeds one at several
points in the run. Again one of the~$28$ outlier modes of this
discretization exchanges with a resolved mode and is mistracked for those intervals.
Although the pattern is cyclical, the exact pattern of the gap and leakage does not repeat from cycle to cycle. This
aligns with the amplitude-dependent period of the
cubic-hardening nonlinearity in~\eqref{eq:membrane-nonlinear}.

\begin{figure}[H]
\centering
\includegraphics[width=\textwidth]{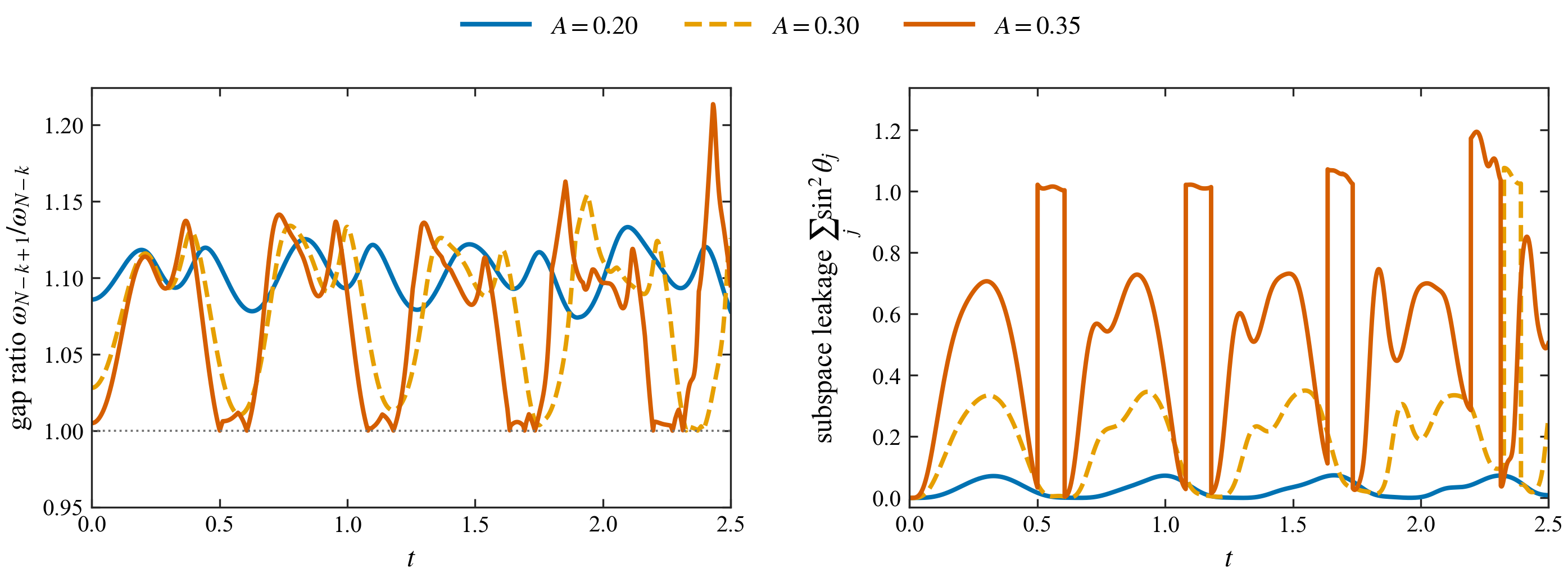}
\caption{Outlier gap~\eqref{eq:outlier-gap} (left) and subspace
leakage~\eqref{eq:subspace-leakage} (right) on the nonlinear membrane.}
\label{fig:tracking-validity}
\end{figure}

\paragraph{Scope of the mistracking}
\label{par:mistracking-scope}
The intervals of mistracking are short, even at the largest amplitude $A=0.35$.
The mistracking affects only one of the~$28$ outlier modes, and the single mode sits on the wrong
branch of the blend only during those small intervals. Correct identification
returns once the deformation recedes and the gap reopens. This indicates that
mistracking is a brief and localized phenomenon.

\subsection{Accuracy of the Rayleigh quotient tracking}
\label{app:rayleigh-accuracy}

This section checks the Rayleigh quotient update of
Algorithm~\ref{alg:cutoff-track} against the exact cutoff to empirically
test its validity.
At each saved time, the exact cutoff is obtained by solving the
generalized eigenproblem $K_t v = \lambda M v$ on the current
tangent stiffness and mass, then taking $\omega_c$ as the smallest of
the top~$k$ frequencies.

Figure~\ref{fig:rayleigh-accuracy} compares the exact and
Rayleigh-tracked cutoff on the nonlinear membrane of
Section~\ref{sec:results-convergence}, on a $12\times12$ mesh with
$N=196$ and $k=52$, at peak amplitude~$A=0.20$, over two return periods
to $T\approx2.61$.
Over the full run, the mean absolute error is~$0.687$ and the mean
relative error is~$1.13\times10^{-2}$. The peak relative error is
$4.33\times10^{-2}$ at~$t\approx2.30$.

\begin{figure}[H]
\centering
\includegraphics[width=\textwidth]{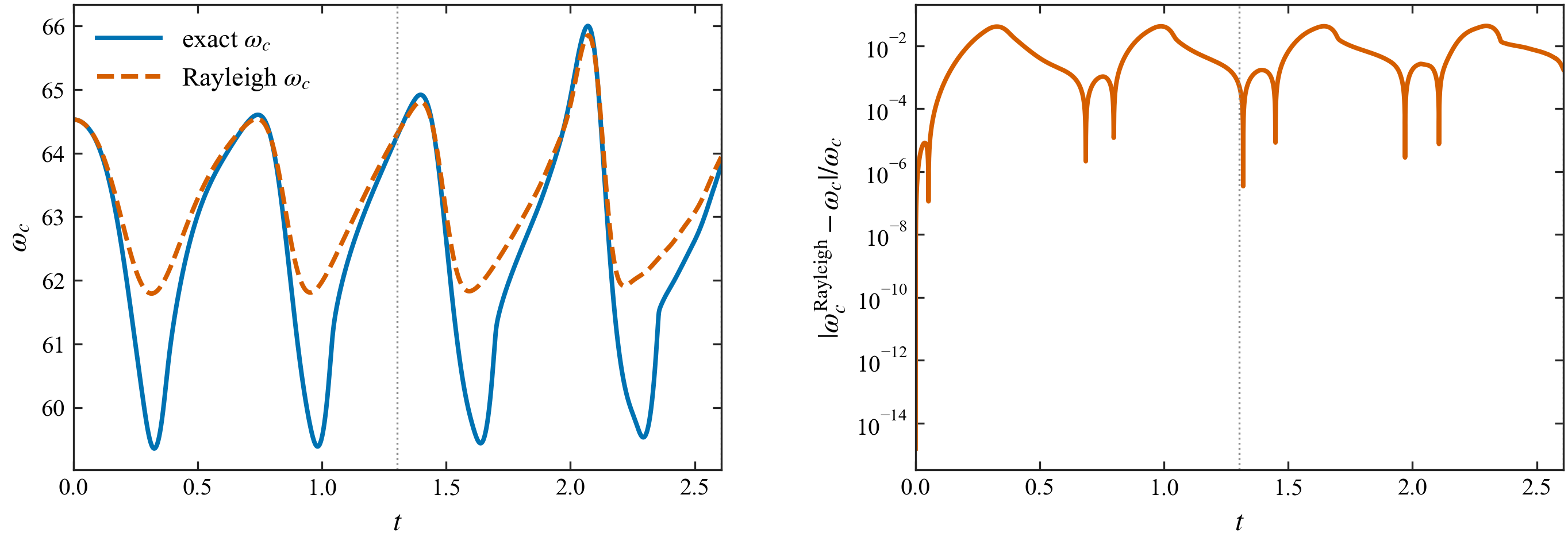}
\caption{Exact and Rayleigh-tracked cutoff~$\omega_c$ over time (left)
and relative error between them (right), on the nonlinear membrane.}
\label{fig:rayleigh-accuracy}
\end{figure}

Table~\ref{tab:rayleigh-accuracy} records the same comparison at intervals
of~$0.05$, together with the final time~$T\approx2.61$. The relative error
repeats periodically across cycles.

\begin{table}[H]
\centering
\scriptsize
\caption{Exact and Rayleigh-tracked cutoff~$\omega_c$ at intervals of~$0.05$
on the nonlinear membrane.}
\label{tab:rayleigh-accuracy}
\begin{tabular}{@{}ccccc@{\hspace{1.2em}}ccccc@{}}
\toprule
$t$ & $\omega_c$ & $\omega_c^{\mathrm{R}}$ & $|\Delta\omega|$ & $|\Delta\omega|/\omega_c$ & $t$ & $\omega_c$ & $\omega_c^{\mathrm{R}}$ & $|\Delta\omega|$ & $|\Delta\omega|/\omega_c$ \\
\midrule
0.00 & 64.5281 & 64.5281 & $9.95\times10^{-14}$ & $1.542\times10^{-15}$ & 1.35 & 64.6757 & 64.6118 & $6.390\times10^{-2}$ & $9.880\times10^{-4}$ \\
0.05 & 64.4358 & 64.4357 & $1.550\times10^{-4}$ & $2.406\times10^{-6}$ & 1.40 & 64.9158 & 64.8060 & $1.098\times10^{-1}$ & $1.691\times10^{-3}$ \\
0.10 & 64.1136 & 64.1356 & $2.208\times10^{-2}$ & $3.443\times10^{-4}$ & 1.45 & 64.3270 & 64.3348 & $7.801\times10^{-3}$ & $1.213\times10^{-4}$ \\
0.15 & 63.4502 & 63.5923 & $1.421\times10^{-1}$ & $2.239\times10^{-3}$ & 1.50 & 62.6398 & 63.0548 & $4.149\times10^{-1}$ & $6.624\times10^{-3}$ \\
0.20 & 62.3682 & 62.8594 & $4.911\times10^{-1}$ & $7.875\times10^{-3}$ & 1.55 & 60.8738 & 62.0430 & $1.1691$ & $1.921\times10^{-2}$ \\
0.25 & 60.9753 & 62.1673 & $1.1920$ & $1.955\times10^{-2}$ & 1.60 & 59.8056 & 61.8328 & $2.0272$ & $3.390\times10^{-2}$ \\
0.30 & 59.6244 & 61.8101 & $2.1857$ & $3.666\times10^{-2}$ & 1.65 & 59.4939 & 62.0002 & $2.5063$ & $4.213\times10^{-2}$ \\
0.35 & 59.6542 & 61.8975 & $2.2433$ & $3.760\times10^{-2}$ & 1.70 & 61.2056 & 62.2749 & $1.0693$ & $1.747\times10^{-2}$ \\
0.40 & 61.2515 & 62.2986 & $1.0471$ & $1.710\times10^{-2}$ & 1.75 & 61.9568 & 62.5864 & $6.296\times10^{-1}$ & $1.016\times10^{-2}$ \\
0.45 & 62.3103 & 62.8079 & $4.976\times10^{-1}$ & $7.987\times10^{-3}$ & 1.80 & 62.4668 & 62.9242 & $4.574\times10^{-1}$ & $7.322\times10^{-3}$ \\
0.50 & 63.0274 & 63.2852 & $2.578\times10^{-1}$ & $4.090\times10^{-3}$ & 1.85 & 62.9433 & 63.2798 & $3.365\times10^{-1}$ & $5.346\times10^{-3}$ \\
0.55 & 63.5226 & 63.6789 & $1.564\times10^{-1}$ & $2.462\times10^{-3}$ & 1.90 & 63.4595 & 63.6849 & $2.254\times10^{-1}$ & $3.551\times10^{-3}$ \\
0.60 & 63.8915 & 63.9943 & $1.028\times10^{-1}$ & $1.609\times10^{-3}$ & 1.95 & 64.0653 & 64.1469 & $8.157\times10^{-2}$ & $1.273\times10^{-3}$ \\
0.65 & 64.2071 & 64.2551 & $4.793\times10^{-2}$ & $7.465\times10^{-4}$ & 2.00 & 64.8768 & 64.7607 & $1.161\times10^{-1}$ & $1.790\times10^{-3}$ \\
0.70 & 64.4844 & 64.4631 & $2.127\times10^{-2}$ & $3.298\times10^{-4}$ & 2.05 & 65.8504 & 65.6822 & $1.682\times10^{-1}$ & $2.555\times10^{-3}$ \\
0.75 & 64.5969 & 64.5289 & $6.802\times10^{-2}$ & $1.053\times10^{-3}$ & 2.10 & 65.5039 & 65.4650 & $3.887\times10^{-2}$ & $5.934\times10^{-4}$ \\
0.80 & 64.1482 & 64.1565 & $8.213\times10^{-3}$ & $1.280\times10^{-4}$ & 2.15 & 62.8475 & 63.1764 & $3.289\times10^{-1}$ & $5.233\times10^{-3}$ \\
0.85 & 62.8099 & 63.1747 & $3.648\times10^{-1}$ & $5.808\times10^{-3}$ & 2.20 & 60.7537 & 61.9611 & $1.2074$ & $1.987\times10^{-2}$ \\
0.90 & 61.0724 & 62.1707 & $1.0983$ & $1.798\times10^{-2}$ & 2.25 & 59.9089 & 61.9898 & $2.0809$ & $3.473\times10^{-2}$ \\
0.95 & 59.7626 & 61.8090 & $2.0464$ & $3.424\times10^{-2}$ & 2.30 & 59.5485 & 62.1245 & $2.5761$ & $4.326\times10^{-2}$ \\
1.00 & 59.5356 & 61.9782 & $2.4426$ & $4.103\times10^{-2}$ & 2.35 & 61.1710 & 62.3266 & $1.1556$ & $1.889\times10^{-2}$ \\
1.05 & 61.3635 & 62.3426 & $9.791\times10^{-1}$ & $1.596\times10^{-2}$ & 2.40 & 61.8747 & 62.5372 & $6.625\times10^{-1}$ & $1.071\times10^{-2}$ \\
1.10 & 62.2466 & 62.7538 & $5.073\times10^{-1}$ & $8.149\times10^{-3}$ & 2.45 & 62.2369 & 62.7979 & $5.610\times10^{-1}$ & $9.014\times10^{-3}$ \\
1.15 & 62.8457 & 63.1635 & $3.178\times10^{-1}$ & $5.057\times10^{-3}$ & 2.50 & 62.6245 & 63.0683 & $4.438\times10^{-1}$ & $7.087\times10^{-3}$ \\
1.20 & 63.3461 & 63.5592 & $2.130\times10^{-1}$ & $3.363\times10^{-3}$ & 2.55 & 63.1214 & 63.4506 & $3.292\times10^{-1}$ & $5.215\times10^{-3}$ \\
1.25 & 63.8028 & 63.9344 & $1.316\times10^{-1}$ & $2.062\times10^{-3}$ & 2.60 & 63.7290 & 63.8816 & $1.526\times10^{-1}$ & $2.394\times10^{-3}$ \\
1.30 & 64.2478 & 64.2871 & $3.927\times10^{-2}$ & $6.113\times10^{-4}$ & 2.607 & 63.8247 & 63.9398 & $1.151\times10^{-1}$ & $1.804\times10^{-3}$ \\
\bottomrule
\end{tabular}
\end{table}

\paragraph{Periodicity of the relative error}
\label{par:error-periodicity}
The largest peaks recur at $t\approx0.35$, $1.00$, $1.65$, and~$2.30$,
each spaced by~$0.65$. Smaller intermediate peaks at~$t\approx0.75$,
$1.40$, and~$2.05$ repeat with the same spacing.
Appendix~\ref{app:tracking-validity} showed that
the quantities tracked there,
which we called the outlier gap and subspace leakage,
drift in timing and magnitude from one cycle to the next.
A plausible explanation is that the Rayleigh update tracks a fixed eigenvector~$\varphi$ against a
single evolving eigenvalue, so its error follows that eigenvector's
rotation each cycle. This indicates that the Rayleigh quotient tracking
is nearly exactly periodic between deformation cycles, while the outlier gap tracking
follows the amplitude-dependent drift of the nonlinear problem.

A mean relative error near~$1\%$ and a worst case near~$4\%$
indicate that the Rayleigh quotient tracking
is sufficient for tracking the blend taper correctly relative to the true
resolved--outlier boundary.
Combined with its cost of one Jacobian--vector product per step, it is a
reasonable choice for tracking the cutoff efficiently in practice.

\subsection{Arch strain-energy snapshots}
\label{app:bridge-psi}

This section shows the complementary Saint-Venant--Kirchhoff
strain-energy density~$\psi$ in Figure~\ref{fig:bridge-psi} at the same
snapshot times as the displacement panels in
Figure~\ref{fig:earthquake}, viewed from the opposite corner.

\begin{figure}[H]
\centering
\includegraphics[width=0.44\textwidth]{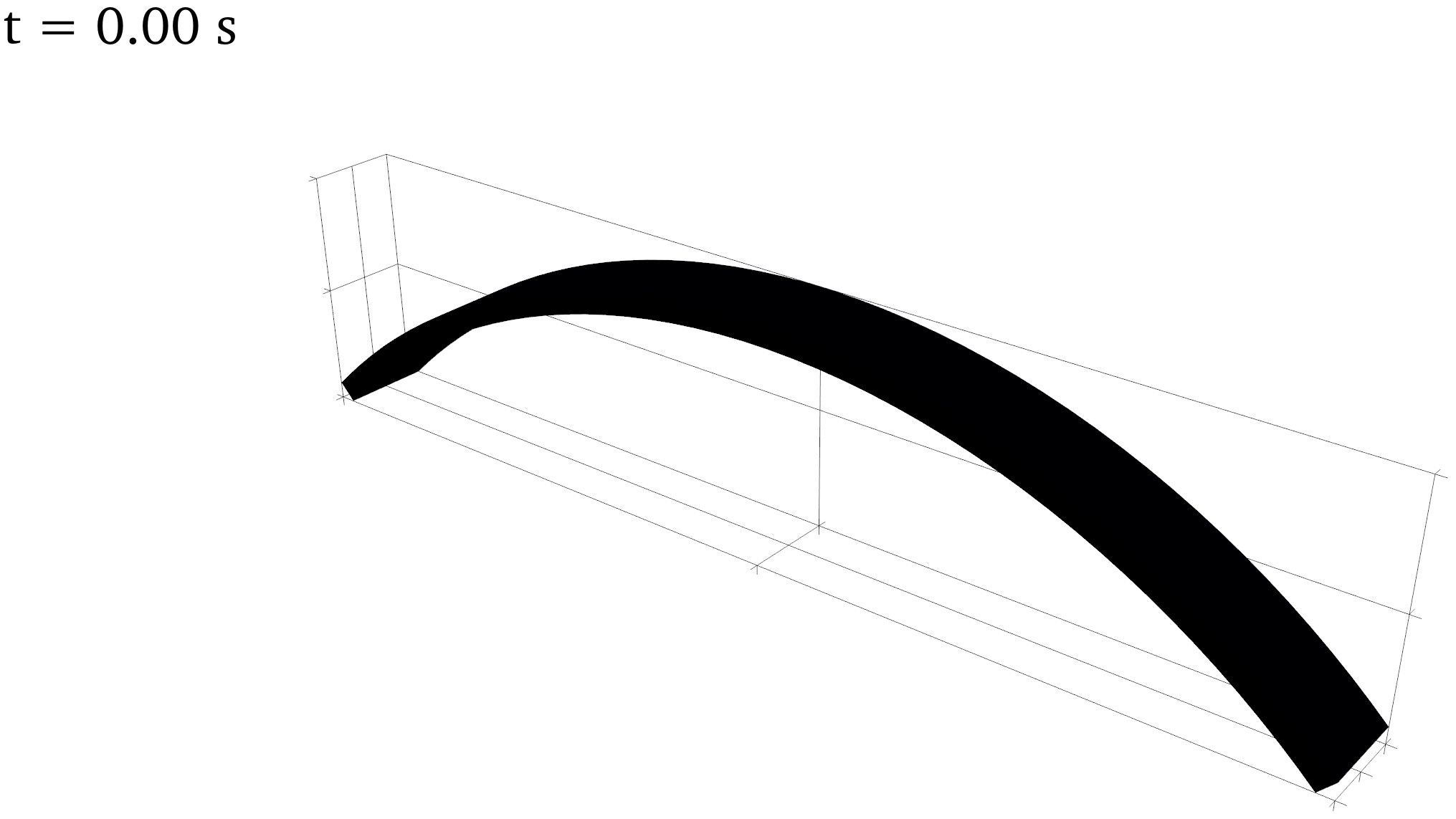}\hfill
\includegraphics[width=0.44\textwidth]{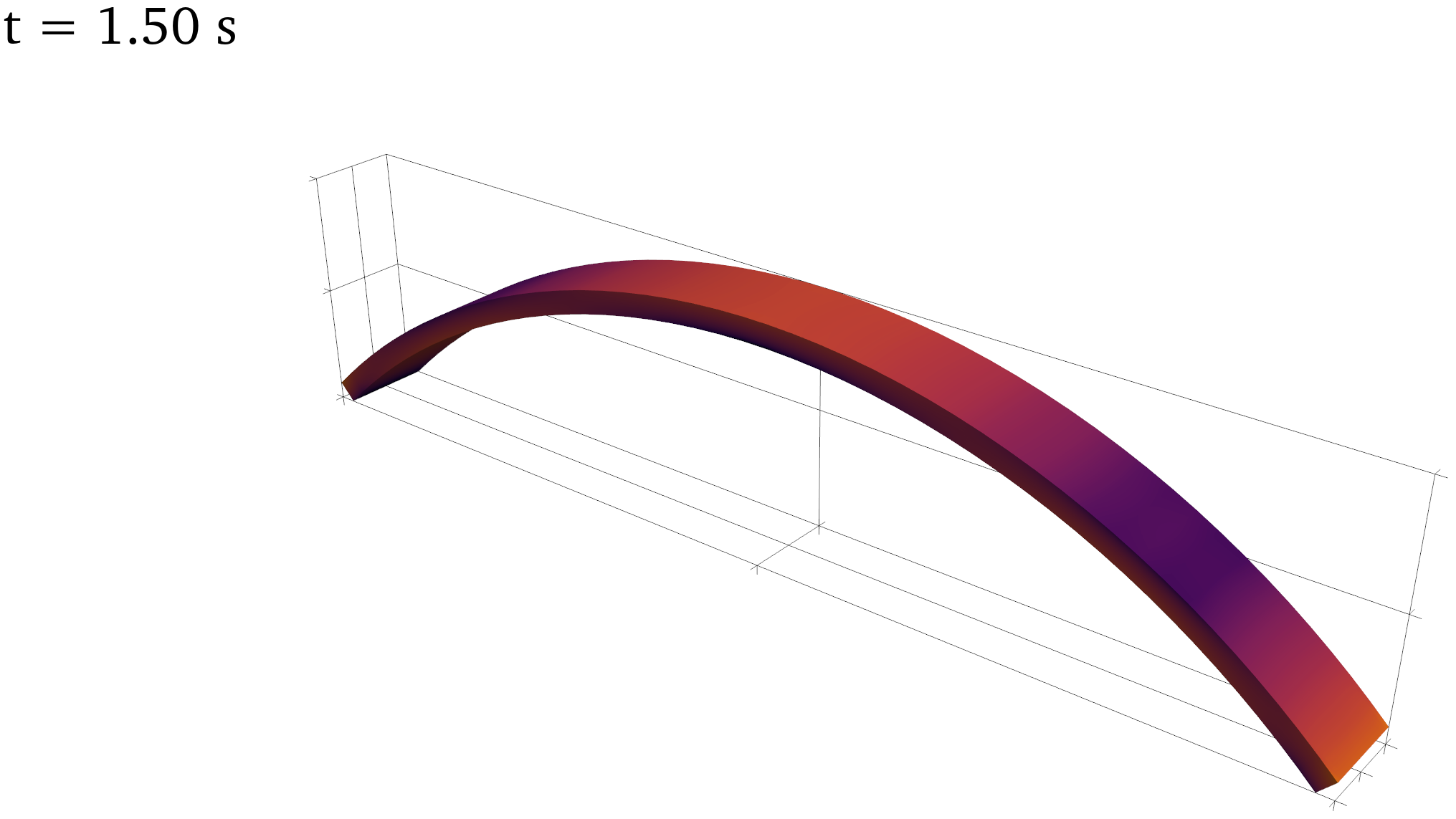}\\[0.2em]
\includegraphics[width=0.44\textwidth]{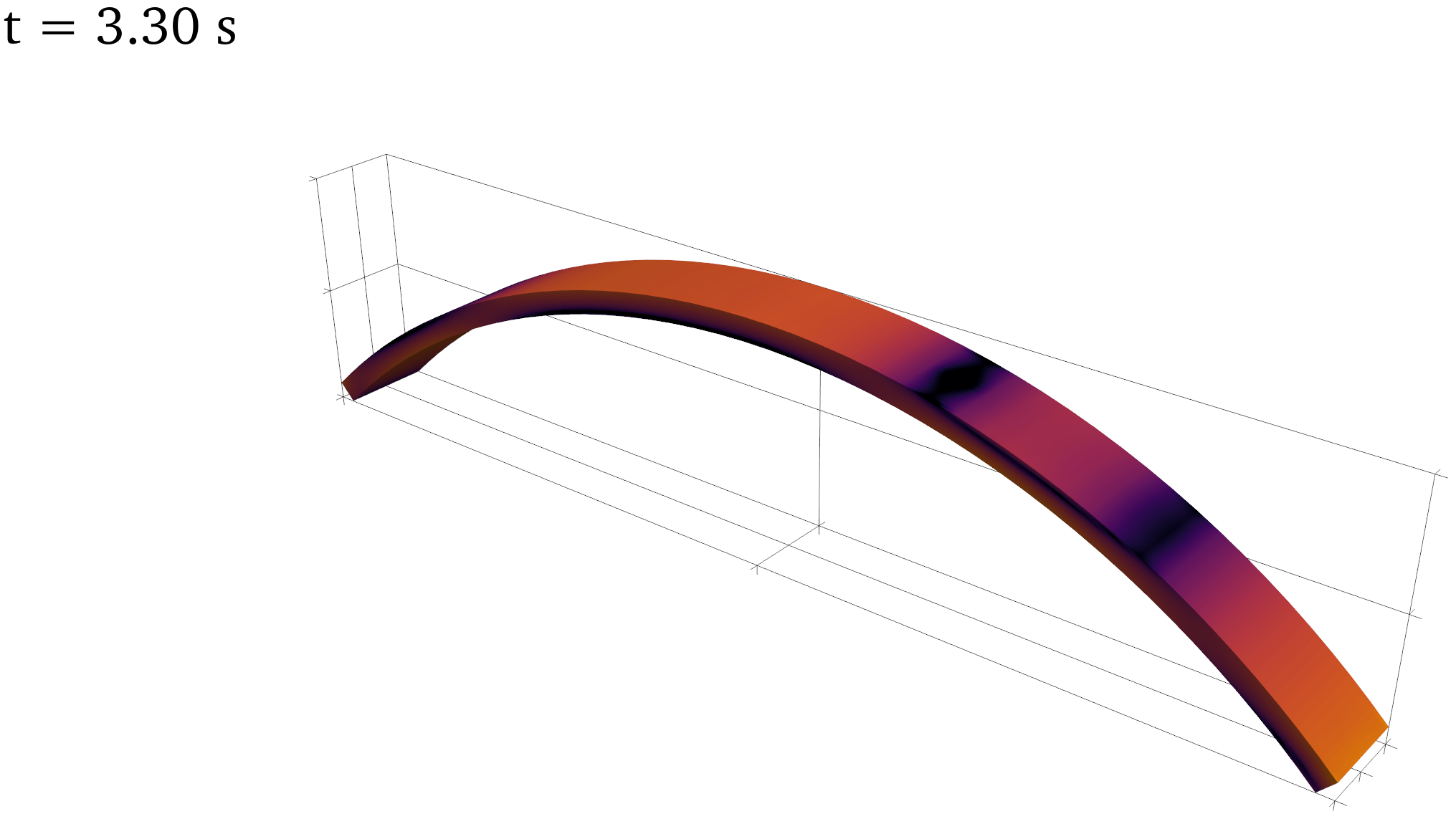}\hfill
\includegraphics[width=0.44\textwidth]{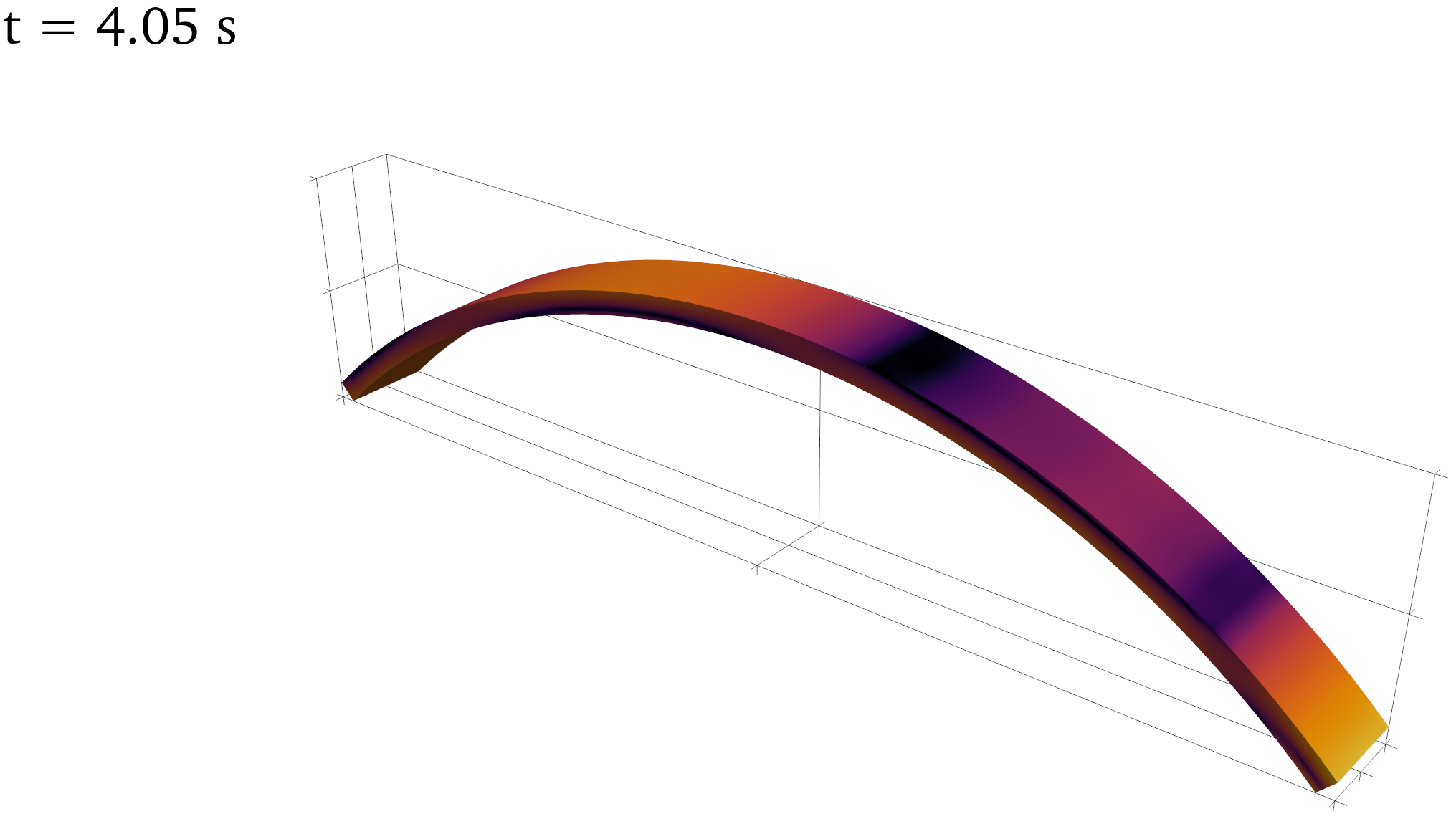}
\end{figure}
\begin{figure}[H]
\centering
\includegraphics[width=0.44\textwidth]{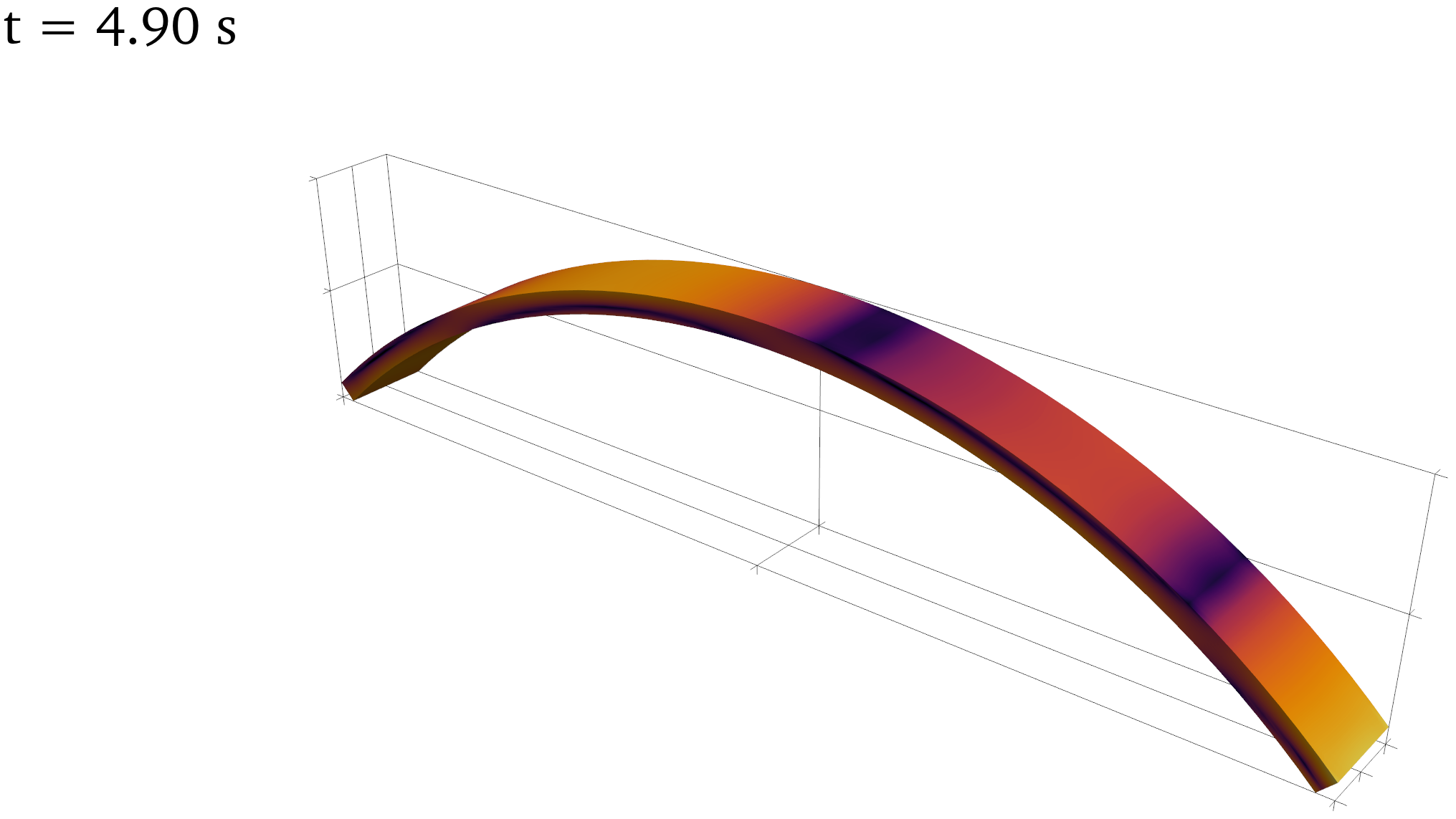}\hfill
\includegraphics[width=0.44\textwidth]{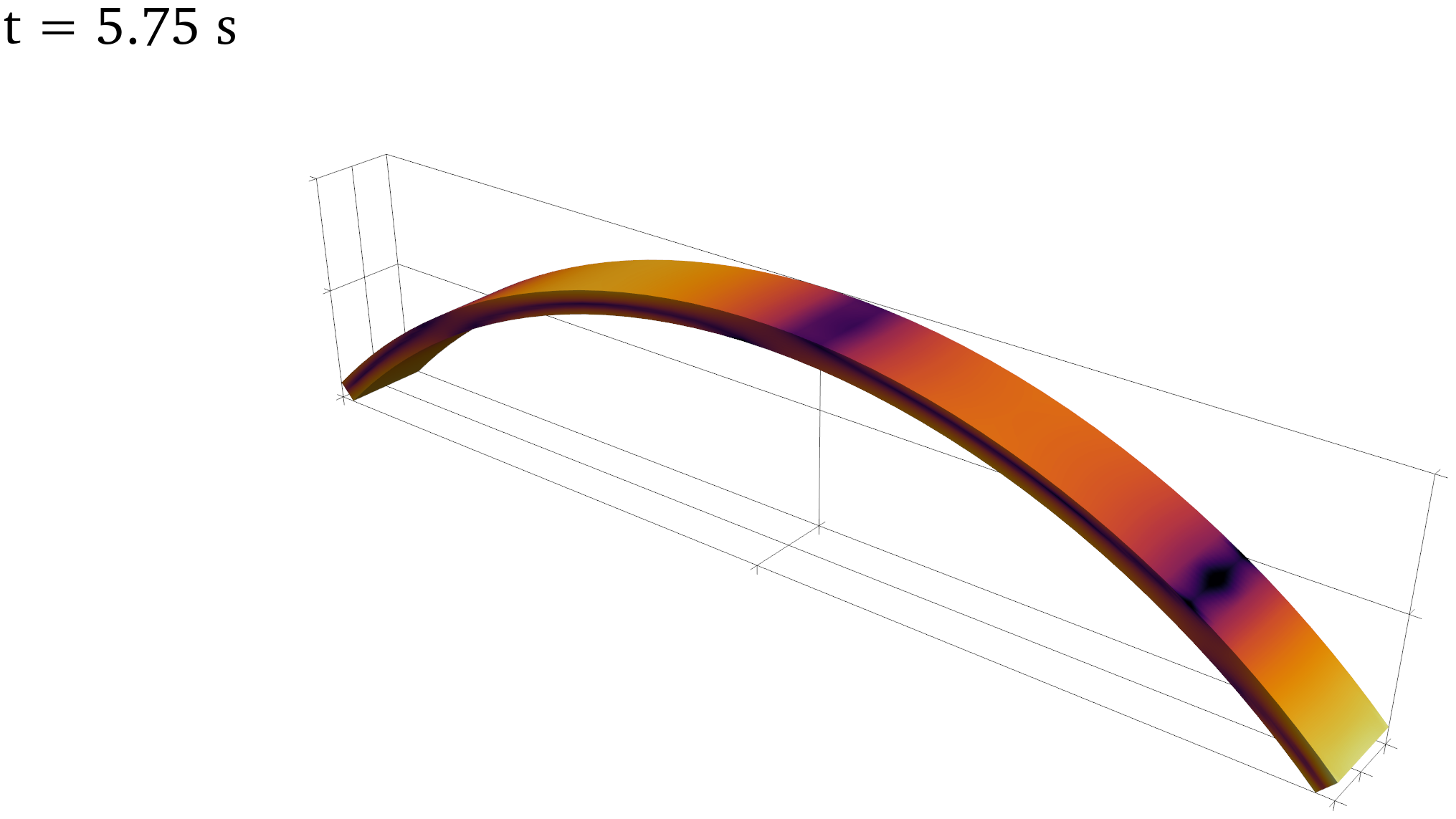}\\[0.2em]
\includegraphics[width=0.44\textwidth]{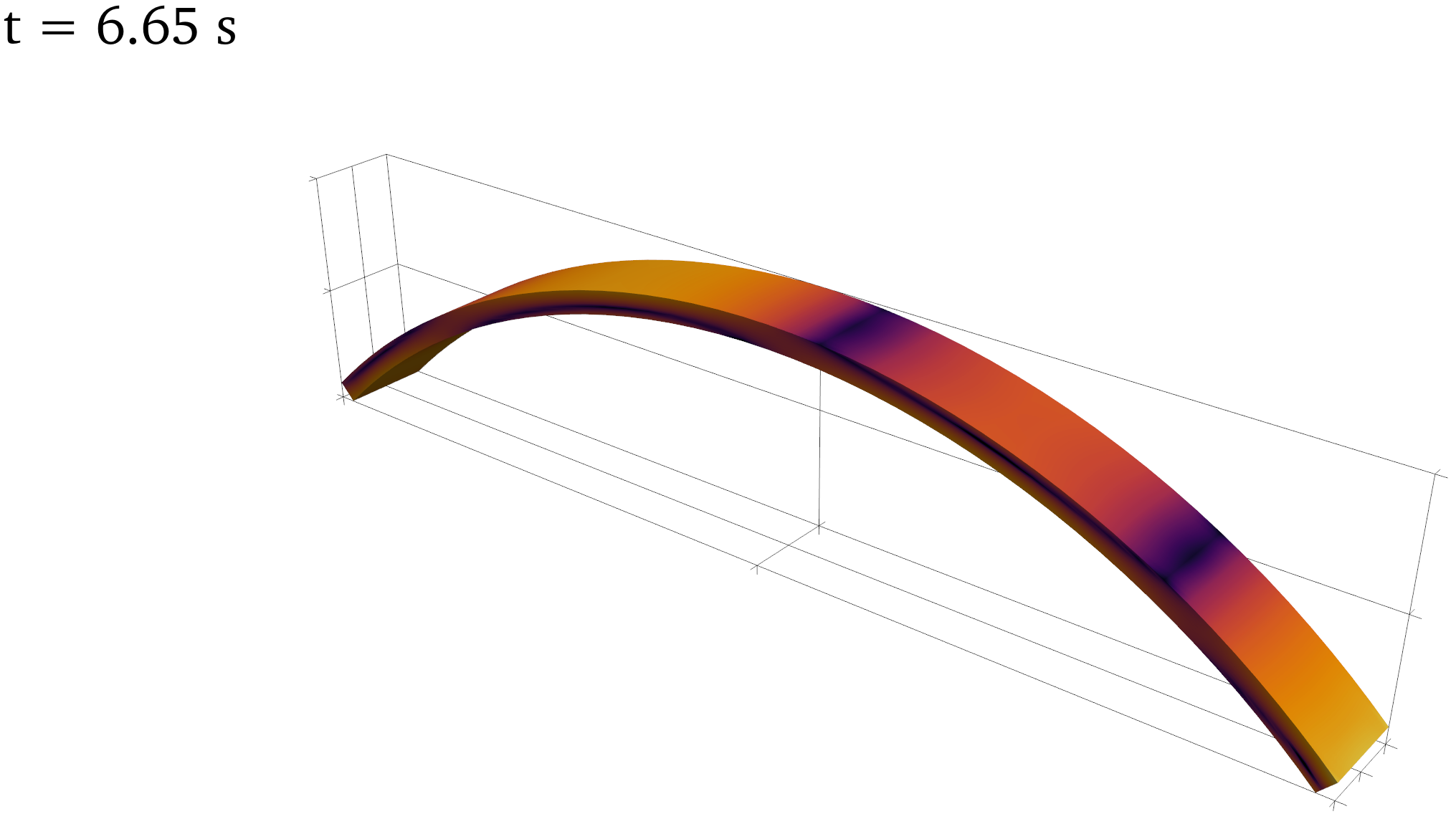}\hfill
\includegraphics[width=0.44\textwidth]{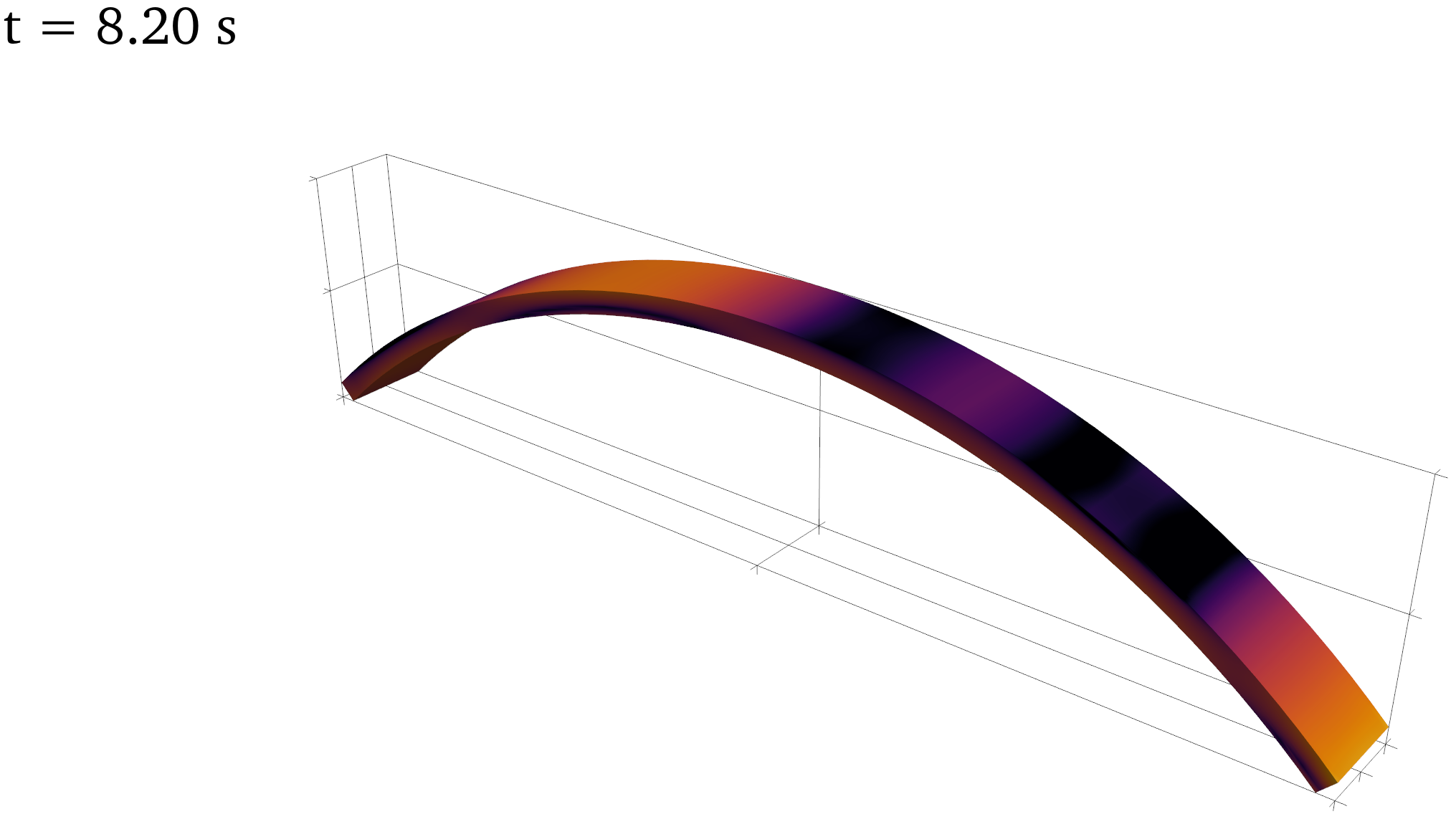}
\end{figure}
\begin{figure}[H]
\centering
\includegraphics[width=0.44\textwidth]{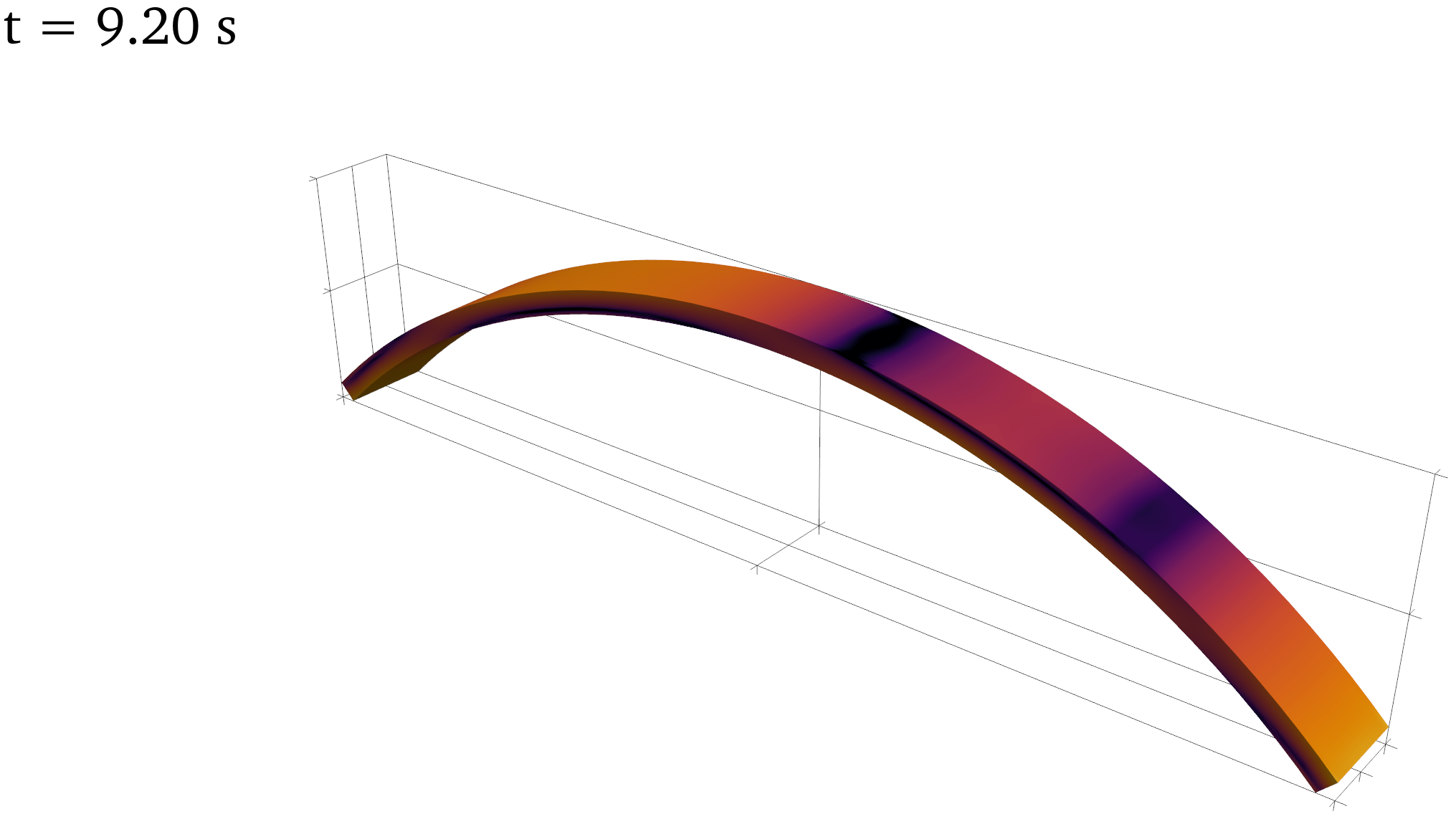}\hfill
\includegraphics[width=0.44\textwidth]{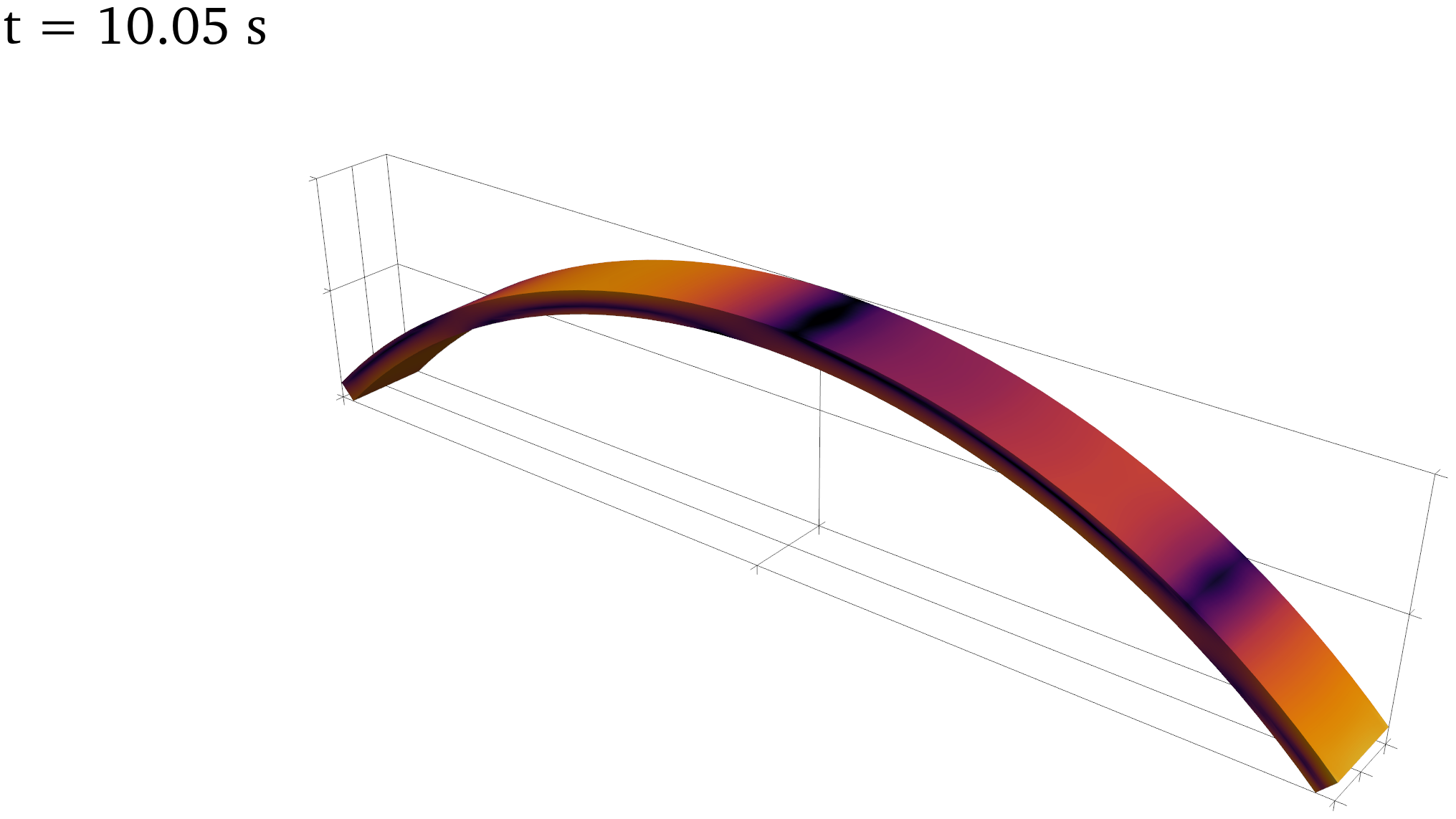}\\[0.2em]
\includegraphics[width=0.44\textwidth]{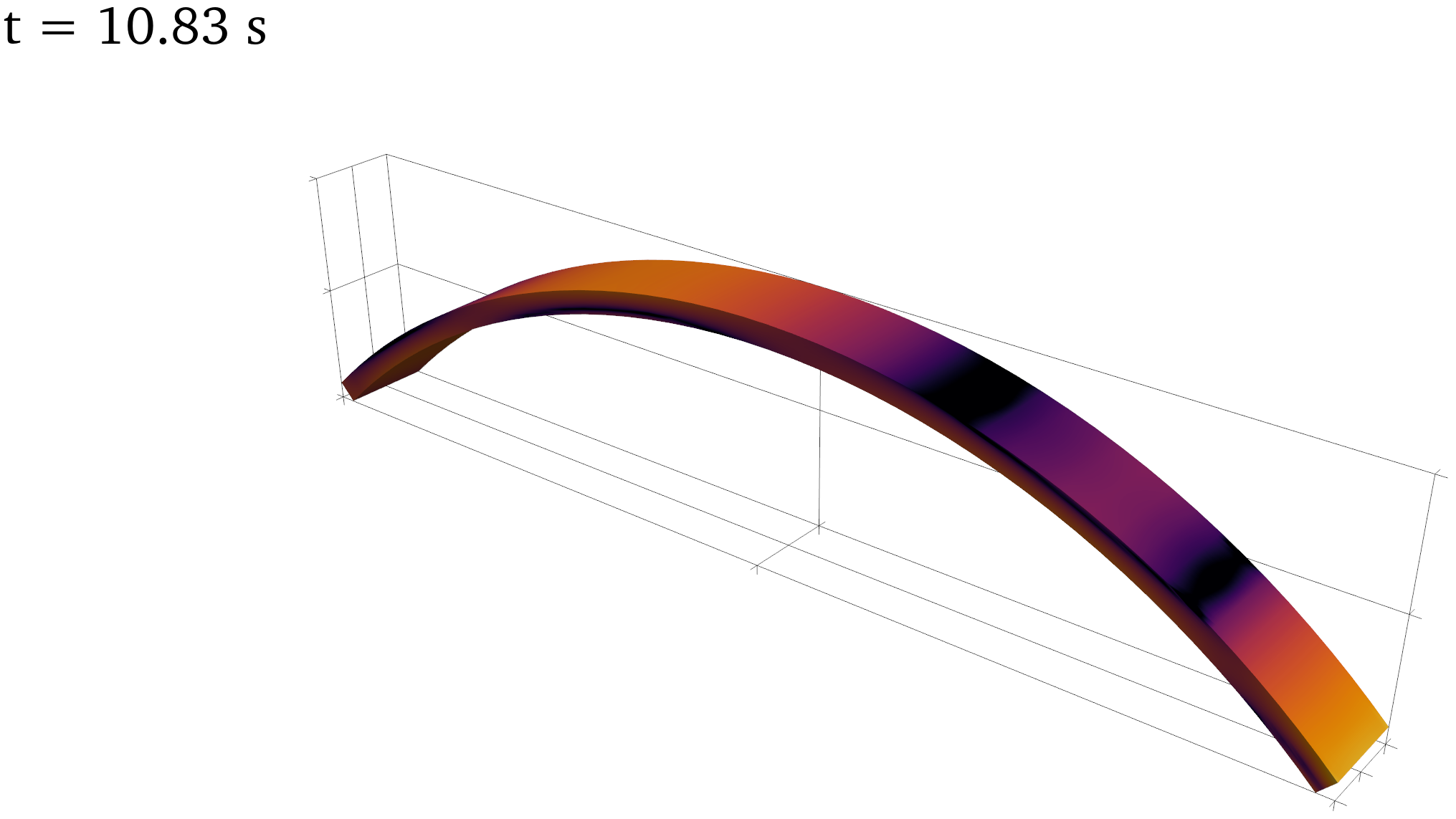}\hfill
\includegraphics[width=0.44\textwidth]{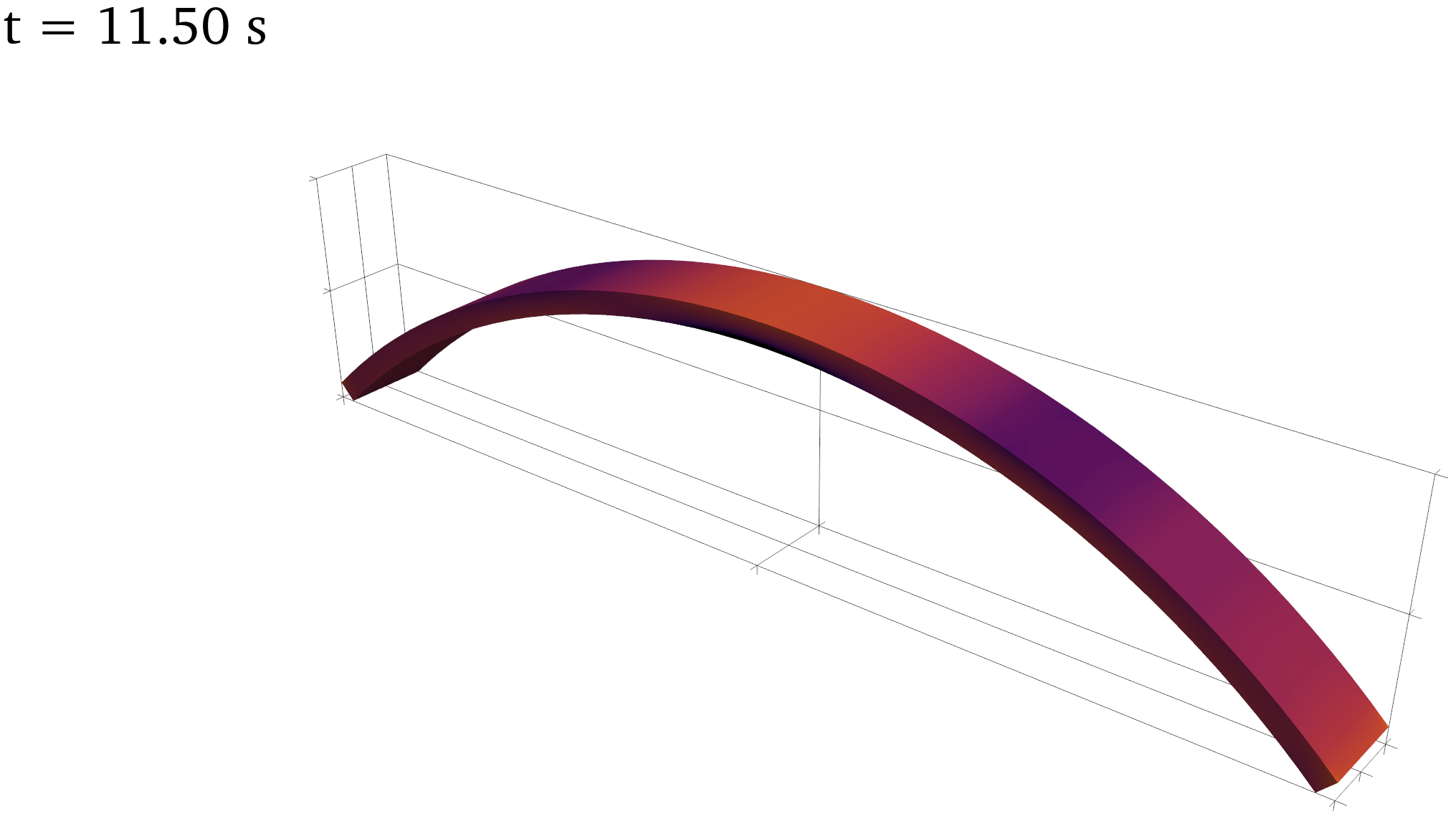}\\[0.3em]
\includegraphics[width=0.55\textwidth]{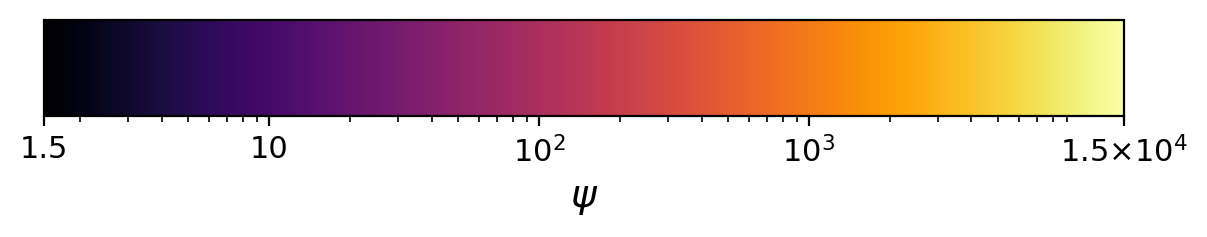}
\caption{Saint-Venant--Kirchhoff strain-energy density~$\psi$ with logarithmic scale on the
arch bridge at the given times.}
\label{fig:bridge-psi}
\end{figure}

\section*{Data availability}
\label{sec:data}

The data, implementation, and scripts that generate the figures and tables
of this study are available at
\url{https://github.com/Sreeram-Shankar/MIISO-Paper}.

\section*{CRediT authorship contribution statement}

\textbf{Sreeram Shankar:} Conceptualization, Methodology, Software,
Validation, Formal analysis, Investigation, Visualization, Writing -- original draft,
Writing -- review \& editing.

\section*{Funding}

This research received no specific grant from any funding agency.

\section*{Declaration of competing interest}

The author declares that there are no known competing financial interests
or personal relationships that could have appeared to influence the work
reported in this paper.

\section*{Acknowledgments}

The author thanks Frimpong Baidoo for extensive discussions and for an
introduction to isogeometric analysis and related spectral issues,
Thomas~J.~R.~Hughes and his research group for valuable discussions and
insight, and for the opportunity to work within the group, and Jesse
Chan for continued mentorship and advice throughout this work and
beyond.
All are affiliated with the Oden Institute for Computational
Engineering and Sciences at The University of Texas at Austin.
The author acknowledges the Texas Advanced Computing Center (TACC) at The
University of Texas at Austin for providing HPC resources on Lonestar6
under allocation ASC26062 that contributed to the research results
reported in this paper. URL: \url{http://www.tacc.utexas.edu}.

\section*{Declaration of generative AI and AI-assisted technologies in the writing process}

During the preparation of this work the author used Claude (Anthropic) and Cursor in order to
polish the writing, assist with debugging and prototyping of the numerical
code, and create visualizations and tables for the experiments.
After using these tools, the author reviewed and edited the content as needed
and takes full responsibility for the content of the published article.

\end{document}